\documentclass{article}
\usepackage[utf8]{inputenc}
\usepackage{amsfonts}
\usepackage{amsmath}
\usepackage{graphicx}
\usepackage{amsthm}
\usepackage{hyperref}
\usepackage{amssymb}
\usepackage{tikz}
\usetikzlibrary{mindmap}
\usepackage{natbib}
\usepackage{float}
\usepackage{color}
\usepackage{verbatim}

\newtheorem{thm}{Theorem}[subsection]
\newtheorem{lem}{Lemma}[subsection]
\newtheorem{defn}{Definition}[subsection]
\newtheorem{cor}{Corollary}[subsection]
\newtheorem{rem}{Remark}[subsection]
\newtheorem{prop}{Proposition}[subsection]

\title{Finite sample rates for logistic regression with small noise or few samples}
\author{Felix Kuchelmeister, Sara van de Geer}

\DeclareMathOperator*{\argmin}{arg\,min}

\DeclareMathOperator*{\hyp}{hyp}

\begin{document}

\maketitle

\paragraph{Abstract}
The logistic regression estimator is known to inflate the magnitude of its coefficients if the sample size $n$ is small, the dimension $p$ is (moderately) large or the signal-to-noise ratio $1/\sigma$ is large (probabilities of observing a label are close to 0 or 1).
With this in mind, we study the logistic regression estimator with $p\ll n/\log n$, assuming Gaussian covariates and labels generated by the Gaussian link function, with a mild optimization constraint on the estimator's length to ensure existence.
We provide finite sample guarantees for its direction, which serves as a classifier, and its Euclidean norm, which is an estimator for the signal-to-noise ratio.
\\
We distinguish between two regimes.
In the low-noise/small-sample regime ($\sigma\lesssim (p\log n)/n$), we show that the estimator's direction (and consequentially the classification error) achieve the rate $(p\log n)/n$ - up to the log term as if the problem was noiseless.
In this case, the norm of the estimator is at least of order $n/(p\log n)$.
If instead $(p\log n)/n\lesssim \sigma\lesssim 1$, the estimator's direction achieves the rate $\sqrt{\sigma p\log n/n}$, whereas its norm converges to the true norm at the rate $\sqrt{p\log n/(n\sigma^3)}$.
As a corollary, the data are not linearly separable with high probability in this regime.
In either regime, logistic regression provides a competitive classifier.

\newpage
\tableofcontents

\newpage
\section{Introduction}
We study a binary classification problem, where the covariates $x_i$ are standard Gaussian and the labels $y_i$ stem from a probit model.
More precisely, let $x_1,\ldots,x_n$ be independent copies of a standard Gaussian random variable taking values in $\mathbb{R}^p$.
The results readily generalize to arbitrary covariance matrices $\Sigma$.
Fix an unknown $\beta^*\in S^{p-1}:=\{\beta\in\mathbb{R}^p:\|\beta\|_2=1\}$, and for $i\in\{1,\ldots,n\}$ define $y_i:=sign(x_i^T\beta^*+\sigma\epsilon_i)$, where $\epsilon_i$ are independent standard Gaussians, and $\sigma>0$ is a fixed, unknown parameter - the \textit{noise-to-signal} ratio.
In the statistics literature, the domain of the labels $y_i$ is often $\{0,1\}$.
Moreover, it is common to fix $\sigma=1$ and allow the length of $\beta^*$ to be unknown.
This leads to the same model, taking $\|\beta^*\|_2:=1/\sigma$.
If both $1/\sigma$ and $\|\beta^*\|_2$ were unknown, the model would not be identifiable.
Those two quantities should be understood as a single parameter, the signal-to-noise ratio, in our case $1/\sigma$.

To make the parameter $\sigma$ more tangible, we mention that it determines the probability of observing a `wrong' label, i.e. $sign(x^T\beta^*)\neq y$.
For example, if $\sigma=1$, then $\mathbb{P}[y\neq sign(x^T\beta^*)]=1/4$.
Moreover, for small $\sigma$, the probability of observing a wrong label behaves as $\sigma$ up to multiplicative constants.

The classical choice to estimate $\beta^*/\sigma$ would be maximum-likelihood estimation, the \textit{probit} model, which is a \textit{generalized linear model} \citep{mccullagh1989generalized}.
The logistic regression estimator has very nice mathematical properties, which pair well with the probit model.
For this reason, we will use it to estimate $\beta^*$ and $\sigma$.
The logistic regression estimator $\hat{\gamma}$ is a solution to the following objective:
\begin{equation}\label{eq_obj_classical}
\argmin_{\gamma\in \mathbb{R}^p}\sum_{i=1}^n\log(1+\exp(-y_ix_i^T\gamma)).
\end{equation}
The estimator $\hat{\gamma}$ can be decomposed into its length $\hat{\tau}:=\|\hat{\gamma}\|_2$, and its direction $\hat{\beta}:=\hat{\gamma}/\|\hat{\gamma}\|_2$.
The label predicted by $\hat{\gamma}$ at a point $x\in\mathbb{R}^p$ is $sign(x^T\hat{\gamma})$, which coincides with $sign(x^T\hat{\beta})$.
In other words, the classification error is only dependent on the orientation $\hat{\beta}$.
On the other hand, $\hat{\tau}$ serves as a natural estimate of the signal-to-noise ratio $1/\sigma$.

Logistic regression is perhaps \textit{the} canonical classification method in statistics, it has been studied for several decades.
Nonetheless, it has not ceased to challenge and surprise theoreticians and practitioners alike.
In the following, we look into a few problems that this setup poses.

\subsection{Problems of logistic regression}

\subsubsection{Linear separation}
The first problem which is encountered in theory and practice is linear separation.
We say that the data are \textit{linearly separable}, if there exists a $\gamma\in\mathbb{R}^p$, such that for all $i\in\{1,\ldots,n\}$, $y_ix_i^T\gamma>0$.
In this case, \eqref{eq_obj_classical} has no solution in $\mathbb{R}^p$.
To see this, take any $\gamma'$ which separates the data.
As we multiply $\gamma'$ with a positive scalar tending to infinity, the loss function in \eqref{eq_obj_classical} vanishes.

To develop a rigorous non-asymptotic analysis of logistic regression, one must specify how linear separation is dealt with.
If one is solely interested in estimating the direction $\beta^*$ and not the signal-to-noise ratio $1/\sigma$, one can allow for infinite values $\|\hat{\gamma}\|_2\rightarrow\infty$.
In this case, given linear separation, the set of solutions to \eqref{eq_obj_classical} are all estimators, which separate the data.
Alternatively, one can modify the objective \eqref{eq_obj_classical}.
For example, \cite{agresti2015foundations} suggests Bayesian methods or penalized likelihood functions.
We constrain the estimator to a Euclidean ball of radius $M>0$, $B_M:=\{\gamma\in\mathbb{R}^p:\|\gamma\|_2\leq M\}$:
\begin{equation}\label{eq_opt}
\argmin_{\gamma\in B_M}\sum_{i=1}^n\log(1+\exp(-y_ix_i^T\gamma)).
\end{equation}
In our analysis, we allow the radius $M$ to be arbitrarily large - but finite.
Consequentially, we could choose $M$ so large, that it coincides with the unconstrained objective \eqref{eq_obj_classical}, but for the case where the data are linearly separable.
This is the least intrusive modification of the logistic regression, which still enforces that it takes finite values with probability 1.
By Lagrangian duality, there exists a $\lambda\geq 0$, such that \eqref{eq_obj_classical} with ridge penalty $+\lambda\|\gamma\|_2$ coincides with \eqref{eq_opt} (this is also known as $\ell^2$- or Tikhonov-regularization).
Whenever $\lambda>0$, the coefficients of the ridge estimator $\hat{\gamma}$ are biased towards zero, which is not the case for \eqref{eq_opt}.
On the other hand, if $\lambda=0$, we have \eqref{eq_obj_classical}, and the problem of linear separation appears again.
Hence, we prefer \eqref{eq_opt}, the constrained form.

In some settings, linear separation occurs with vanishingly small probability and hence can be neglected.
This is the case in the classical asymptotic literature \citep{van2000asymptotic}.
In the high-dimensional asymptotic setting, where both $p$ and $n$ tend to infinity, \cite{candes2020phase} have shown that the data are linearly separable, if and only if the ratio $p/n$ in the limit is lower than a constant, depending on the signal-to-noise ratio $1/\sigma$.
They study logistic noise.
This result has been generalized in \cite{montanari2019generalization}, which for example includes the Gaussian noise model (probit).
For finite $n$ and $p$, \cite{cover1965geometrical} showed that if there is no signal ($y$ and $x$ are independent, i.e. $\sigma\rightarrow\infty$), that the probability of linear separation is exactly $2^{n-1}\sum_{k=0}^{p-1}{n-1\choose k}$.
However, for finite $n$ the literature on linear separability in generalized linear models is still lacking.

Although we are primarily concerned with providing guarantees for the estimated signal strength $\|\hat{\gamma}\|_2$ and the estimated direction $\hat{\gamma}/\|\gamma\|_2$, our results allow conclusions for the linear separation problem.
From Theorem \ref{thm_cor_large}, we derive an upper bound for the probability of linear separation.
Namely, if $p\log (n)/n\lesssim \sigma\lesssim 1$, the data are not linearly separable with large probability (see Proposition \ref{prop_cor_separation}).

\subsubsection{Overestimating coefficient magnitude}
Logistic regression appears to overestimate the signal-to-noise ratio $1/\sigma$ in some settings.
In other words, the magnitude of the coefficients $|\hat{\gamma}_j|$ is too large, there is a bias away from zero.
For example, this was observed in \cite{nemes2009bias} in a simulation where this upward bias was especially present for smaller sample sizes.
A similar observation was made in a simulation in \cite{sur2019modern}.
They performed a simulation with $n=4000$, $p=800$ and $\sigma=1/5$ (using logistic noise).
Moreover, they proved that under some regularity conditions in a high-dimensional asymptotic setting (letting $n,p\rightarrow\infty$ at a constant ratio), the estimator $\hat{\gamma}$ does not converge to $\beta^*/\sigma$, but instead fluctuates around $\alpha_*\beta^*/\sigma$, where $\alpha_*>1$.
In other words, the magnitude of the coefficients is overestimated.

In Theorem \ref{thm_cor_small}, we see that if $\sigma\lesssim (p\log n)/n$, the estimate of the signal-to-noise ratio $\hat{\tau}$ is at least of order $n/(p\log n)$, although it may be larger.
On the other hand, we see in Theorem \ref{thm_cor_large} that if $\sigma\gtrsim (p\log n)/n$, the deviation of $\hat{\tau}$ from its target is of order no larger than $\sqrt{(p\log n)/(n\sigma^3)}$.
Our findings suggest that in the regime $\sigma\gtrsim (p\log n)/n$ the magnitude of the coefficients can be reliably estimated.

\subsubsection{Inadequacy of classical asymptotic approximations}
It is known that as $n\rightarrow\infty$, for fixed $p$ and $\sigma$, the logistic regression estimator converges in distribution to a normally distributed random variable.
This is exploited in statistical software, such as the function \textit{glm} in the R-package `stats', to provide approximate p-values.

However, classical asymptotic approximations can be very poor in logistic regression.
A famous early reference for this is \cite{hauck1977wald}, who showed that the Wald statistic is unreliable if the signal-to-noise ratio is large.
A more recent example is \cite{candes2016panning}, who observed in a simulation with $n=500$, $p=200$ and no signal ($y$ independent of $x$, in our language $\sigma\rightarrow\infty$), that the p-values from asymptotic maximum likelihood theory for the global null-hypothesis are left-skewed.
The same observation was made in \cite{sur2019modern} with $n=4000$, $p=800$.

An attempt to overcome these shortcomings is to use a different regime, such as high dimensional asymptotics \citep{sur2019modern,montanari2019generalization}.
While in this regime, the bias of the coefficients becomes visible, it only applies when the maximum likelihood estimator exists asymptotically almost surely, i.e. when the data are not linearly separable.
For a finite number of samples, however, there is no way around dealing with possible linear separability.

\subsection{Related results} \label{sub_related}
The literature for finite sample guarantees for unregularized logistic regression in the regime $p<n$ is still lacking, in particular for the case of small noise $\sigma$.
We consider the case $p\ll n/\log n$.
There are finite sample results for logistic regression in high dimensions $p>n$ (with Lasso penalty, e.g. \cite{van2008high}, and $\ell^0$-penalty, e.g. \cite{abramovich2018high}).

In the noiseless case, if the radius $M$ is large enough, the logistic regression estimator always separates the data.
Hence, it is also a minimizer of the 0/1-loss.
It is known that in the noiseless case, the optimal rate for estimating $\beta^*$ is $p/n$ \citep{long1995sample}, which is achieved by any such minimizer of the 0/1-loss \citep{long2003upper}.
We show that logistic regression obtains the same rate - up to the log term as if the problem were noiseless - even in the presence of a little noise, namely $\sigma\lesssim  (p\log n)/n$.
If instead $(p\log n)/n\lesssim \sigma\lesssim 1$, the classification error achieves the parametric rate $\sqrt{(\sigma p\log n)/n}$.

Some state-of-the-art finite-sample rates for logistic regression can be obtained in \cite{ostrovskii2021finite}.\footnote{The work \cite{ostrovskii2021finite} does not deal with linear separation, while assuming that the parameter space is a subset of $\mathbb{R}^p$. 
So, for their results to apply, one technically needs to assume that $\tau^*$ and $\hat{\tau}$ are bounded.
Otherwise, the estimator is not well-defined with positive probability.
However, the statistician then needs to know an upper bound on the length of the target $\|\gamma^*\|_2\sim 1/\sigma$.
We do not need to make such an assumption on $\tau^*$.
However, we impose that $\hat{\tau}\leq M$, where $M\geq n/(p\log(n)+t)$.
Although this is a technical contribution of our work, we suppose that such an argument could be added to works avoiding the discussion of linear separability, such as \cite{ostrovskii2021finite}.
}
They show the rate $\sqrt{(1+1/\sigma^{3})p/n}$ for $\|\hat{\gamma}-\gamma^*\|_2$.
In the regime $1\lesssim 1/\sigma\lesssim n/(p\log n)$, we obtain the same rate with a logarithmic factor (see Theorem \ref{thm_cor_large} and Lemma \ref{lem_proxVSeuclid}).
In the regime $\sigma \lesssim p\log(n)/n$, such a control without further constraints is generally not possible.
To see why, note that if $\sigma$ is small enough, the probability of linear separability is close to 1. 
Consequentially, for small enough $\sigma$, $\|\hat{\gamma}-\gamma^*\|_2=+\infty$ with large probability.

Our analysis exposes that this slow dependence $\sqrt{(1+1/\sigma^{3})p/n}$ on the signal-to-noise ratio $1/\sigma$ is only present when estimating the signal strength $\tau^*:=\|\gamma^*\|_2\sim1/\sigma$, and not when performing classification.
In fact, Theorems \ref{thm_cor_large} and \ref{thm_cor_small} show that the classifier of logistic regression $\hat{\beta}:=\hat{\gamma}/\|\hat{\gamma}\|_2$ converges faster to the target classifier $\beta^*=\gamma^*/\|\gamma^*\|_2$, if the signal-to-noise ratio is larger, with rates $\|\hat{\beta}-\beta^*\|_2\lesssim \sqrt{\sigma p\log(n)/n}$ if $p\log(n)/n\lesssim \sigma\lesssim 1$ and $\|\hat{\beta}-\beta^*\|_2\lesssim  p\log(n)/n$ if $\sigma \lesssim p\log(n)/n$.
To obtain these results, it was crucial in our analysis to treat the estimated signal-to-noise ratio $\hat{\tau}:=\|\hat{\gamma}\|_2$ and the estimated classifier $\hat{\beta}$ separately.

Shortly after the appearance of the first draft of this work, \cite{hsu2023sample} appeared.
They study finite sample classification with logistic noise.
Their parameter space is $S^{p-1}$, so they estimate only the direction $\beta^*$, not the signal-to-noise ratio $\tau^*\sim 1/\sigma$.
Although they study logistic noise, while we study Gaussian noise, the results should be comparable.
They also consider the case $\sigma\gtrsim 1$, which is out of the scope of this work.
Moreover, they provide upper bounds that match their lower bounds up to constant factors, but with estimators other than the logistic regression estimator.

They show that in order to achieve a rate $\varepsilon$ for $\|\beta^*-\hat{\beta}\|_2$, at least $n\gtrsim \sigma p/\varepsilon^2$ samples are needed, and if $\sigma\lesssim\varepsilon$ then $n\gtrsim p/\varepsilon$ samples are needed (ignoring a logarithmic factor in $\varepsilon$).
Up to logarithmic terms, this matches our upper bounds in Theorems \ref{thm_cor_large} and \ref{thm_cor_small}.

\newpage 

\subsection{Outline}
In this paper, we prove guarantees for the estimated signal-to-noise ratio $\hat{\tau}:=\|\hat{\gamma}\|_2$ and the estimated direction $\hat{\beta}:=\hat{\gamma}/\|\hat{\gamma}\|_2$, for the logistic regression estimator constrained to the Euclidean ball of radius $M>0$, see \eqref{eq_opt}.
We distinguish between two regimes.
The first regime covers the case $\sigma\gtrsim (p\log n)/n$, whereas the second regime covers  $\sigma\lesssim (p\log n)/n$, the ``large noise case" and the ``small noise case".
We prove two theorems, one for each case. 
 They are presented in Section \ref{sec_main} (Theorems \ref{thm_cor_large} and \ref{thm_cor_small}).
There, we provide sketches of the two proofs, along with a consequence for the probability of linear separability (Proposition \ref{prop_cor_separation}).
The proofs of the main theorems are given in Section \ref{sec_state}, where we prove two slightly more general statements (Theorems \ref{thm_large} and \ref{thm_small}).
Sections \ref{sec_geo}, \ref{sec_con_b}, \ref{sec_con_u}, as well as Appendix \ref{sec_app} all lead up to the proof of the main theorems, see Figure \ref{fig_structure}.

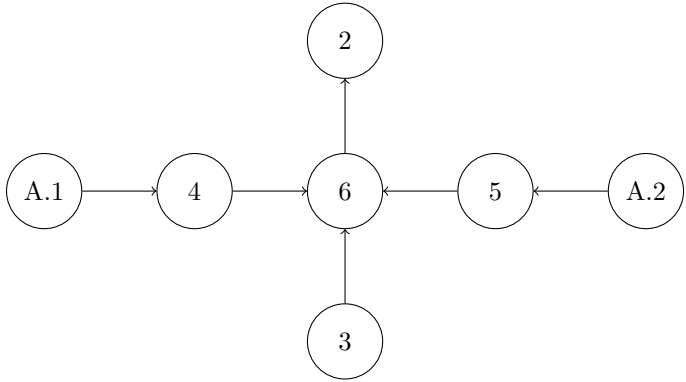
\begin{figure}[h]
\centering
\begin{tikzpicture}[node distance=2cm, every node/.style={circle, draw, minimum size=10mm}]
  \node (6) {6};
  \node[left of=6] (4) {4};
  \node[left of=4] (A1) {A.1};
  \node[right of=6] (5) {5};
  \node[right of=5] (A2) {A.2};
  \node[above of=6] (2) {2};
  \node[below of=6] (3) {3};

  \draw[->] (A1) -- (4);
  \draw[->] (4) -- (6);
  \draw[->] (3) -- (6);
  \draw[->] (5) -- (6);
  \draw[->] (A2) -- (5);
  \draw[->] (6) -- (2);
\end{tikzpicture}
\caption{This figure shows the structure of the paper.
The main results are given in Section \ref{sec_main} and are proved in \ref{sec_state}.
These proofs rely on Sections \ref{sec_geo}, \ref{sec_con_b} and \ref{sec_con_u}.
Appendices A.1 and A.2 contain supporting inequalities for Section \ref{sec_con_b} and \ref{sec_con_u}.
}\label{fig_structure}
\end{figure}

Section \ref{sec_geo} establishes the relationships between classification and Euclidean geometry,  the noise-to-signal ratio $\sigma$ and the length of the target $\tau^*:=\|\gamma^*\|_2$, and introduces the $*$-norm, which is used to quantify the distance $d_*(\hat{\gamma})$ between the estimator $\hat{\gamma}$ and its target $\gamma^*$ in the large noise case.

Throughout our analysis, we decompose the logistic loss as follows:
\[
\log(1+\exp(-yx^T\gamma))=\log(1+\exp(-|x^T\gamma|))+|x^T\gamma|1\{yx^T\gamma<0\}.
\]
Introducing some notation, we re-state this equality as:
\[
l(\gamma,x,y)=b(\gamma,x)+u(\gamma,x,y).
\]
We refer to $b$ as the `bounded term' and call $u$ the `unbounded term'.
Thanks to this decomposition, we can prove strong upper bounds on their empirical processes in Sections \ref{sec_con_b} and \ref{sec_con_u}, which in turn gives us a strong control of the empirical process of $l$.

\subsection{Notation}

\begin{itemize}
    \item $n$ is the number of observations and $p<n$ is the dimension of the problem. 
    \item $x\sim\mathcal{N}(0,I_p)$ are the observed covariates.
    \item $\epsilon\sim\mathcal{N}(0,1)$ is unobserved noise independent of $x$.
    \item $S^{p-1}:=\{\beta\in \mathbb{R}^p:\|\beta\|_2=1\}$ is the unit sphere. For $\delta>0$ and $\beta'\in S^{p-1}$, we use the notation $S(\delta,\beta'):=\{\beta\in S^{p-1}:\|\beta-\beta'\|_2\leq \delta\}$ for a spherical cap.
    \item For $r>0$, $B_r:=\{\gamma\in\mathbb{R}^p:\|\gamma\|_2\leq r\}$ is the closed Euclidean ball of radius $r$.
    We use the notation $B_r^\circ$ for the open ball of radius $r>0$.
    \item $\beta^*\in S^{p-1}$ is the unknown orientation we want to learn.
    \item $\sigma>0$ is the noise-to-signal ratio.
    \item $y=sign(x^T\beta^*+\sigma\epsilon)$ are the observed labels.
    \item $M>0$ is a known hyperparameter, chosen by the statistician. It determines the upper bound on the Euclidean norm of the estimator, defined below.
    \item We use $P$ to denote the expectation and $P_n$ to denote the average. More formally, $P_n$ is the expectation with respect to the empirical measure, that is $n^{-1}\sum_{i=1}^n\delta_{(x_i,y_i)}$.
    \item The logistic loss function is the following mapping:
    \[
    l:\mathbb{R}^p\times\mathbb{R}^p\times\{-1,1\}\rightarrow(0,\infty),\quad
    (\gamma,x,y)\mapsto \log(1+\exp(-yx^T\gamma)).
    \]
    \item Our estimator is:
    \[
    \hat{\gamma}:=\argmin_{\gamma\in B_M}\sum_{i=1}^n\log(1+\exp(-y_ix_i^T\gamma))
    =\argmin_{\gamma\in B_M}P_nl(\gamma,x,y).
    \]
    \item The target of the estimator is:
    \[
    \gamma^*:=\argmin_{\gamma\in \mathbb{R}^p}Pl(\gamma,x,y).
    \]
    \item We will frequently decouple a vector $\gamma$ into its orientation $\beta:=\gamma/\|\gamma\|_2$ and its length $\tau:=\|\gamma\|_2$. 
    This may seem like an abuse of notation, as $\beta^*$ would both denote the unknown orientation generating the data, and $\gamma^*/\|\gamma^*\|_2$. However, we will see in Lemma \ref{lem_opt_direction} that those two coincide.
    As a side remark, we will also see that $1/\sigma$ and $\tau^*:=\|\gamma^*\|_2$ generally do not coincide (Lemma \ref{lem_tau_sigma_derivatives}).
    \item We write the logistic loss as the function $l:\mathbb{R}^p\times\mathbb{R}^p\times\{-1,1\}\rightarrow\mathbb{R}$, mapping $(\gamma,x,y)\mapsto\log(1+\exp(-yx^T\gamma))$.
    Occasionally, we omit the arguments $x,y$, writing simply $l(\gamma)$.
    \item The loss function $l$ can be decomposed into two parts. One is bounded $b$, the other unbounded $u$:
    \[
    \log(1+\exp(-yx^T\gamma))=\log(1+\exp(-|x^T\gamma|))+|x^T\gamma|1\{yx^T\gamma<0\}
    \]
    \[
    =:b(\gamma,x)+u(\gamma,x,y).
    \]
    If clear from context, we omit the arguments $x$ and $y$, writing simply $b(\gamma)$ and $u(\gamma)$. 
    We will frequently use the (un)bounded term, subtracting the term evaluated at $\gamma^*$:
    \[
    \tilde{b}(\gamma):=\frac{b(\gamma)-b(\gamma^*)}{\log 2}.
    \]
    When studying the large noise case, we will use the notation, $\tilde{u}(\gamma):=u(\gamma)-u(\gamma^*)$ whereas in the small noise case, we use $\tilde{u}(\beta):=u(\beta)-u(\beta^*)$.
    Moreover, we define the following set of functions:
    \[
    \mathcal{B}:=\left\{\tilde{b}(\gamma):\gamma\in\mathbb{R}^p\right\}.
    \]
    \item We use $e_j\in\mathbb{R}^p$ to denote the $j$-th coordinate unit vector. For example, $e_1=(1,0,\ldots,0)$.
    \item We define a weighted Euclidean norm:
    \[
    \|\cdot\|_*:\mathbb{R}^{p+1}\rightarrow [0,\infty),\quad (\tau,\beta)\mapsto \sqrt{\frac{|\tau|^2}{\tau^{*3}}+\tau^*\left\|\beta\right\|_2^2}.
    \]
    \item We define a metric, to quantify the distance between a vector $\gamma$ and $\gamma^*$ in terms of the norm $\|\cdot\|_*$:
    \[
    d_*:\mathbb{R}^p\setminus\{0\}\rightarrow[0,\infty),\quad \gamma\mapsto \left\|\left(\|\gamma\|_2,\frac{\gamma}{\|\gamma\|_2}\right)-(\tau^*,\beta^*)\right\|_*.
    \]
    \item We define the ball of radius $\delta$ in the $\|\cdot \|_*$-norm, centered at $(\tau^*,\beta^*)$ as: \[B^*_\delta:=\{\gamma\in\mathbb{R}^p\setminus\{0\}:d_*(\gamma)\leq \delta\}.\]
    \item For the `peeling' argument, we will define the `peels' $B^*_{r,R}:=B_R^*\setminus B_r^*$ for $r, R>0$.
    \item In some occasions, to compare two quantities $f,g$, we use the notation $f\lesssim g$.
    By this, we mean that there exists an absolute constant $C>0$, not depending on any model parameters such as $n,p,\sigma,\beta^*$, etc, such that $f\leq C g$.
    If $f\lesssim g$ and $g\lesssim f$, we write $f\sim g$.
\end{itemize}

\newpage
\section{Main results}\label{sec_main}
Here, we state the main results of the paper: Theorems \ref{thm_cor_large} and \ref{thm_cor_small}.
We start with the regime where $(p\log n)/n\lesssim \sigma\lesssim 1$, and continue with the regime $\sigma\lesssim (p\log n)/n$.
The proofs of both theorems are given in Section \ref{sec_state}, although we provide a sketch for each in this section.
Finally, we show that the data are not linearly separable with large probability in the first regime, which follows from the behavior of the logistic regression estimator.

\subsection{Regime 1 (``large noise")}
\subsubsection{Statement of the result}
We state the main result for regime 1, where the noise is `large'. 
The proof is given in Section \ref{sec_state}, where we prove the slightly stronger Theorem \ref{thm_large}.
The exact universal constants can be found in the proof.
Recall that $\hat{\beta}:=\hat{\gamma}/\|\hat{\gamma}\|_2$ and $\hat{\tau}:=\|\hat{\gamma}\|_2$, where $\hat{\gamma}$ is the logistic regression estimator constrained to a Euclidean ball of radius $M$, see \eqref{eq_opt}.

\begin{thm}\label{thm_cor_large}
For any $t>0$, if
\[
\frac{p\log n+t}{n}\lesssim \sigma\leq \frac{1}{\sqrt{6}},\quad M \geq \frac{n}{p\log n + t},
\]
then, with probability at least $1-4\exp(-t)$, 
\[
\|\hat{\beta}-\beta^*\|_2\lesssim \sqrt{\sigma\frac{p\log n+t}{n}},\quad
|\hat{\tau}-\tau^*|\lesssim \sqrt{\frac{1}{\sigma^3}\frac{p\log n+t}{n}}.
\]
\end{thm}

In other words, as long as $\sigma$ is not too small (or $n$ is very large) and $M$ is large enough, up to a factor involving $\sigma$, we get the rate $\sqrt{(p\log n)/n}$ both for the estimates of the orientation as well as the signal-to-noise ratio (for the relationship between $\tau^*$ and $1/\sigma$, see Lemmas \ref{lem_tau_sigma_derivatives} and \ref{lem_sigmaVStau}).

The result translates into a bound for $\|\hat{\gamma}-\gamma^*\|_2$ of order $\sqrt{\tau^{*3}(p\log n)/n}$, using Lemma \ref{lem_proxVSeuclid}.
This matches the rate in \cite{ostrovskii2021finite}.
Theorem \ref{thm_cor_large} reveals a more subtle insight into how the signal-to-noise ratio (SNR) $1/\sigma$ affects the estimated orientation $\hat{\beta}$ and the estimated SNR $\hat{\tau}$.
As the noise level is smaller, it is easier to classify, but harder to make statements about the SNR, see also our discussion in Section \ref{sub_star_norm}.

The assumption $\sigma\leq 1/\sqrt{6}$ is made due to technical reasons.
In other words, the probability of observing a wrong label ($y\neq sign(x^T\beta^*)$) is no more than $\arccos(1/\sqrt{1+1/6})/\pi\approx 0.123$ (see \eqref{eq_cor_noise_euclid_bound}).
We rely on this assumption in particular when using lower bounds on Gaussian integrals, see Appendix \ref{app_bounded}.
Observe that in the case where $\sigma\geq 1/\sqrt{6}$ does not grow with $n,p,t$, we are back in the classical regime, where \cite{ostrovskii2021finite} provides the optimal rate of $\sqrt{p/n}$.

\subsubsection{Idea of the proof}
The proof of Theorem \ref{thm_cor_large} is given in Section \ref{sec_state}.
Here, we provide a sketch of the main ideas.
To obtain a fast rate, we use localization.
In other words, we exploit that the empirical process of the excess risk is smaller if $\hat{\gamma}$ is closer to $\gamma^*$.
To exploit this, we first create a convex combination $\tilde{\gamma}:=\alpha\hat{\gamma}+(1-\alpha)\gamma^*$.
Here, $\alpha$ is chosen such that on the one hand $\|\tilde{\gamma}-\gamma^*\|_2\leq \tau^*/6$, but on the other hand once $\|\tilde{\gamma}-\gamma^*\|_2$ is small enough, $\|\hat{\gamma}-\gamma^*\|_2\leq 2\|\tilde{\gamma}-\gamma^*\|_2$.
Moreover, as we assume $\tau^*\leq  M$, it holds that $P_nl(\tilde{\gamma})\leq P_nl(\gamma^*)$.
By convexity, also $P_nl(\tilde{\gamma})\leq P_nl(\gamma^*)$.
The separate Lemma \ref{lem_large_aux} shows that for such $\tilde{\gamma}$, the distance of $\tilde{\gamma}$ to $\gamma^*$ in the $*$-norm is small with high probability, i.e. $d_*(\tilde{\gamma})\lesssim\sqrt{(p\log (n)+t)/n}$.
On this event of large probability, it follows that $\tilde{\gamma}$ does not deviate from $\gamma^*$ by more than  $\tau^*/6$ in Euclidean distance.
It follows that $\|\hat{\gamma}-\gamma^*\|_2\leq \tau^*/6$ too.
So, we can apply Lemma \ref{lem_large_aux} again, with $\hat{\gamma}$ taking the role of $\tilde{\gamma}$, which allows us to conclude the proof.

So, the heart of the proof is Lemma \ref{lem_large_aux}.
It states that the localization $\|\tilde{\gamma}-\gamma^*\|_2\leq \tau^*/6$, together with the condition that the empirical loss of $\tilde{\gamma}$ is smaller than the empirical loss of $\gamma^*$ allows to conclude that $d_*(\tilde{\gamma})\lesssim \sqrt{(p\log (n)+t)/n}$.
This expression $d_*(\tilde{\gamma})$ occurs naturally from the Taylor expansion of the excess risk.
It provides a lower bound for the latter, which in turn can be bounded by the empirical process:
\[
d_*(\tilde{\gamma})^2:=\frac{|\tilde{\tau}-\tau^*|^2}{\tau^{*3}}+\tau^*\|\tilde{\beta}-\beta^*\|_2^2
\lesssim P\left(l(\tilde{\gamma})-l(\gamma^*)\right)
\leq (P-P_n)\left(l(\tilde{\gamma})-l(\gamma^*)\right).
\]
Here, as throughout this paper, we use the notation $\tilde{\tau}:=\|\tilde{\gamma}\|_2$ and $\tilde{\beta}:=\tilde{\gamma}/\|\tilde{\gamma}\|_2$.
This step already relies on the localization argument through the condition $|\tilde{\tau}-\tau^*|_2\leq \tau^*/6$.
Moreover, it requires that $\tilde{\tau},\tau^*\geq  \sqrt{6+\sqrt{51}}$.
These conditions appear throughout the proof of the large noise case.
They stem from lower bounds on Gaussian integrals, see Appendix \ref{app_bounded}.

To upper bound the empirical process, we split the loss into the bounded and unbounded parts.
These two are controlled with Bousquet's and Bernstein's inequality, together with covering and peeling arguments.
In both parts, we rely on the localization and that $\tau^*\gtrsim 1$ in several parts of the proof, for the same reasons as when proving the margin condition.
These calculations are given in Sections \ref{sec_con_b} and \ref{sec_con_u}.
We look at two cases, depending on whether the bounded or unbounded term dominates.
By doing so, we arrive at a bound that allows us to conclude that $d_*(\tilde{\gamma})\lesssim \sqrt{(p\log (n)+t)/n}$, which completes the proof.

\begin{rem}\label{rem_logn}
In the part of the proof where we assume that the bounded term dominates, the logarithmic factor $\log(en/p)$ would be sufficient.
In the unbounded part, however, we incur the logarithmic factor $\log(p)$.
Together, this is equivalent to the factor $\log(n)=\log(n/p)+\log(p)$.
The same holds for the small noise case.
\end{rem}

\subsection{Regime 2 (``small noise")}

\subsubsection{Statement of the result}

We state the main result for regime 2, where the noise is `small'.
The proof is given in Section \ref{sec_state}.
There, we provide a slightly more general version of the theorem.

\begin{thm}\label{thm_cor_small}
For any $t>0$, if
\[
\sigma\leq\frac{p\log n+t}{n}\lesssim 1,\quad M \geq \frac{n}{p\log n + t},
\]
then, with probability at least $1-6\exp(-t)$,
\[
\|\hat{\beta}-\beta^*\|_2\lesssim \frac{p\log n + t}{n},\quad 
\hat{\tau}\gtrsim \frac{n}{p\log n + t}.
\]
\end{thm}

We remark that the constants can be chosen, such that there is no gap between the two regimes.
The more general Theorem \ref{thm_small} allows arbitrary values $\sigma\lesssim 1$.
Hence, instead of $\sigma\leq (p\log n+t)/n$, one could impose the condition $\sigma\leq c(p\log n+t)/n$, with a constant $c$ matching the universal constant in the large noise case Theorem \ref{thm_cor_large}.
We chose $c=1$ here for aesthetic reasons.

The first interpretation of this result is, that if $\sigma$ is close to zero, the classifier $\hat{\beta}$ achieves the fast rate $(p\log n)/n$.
This is the same rate as a minimizer of the 0-1 loss obtains in the noiseless case, up to the factor $\log n$ (see \cite{long2003upper}).
Thus, in some sense, Theorem \ref{thm_cor_small} tells us that if $\sigma\leq (p\log n)/n$, the logistic regression classifier performs as if there was no noise.

There is another viewpoint, that may be interesting from a practical standpoint.
The condition could be rewritten as $n/\log (n)\leq p/\sigma$.
Informally, this suggests collecting data until regime 1 is reached, as improving the classifier is `cheaper' in regime 2.
Moreover, the first regime also brings the benefit that the signal-to-noise ratio can be estimated with $\hat{\tau}$, whereas here we only get a lower bound.
We note that this small sample conclusion which Theorem \ref{thm_small} allows is of a more philosophical than practical nature since the constants in the theorem are too large to be interesting for small sample sizes.

\subsubsection{Idea of the proof}
The proof of Theorem \ref{thm_cor_small} is given in Section \ref{sec_state}.
Here, we provide a sketch of the main ideas.
In the small noise regime, we cannot hope to estimate $1/\sigma\sim \tau^*$ (see e.g. the discussion in Section \ref{sub_star_norm}).
However, we should be able to classify well and thus estimate $\hat{\beta}$.
Furthermore, we can show that $\hat{\tau}:=\|\hat{\gamma}\|_2$ is somewhat large.

The big technical obstacle is, that we have no lower bound on the noise-to-signal ratio $\sigma$.
The first problem this brings is that no matter how large we choose $M$, since we do not know $\sigma$ it may be that $\|\gamma^*\|_2>M$.
If so, then potentially $P_nl(\hat{\tau}\hat{\beta})>P_nl(\tau^*\beta^*)$.
However, $P_nl(\hat{\tau}\hat{\beta})\leq P_nl(\hat{\tau}\beta^*)$ still holds.
While comparing $\hat{\tau}\hat{\beta}$ to $\hat{\tau}\beta^*$ is not enough to estimate $\|\hat{\gamma}-\gamma^*\|_2$, this is not our ambition here in the first place.
It is enough to allow us to make statements about $\|\hat{\beta}-\beta^*\|_2$ and $\hat{\tau}$, which is all that we want.

Therefore, we compare $\hat{\tau}\hat{\beta}$ with $\hat{\tau}\beta^*$.
Here, the excess risk is directly linked to the Euclidean distance as follows: 
\[
\hat{\tau}\|\hat{\beta}-\beta^*\|_2^2
\sim P(u(\hat{\tau}\hat{\beta})-u(\hat{\tau}\beta^*))
=P(l(\hat{\tau}\hat{\beta})-l(\hat{\tau}\beta^*) ).
\]
If we had access to a good lower bound on $\hat{\tau}$, we would be close to being finished.
However, at this point no lower bound for $\hat{\tau}$ is available (although our proof eventually provides such a lower bound).
A strategy to overcome this obstacle is given by the inequality of arithmetic and geometric means:
\[
\|\hat{\beta}-\beta^*\|_2
=\sqrt{\frac{1}{\hat{\tau}}}\sqrt{\hat{\tau}\|\hat{\beta}-\beta^*\|_2^2}
\leq \frac{1}{2\hat{\tau}}+\frac{\hat{\tau}\|\hat{\beta}-\beta^*\|_2^2}{2}.
\]
This shows that if we could get the extra summand $1/\hat{\tau}$, then we could upper bound the distance between $\hat{\beta}$ and $\beta^*$ with the empirical process of the loss function.
We obtain this extra summand using that:
\begin{equation}\label{eq_trick_bounded}
\frac{1}{\hat{\tau}}\sim Pb(\hat{\tau}\hat{\beta})-Pb(2\hat{\tau}\beta^*).
\end{equation}
Simultaneously, the unbounded term is homogeneous in $\hat{\tau}$, and easy to control for small $\sigma$.
This brings us to the comparison of $\hat{\tau}\hat{\beta}$ with $2\hat{\tau}\beta^*$.
Unfortunately, it is too early to celebrate.
Now we face the issue that possibly $2\hat{\tau}>M$, which we wanted to avoid in the first place.
However, this can be resolved with a case distinction: 
We either consider the case $M/2\leq \hat{\tau}\leq M$ or $\hat{\tau}\leq M/2$.

In case 1 we assume $M/2\leq \hat{\tau}\leq M$, so we have excellent control on $\hat{\tau}\sim M$.
We then use that:
\[
\frac{1}{\hat{\tau}}\lesssim Pb(\hat{\tau}\hat{\beta})-Pb(2\hat{\tau}\beta^*)
\lesssim Pb(\hat{\tau}\hat{\beta})-Pb(\hat{\tau}\beta^*)+\frac{1}{M}.
\]
Since we require that the radius $M$ is large, this term is fine.
We are back to comparing $\hat{\tau}\hat{\beta}$ with $\hat{\tau}\beta^*$, so we can use that $P_nl(\hat{\tau}\hat{\beta})\leq P_nl(\hat{\tau}\beta^*)$.
We proceed to upper bound the bounded and unbounded terms separately, using Bousquet's inequality and Bernstein's inequality.
In these upper bounds, we benefit from the assumption that $\sigma$ is small, to get fast rates.

In case 2, we can directly compare $\hat{\tau}\hat{\beta}$ with $2\hat{\tau}\beta^*$ since $2\hat{\tau}\leq M$ and so $P_nl(\hat{\tau}\hat{\beta})\leq P_nl(2\hat{\tau}\beta^*)$.
The upper bounds on the bounded and unbounded terms are proved similarly as in case 1.
Although here, we additionally need to prove that $\hat{\tau}$ is lower bounded by $n/(p\log(n)+t)$.
This is possible since $1/\hat{\tau}$ is a lower-bound for the expectation of the difference $ Pl(\hat{\tau}\hat{\beta})-Pl(2\hat{\tau}\beta^*)$, which we show is small.

There is a final complication, which we have not mentioned yet.
For the lower bound in the trick \eqref{eq_trick_bounded} to work, we need to assume that $\hat{\tau}$ is somewhat large.
Since if $\hat{\tau}$ is upper-bounded by a constant (say $\sqrt{6}$), then in fact:
\[
\hat{\tau}\sim Pb(\hat{\tau}\hat{\beta})-Pb(2\hat{\tau}\beta^*),
\]
see Lemma \ref{lem_moment_bounded_difference}.
This is a problem, as we intended to use \eqref{eq_trick_bounded} to show that $\hat{\tau}$ cannot be too small.
What saves us is that the event $\hat{\tau}\lesssim 1$ occurs with small probability, if $n$ is large and $\tau^*$ is large.
Intuitively, as $\hat{\gamma}$ `converges to' $\gamma^*$, it cannot stay in a ball with a small constant radius with large probability, if the target $\gamma^*$ lies outside this ball.
We thus show that the event $\hat{\tau}\leq \sqrt{6}$ does not occur with large probability if $n$ is somewhat large.
To do so, we show that on the event $\hat{\tau}\leq \sqrt{6}$, with large probability,
\[
 \sqrt{\frac{p\log (n)+t}{n}}\gtrsim Pl(\hat{\tau}\hat{\beta})-Pl(2\sqrt{6}\beta^*)
 \geq  Pl(\sqrt{6}\beta^*)-Pl(2\sqrt{6}\beta^*)\sim 1.
\]
Here, we use that the signal-to-noise ratio is large, in particular $\tau^*>\sqrt{6}$.
This leads to a contradiction as soon as $n$ is large enough, so this case cannot occur.
The proof is thereby completed.

\subsection{A consequence about linear separation}

Theorem \ref{thm_cor_large} also allows a consequence for linear separation.
In particular, Proposition \ref{prop_cor_separation} below shows that in regime 1 the data are not linearly separable with large probability.
This means, for example, that logistic regression in its original formulation can be applied, as the classical optimization problem \eqref{eq_obj_classical} is well-posed with large probability.
This is a step towards a non-asymptotic analog of the results in \cite{candes2020phase}.

\begin{prop}\label{prop_cor_separation}
For any $t>0$, if:
\[
\frac{p\log n+t}{n}\lesssim \sigma\leq \frac{1}{\sqrt{6}},
\]
then, the data are linearly separable with probability at most $4\exp(-t)$.
\end{prop}

The proof is given in Section \ref{sub_sep}.
\newpage
\section{Geometry \& Fisher consistency}\label{sec_geo}
In this section, we introduce some results which are fundamental in several parts of the proofs.
First, we state Grothendieck's identity, which provides us with a duality between the classification error and the Euclidean distance.
Second, we show how to recover the unknown parameters $\beta^*$ and $\sigma$, given that we have a misspecified model.
Third, we introduce a new metric, which is tailored to our problem, allowing us to lower-bound the excess risk with a distance between our estimator and its target.

\subsection{The classification error and the Euclidean distance}

Here, we recall the duality between the classification error for linear classifiers and the angle between two vectors, which is given by Grothendieck's identity.
We show how to translate this to Euclidean distances.
Finally, we exploit this equivalence and the Gaussianity of the vector $(x^T,\epsilon)^T$ to control the probability of observing a ``wrong" label.

The label predicted by a linear classifier $\gamma\in\mathbb{R}^p\setminus\{0\}$ at location $x\in\mathbb{R}^p$ is $sign(x^T\gamma)$\footnote{We note that $sign(z)$ is not defined if $z=0$. As the distribution of our covariates $x$ are absolutely continuous with respect to the Lebesgue measure, these events occur with probability zero. Hence, we ignore this issue.}.
Note that scaling $\gamma$ by a positive factor does not change this prediction.
Consequentially, it suffices to study vectors on the sphere $S^{p-1}$.
Grothendieck's identity states that the probability that two vectors disagree about the label at a random position $x\sim\mathcal{N}(0, I_p)$ is proportional to their angle.
It is originally due to \citet[p. 50]{grothendieck1956resume}, who used it to prove Grothendieck's inequality (see e.g. \cite{vershynin2018high}).

\begin{thm}\label{thm_grothendieck}
Let $a,b\in\mathbb{R}^p\setminus\{0\}$ and $x\sim\mathcal{N}(0,I_p)$.
It holds that:
\[
\mathbb{P}[a^Txx^Tb\leq 0]=\frac{1}{\pi}\arccos\left(\frac{a^Tb}{\|a\|_2\|b\|_2}\right).
\]
\end{thm}
We now know how the disagreement between two vectors $a,b\in S^{p-1}$ about the label at a point $x$ relates to the angle between, the two vectors, i.e. their geodesic distance on $S^{p-1}$.
Part of our analysis, later on, is of Euclidean nature, hence we would prefer to measure with the Euclidean distance.
Fortunately, the geodesic and Euclidean distances are bilipschitz equivalent.

\begin{prop}\label{prop_bilip}
For any $\beta,\beta'\in S^{p-1}$,
\[
\|\beta-\beta'\|_2\leq\arccos(\beta^T\beta')\leq \frac{\pi}{2}\|\beta-\beta'\|_2.\]
\end{prop}

Using that $\sigma\epsilon$ is a centred Gaussian with variance $\sigma^2$, the probability of observing a wrong label is:
\begin{equation}\label{eq_cor_noise_euclid_bound}
\mathbb{P}[yx^T\beta^*<0]=\frac{1}{\pi}\arccos\left(\frac{1}{\sqrt{1+\sigma^2}}\right),
\end{equation}
by Grothendieck's identity.
One easily verifies that:
\begin{equation}\label{eq_pi_bound}
\frac{\sigma}{\pi(1+\sigma^2)}
\leq
\mathbb{P}[yx^T\beta^*<0]\leq \frac{\sigma}{\pi}.
\end{equation}

\subsection{Fisher consistency and misspecification}
Our loss function \eqref{eq_opt} minimizes an empirical risk.
We will exploit techniques from M-estimation (see e.g. \cite{van2000empirical}) to prove that this empirical risk minimizer converges to the minimizer of the `true risk', i.e.
\[
\gamma^*:=\argmin_{\gamma\in\mathbb{R}^p}P\log(1+\exp(-yx^T\gamma)).
\]
Here and throughout, we use $P$ to denote the expectation.
If $\gamma^*$ equals the parameter of interest $\beta^*/\sigma$, we say that our estimator is \textit{Fisher consistent}.

Recall that our data comes from a probit model, a generalized linear model with the Gaussian link function.
To estimate $\beta^*/\sigma$, we use the logistic link function, so we cannot exploit maximum likelihood theory to conclude that $\gamma^*=\beta^*/\sigma$.
In fact, this turns out to be false.
However, there is a one-to-one correspondence between $\tau^*$ and $\sigma$, so that finding $\tau^*$ is equivalent to finding $\sigma$.
We start by establishing this relationship.
Then, we show that the estimator is Fisher consistent in direction, i.e. $\beta^*$.
Finally, we show that up to constant factors, $\|\gamma^*\|_2$ behaves like $1/\sigma$, aiding us in our calculations.

\subsubsection{Fisher consistency for the signal-to-noise ratio} 
Recall the notation $\tau^*:=\|\gamma^*\|_2$.
As hinted before, the signal-to-noise ratio $1/\sigma$ and the length of the target vector $\tau^*:=\|\gamma^*\|_2$ are generally not equal.
Yet, they are in a one-to-one correspondence.
This follows from the next lemma.

\begin{lem}\label{lem_tau_sigma_derivatives}
Let $\sigma>0$ and $z\sim\mathcal{N}(0,1)$. Then:
\[
P\frac{|z|}{1+\exp(\tau^*|z|)}=\frac{1}{\sqrt{2\pi}}\left(1-\frac{1}{\sqrt{1+\sigma^2}}\right).
\]
\end{lem}

\begin{proof}
We recall:
\[
\log(1+\exp(-yx^T\gamma))=\log(1+\exp(-|x^T\gamma|))+|x^T\gamma|1\{yx^T\gamma<0\}.
\]
Let $R:(0,\infty)\rightarrow\mathbb{R}$, $\tau\mapsto P\log(1+\exp(-yx^T(\tau\beta^*)))$.
Differentiating $R$ twice, we see that $R$ is strictly convex.
So, it reaches its minimum (if at all) at a root of its derivative.
This gives us the following condition:
\[
P\frac{|x^T\beta^*|}{1+\exp(\tau^*|x^T\beta^*|)}=P|x^T\beta^*|1\{yx^T\beta^*<0\}.
\]
Since $x^T\beta^*\sim\mathcal{N}(0,1)$, the left-hand side is equal to the left-hand side in the conclusion.
The result now follows from Corollary \ref{cor_trig}.
\end{proof}

Using Lemma \ref{lem_tau_sigma_derivatives}, we can show that $\tau^*$ and $1/\sigma$ behave similarly, up to multiplicative constants.

\begin{lem}\label{lem_sigmaVStau}
For $\sigma\in(0,1/\sqrt{2}]$,
\[
1\leq \sigma\tau^*\leq \sqrt{2\pi}.
\]
\end{lem}

\begin{proof}
We first prove the lower bound.
By Lemma \ref{lem_tau_sigma_derivatives}, if $\sigma\leq 1/\sqrt{2}$ then $\tau^*>\sqrt{6}$.
Therefore, by Lemma \ref{lem_bound_noiseless}, Lemma \ref{lem_tau_sigma_derivatives} and since $1-\frac{1}{\sqrt{1+\sigma^2}}\leq \frac{\sigma^2}{2}$,
\[
 \frac{1}{\sqrt{2\pi}}\frac{1}{2\tau^{*2}}
\leq \sqrt{\frac{1}{2\pi}}\frac{1}{\tau^{*2}}\left(1-\frac{3}{\tau^{*2}}\right)\leq 
P\frac{|z|}{1+\exp(\tau^*|z|)}
\]
\[
=\frac{1}{\sqrt{2\pi}}\left(1-\frac{1}{\sqrt{1+\sigma^2}}\right)\leq \frac{1}{\sqrt{2\pi}}\frac{\sigma^2}{2}.
\]
For the upper bound, we use Lemma \ref{lem_bound_noiseless}, Lemma \ref{lem_tau_sigma_derivatives} and that for $\sigma\leq 1/\sqrt{2}$, $\sigma^2/\pi\leq 1-\frac{1}{\sqrt{1+\sigma^2}}$, which gives: 
\[
\sqrt{\frac{2}{\pi}}\frac{1}{\tau^{*2}}\geq 
P\frac{|z|}{1+\exp(\tau^*|z|)}=\frac{1}{\sqrt{2\pi}}\left(1-\frac{1}{\sqrt{1+\sigma^2}}\right)\geq \frac{1}{\sqrt{2\pi}}\frac{\sigma^2}{\pi}.
\]
\end{proof}

\subsubsection{Fisher consistency of orientation}
Here, we show that the logistic regression estimator achieves Fisher consistency in our model for the direction, for any fixed length.
The proof is readily generalized to centrally symmetric distributions, non-identity covariance matrices, and other monotone loss functions.

\begin{lem}\label{lem_opt_direction}
Fix an element $\beta^*\in S^{p-1}$.
Then, 
\[
\beta^*=\argmin_{\beta\in S^{p-1}}Pl(yx^T\beta).
\]
\end{lem}

\begin{proof}
Fix any $\beta\in S^{p-1}$ and define the corresponding label $\hat{y}:=sign(x^T\beta)$, as well as $y':=sign(x^T\beta^*-sign(x^T\beta^*)\epsilon)$.
Since $\epsilon$ is independent of $x$ and $\epsilon$ has the same distribution as $-\epsilon$, $y$ and $y'$ have the same distribution. 
Moreover, $yx^T\beta$ and $y'x^T\beta$ have the same distribution.
To simplify notation, define the auxiliary function $L(z):=l(-|z|)-l(|z|)$. 
We find:
\[
Pl(yx^T\beta)=Pl(y'x^T\beta)
\]
\[
=Pl(|x^T\beta|)\left(1\{y' =y^*\}1\{y^*=\hat{y}\}+1\{y'\neq  y^*\}1\{y^*\neq  \hat{y}\}\right)
\]
\[
+Pl(-|x^T\beta|)\left(1\{y'\neq y^*\}1\{y^*= \hat{y}\}+1\{y'= y^*\}1\{y^*\neq \hat{y}\}\right)
\]
\[
=Pl(|x^T\beta|)\left((1-1\{y'\neq y^*\})(1-1\{y^*\neq \hat{y}\})+1\{y'\neq y^*\}1\{y^*\neq \hat{y}\}\right)
\]
\[
+Pl(-|x^T\beta|)\left(1\{y'\neq y^*\}(1-1\{y^*\neq \hat{y}\})+(1-1\{y'\neq y^*\})1\{y^*\neq\hat{y}\}\right)
\]
\[
=
P1\{\hat{y}\neq y^*\}(1-2\cdot 1\{y'\neq y^*\})L(x^T\beta)
+Pl(|x^T\beta|)+PL(x^T\beta)1\{y'\neq y^*\}
\]
\[
\stackrel{(i)}{=}
P1\{\hat{y}\neq y^*\}(1-2\Phi(-|x^T\beta^*|))L(x^T\beta)
+Pl(|x^T\beta|)+PL(x^T\beta)1\{|x^T\beta^*|\leq \epsilon\}.
\]
Equality (i) follows after conditioning on $x$, and recalling that $\{y'\neq y^*\}=\{|x^T\beta^*|\leq \epsilon\}$.
Note that if $\beta=\beta^*$, this reads,
\begin{equation}\label{eq_lem_opt_direction_star}
Pl(yx^T\beta^*)=Pl(|x^T\beta^*|)+PL(x^T\beta^*)1\{|x^T\beta^*|\leq \epsilon\}.
\end{equation}
Since $l$ is strictly decreasing, $L(x^T\beta)>0$ almost surely.
So, if instead $\beta^*\neq \beta$, using that by rotational invariance, $Pl(|x^T\beta|)=Pl(|x^T\beta^*|)$,
\[
Pl(yx^T\beta)-Pl(|x^T\beta^*|)>PL(x^T\beta)1\{|x^T\beta^*|\leq \epsilon\}
\]
\[
=
PL(x^T\beta)1\{|x^T\beta^*|\leq \epsilon\}1\{|x^T\beta|\leq \epsilon\}
+L(x^T\beta)1\{|x^T\beta^*|\leq \epsilon\}1\{|x^T\beta|> \epsilon\}
\]
\[
\stackrel{(ii)}{>}PL(x^T\beta^*)1\{|x^T\beta^*|\leq \epsilon\}1\{|x^T\beta|\leq \epsilon\}
+L(x^T\beta^*)1\{|x^T\beta^*|\leq \epsilon\}1\{|x^T\beta|> \epsilon\}
\]
\[
=PL(x^T\beta^*)1\{|x^T\beta^*|\leq \epsilon\}.
\]
Inequality $(ii)$ follows since in the first summand, by rotational invariance we may exchange $x^T\beta$ and $x^T\beta^*$, while in the second one, $|x^T\beta|>|x^T\beta^*|$.
By \eqref{eq_lem_opt_direction_star}, we conclude that for any $\beta\neq \beta^*$, $Pl(yx^T\beta)>Pl(yx^T\beta^*)$. 
\end{proof}

\subsection{Introducing the $*$-norm}\label{sub_star_norm}
In the proof for the large noise case, we establish an upper bound on the excess risk $P(l(\hat{\gamma})-l(\gamma^*))$.
In the end, however, we want to show that $\hat{\gamma}$ is close to $\gamma^*$.
Ideally, we would want to lower-bound the excess risk with the (squared) Euclidean distance (this is sometimes referred to as the \textit{margin condition}).
This is not generally possible in our setting, as we argue below.
Yet, we introduce a new metric, for which we can prove such a margin condition.

Before we introduce the metric we will use, let us consider an example that shows why the Euclidean distance fails.
In a data set of realistic sample size, we see the correct labels with high probability, whether $\gamma^*$ is equal to $10^{10}e_1$, or $10^{20}e_1$.
So, both these vectors would give us the same observations with high probability.
However, the Euclidean distance between these two vectors is of order $10^{20}$.
Therefore, attempting to explain the behavior of logistic regression with the Euclidean distance is hopeless, at least in the absence of bounds on the signal-to-noise ratio (SNR).

We will overcome this issue by separately comparing the vector's lengths and orientations.
We will invoke a distance over $\mathbb{R}^p\setminus\{0\}\times \mathbb{R}^p\setminus\{0\}$, or alternatively a norm on $\mathbb{R}^{p+1}$:
\[
\|\cdot\|_*:\mathbb{R}^{p+1}\rightarrow [0,\infty),\quad (\tau,\beta)\mapsto \sqrt{\frac{|\tau|^2}{\tau^{*3}}+\tau^*\left\|\beta\right\|_2^2}.
\]
The reader will already anticipate from our notation that the first dimension is used to compare the length of two vectors $\gamma,\gamma'$, and the remaining $p$ dimensions are used to compare orientations.
In particular, we apply this norm as follows:
\[
\|(\tau,\beta)-(\tau^*,\beta^*)\|_*= \sqrt{\frac{|\tau-\tau^*|^2}{\tau^{*3}}+\tau^*\left\|\beta-\beta^* \right\|_2^2}.
\]
As we will primarily employ this norm to quantify the distance between $\gamma$ and $\gamma^*$, we use the following, more compact notation:
\[
d_*:\mathbb{R}^p\setminus\{0\}\rightarrow[0,\infty),\quad \gamma\mapsto \left\|\left(\|\gamma\|_2,\frac{\gamma}{\|\gamma\|_2}\right)-(\tau^*,\beta^*)\right\|_*.
\]

If the SNR $\tau^*$ is large, more weight is given to the orientation. The latter can still be estimated in this case. 
On the other hand, if the SNR is small, it can be estimated, and so the $\|\cdot\|_*$-norm gives more weight to the difference between the two lengths of the vectors.

Although this will not play a role in our analysis, we briefly comment in what sense $\|\cdot\|_*$ is a norm.
As it is just a weighted Euclidean norm, it is itself a norm on $\mathbb{R}^{p+1}$.
However, if understood as the function $\mathbb{R}^p\setminus\{0\}\rightarrow[0,\infty)$, mapping $\gamma\mapsto \|(\|\gamma\|_2,\gamma/\|\gamma\|_2)\|_*$, it is not a norm, as it is not absolutely homogeneous, and not defined at $\gamma=0$.
From this viewpoint, it is quickly verified that this function still is a metric for the space $\mathbb{R}^p\setminus\{0\}$.

In the following, we show that the $\|\cdot\|_*$-norm and the Euclidean distance behave similarly, up to a constant depending on $\tau^*$.
We again use the notation $\tau:=\|\gamma\|_2$ and $\beta:=\gamma/\tau$ for an arbitrary vector $\gamma\in\mathbb{R}^p$.

\begin{lem}\label{lem_proxVSeuclid}
If $\tau^*:=\|\gamma^*\|_2\geq1$, then for all $\gamma\in\mathbb{R}^p\setminus\{0,\gamma^*\}$,
\[
\frac{\sqrt{\tau^*}}{3}\leq \frac{\|\gamma-\gamma^*\|_2}{d_*(\gamma)}\leq \sqrt{2\tau^{*3}}.
\]
\end{lem}

\begin{proof}
We prove the upper bound.
Using a triangular inequality, Jensen's inequality, and that $\tau^*\geq 1$, we find:
\[
\|\gamma-\gamma^*\|_2\leq |\tau-\tau^*|+\tau^*\|\beta-\beta^*\|_2
\]
\[
\stackrel{(i)}{\leq} \sqrt{2}\sqrt{|\tau-\tau^*|^2+\tau^*\|\beta-\beta^*\|_2^2}
\leq \sqrt{2\tau^{*3}}\sqrt{\frac{|\tau-\tau^*|^2}{\tau^{*3}}+\tau^*\|\beta-\beta^*\|_2^2}.
\]
This proves the upper bound.
For the lower bound, we combine two triangular inequalities, namely $\|\gamma-\gamma^*\|_2\geq |\tau-\tau^*|$ and $\|\gamma-\gamma^*\|_2\geq \tau^*\|\beta-\beta^*\|_2-|\tau-\tau^*|$.
Using that $\tau^*\geq 1$, we find:
\[
3\|\gamma-\gamma^*\|_2\geq \tau^*\|\beta-\beta^*\|+|\tau^*-\tau| \geq \sqrt{\tau^*}\left(\sqrt{\tau^*}\|\beta-\beta^*\|_2+\frac{|\tau^*-\tau|}{\tau^{*3/2}}\right).
\]
The inequality $a+b\geq \sqrt{a^2+b^2}$, which holds for any positive scalar $a,b>0$, completes the proof.
\end{proof}

\newpage
\section{Concentration for the bounded term}\label{sec_con_b}
In this section, we control the empirical process of the bounded term.
We define the bounded term as:
\[
b:\mathbb{R}^p\times\mathbb{R}^p\rightarrow (0,\log 2],\quad
(\gamma,x)\mapsto\log(1+\exp(-|x^T\gamma|)).
\]
We often omit the argument $x$, and write $b(\gamma)$, which is to be understood as the random variable $b(\gamma,x)$.
Moreover, we are often interested in the difference of the bounded term evaluated at $\gamma$ with the same term evaluated at $\gamma^*$, which we give the notation $\tilde{b}(\gamma):=(b(\gamma)-b(\gamma^*))/\log 2$.

\subsection{Bousquet's inequality}

We will use Bousquet's inequality to control the bounded term's empirical process.
It was introduced in \cite[Theorem 2.3]{bousquet2002bennett}.

\begin{thm}[{\cite{bousquet2002bennett}}]\label{thm_bousquet}
Let $X_1,\ldots, X_n$ be identically distributed random variables with values in $\mathcal{X}$.
Let $\mathcal{F}$ be a countable class of measurable mappings $\mathcal{X}\rightarrow\mathbb{R}$.
Suppose that there exists a $\varsigma>0$, such that for all $f\in\mathcal{F}$, $Var[f(X_1)]\leq \varsigma^2$ and $\|f\|_\infty\leq 1$.
Define:
\[
Z:=\sup_{f\in\mathcal{F}}|(P_n-P)f|.
\]
For $t>0$, with probability at most $\exp(-t)$,
\[
Z\geq PZ+\sqrt{\frac{2t(\varsigma^2+2PZ)}{n}}+\frac{t}{3n}.
\]
\end{thm}

We will use two simplifications for this lower bound.
Using the sub-additivity of the square root and the inequality between arithmetic and geometric means (AM-GM), 
\[
PZ+\sqrt{\frac{2t(\varsigma^2+2PZ)}{n}}+\frac{t}{3n}
\leq 2PZ+\varsigma\sqrt{\frac{2t}{n}}+\frac{4t}{3n}.
\]
On the other hand, for any $\lambda>0$, by AM-GM,
\[
PZ+\sqrt{\frac{2t(\varsigma^2+2PZ)}{n}}+\frac{t}{3n}
\leq PZ(1+\lambda^{-1})+\frac{\varsigma^2}{2\lambda}+\frac{t(\lambda+1/3)}{n}.
\]
The reader will already anticipate that we will choose $\mathcal{F}$ as a subset of all functions $\tilde{b}(\gamma):=(b(\gamma)-b(\gamma^*))/\log2$, which we denote by\footnote{Bousquet's inequality requires that the set $\mathcal{F}$ is countable to avoid measurability issues.
We will omit this technical discussion.}:
\[
\mathcal{B}:=\left\{
\tilde{b}(\gamma):\gamma\in\mathbb{R}^p
\right\}.
\]
In particular, due to our localization argument, we will only need to control the terms $\tilde{b}$ indexed by $\gamma$, which are somewhat close to $\gamma^*$.
Moreover, recall that in the end, our goal is to prove that some $\gamma$ are close to $\gamma^*$. 
Therefore, we do not need to control terms indexed by $\gamma$ which are very close to $\gamma^*$, as for such $\gamma$ there is nothing left to prove from the beginning.
This will allow us to apply the peeling technique, see e.g. \cite{van2000empirical}.

\subsubsection{Expectation of the supremum of the empirical process}

Here, we bound the term $PZ$, as introduced in Bousquet's inequality.
In the proof, we will take a covering with respect to the $L^1(P_n)$-norm.
Since $P_n$ is a random measure, we will bound the covering number with the largest $L^1(Q)$-covering number, where $Q$ is any probability measure on $\mathbb{R}^p$.
This is possible with the help of Vapnik-Chervonenkis (VC) theory.
We denote the set of all probability measures over $\mathbb{R}^p$ by $\Pi$, and a minimal covering of a set $S\subset\mathcal{B}$ by balls of radius $\epsilon$ in the $L^1(Q)$-norm by $\mathcal{N}(\epsilon,S,L^1(Q))$.

\begin{lem}\label{lem_VC}
For any $\epsilon>0$, and any probability measure $Q$ on $\mathbb{R}^p$,
\[
|\mathcal{N}(\epsilon,\mathcal{B},L^1(Q))|
\leq 2e(p+2.5)\left(\frac{4e}{\epsilon}\right)^{2(p+2)}.
\]
\end{lem}

This result follows from Vapnik-Chervonenkis theory.
If $\nu$ is the VC-index of the hypographs of $\mathcal{B}$, then by \cite[Corollary 3]{haussler1995sphere}, for any $\epsilon>0$,
\[
|\mathcal{N}(\epsilon,\mathcal{B},L^1(Q))|\leq e\nu\left(\frac{4e}{\epsilon}\right)^{\nu-1}.
\]
Using elementary properties of the VC-index \citep[Lemma 2.6.17-18]{van1996weak}, we can bound $\nu$ from above with $2(p+2)+1$.

We now turn to the bound for the expectation of the supremum of the empirical process of the bounded term $\tilde{b}(\gamma):=(b(\gamma)-b(\gamma^*))/\log(2)$.
We recall that this is a stepping stone for Bousquet's inequality, which gives us a tail-bound for the supremum of the empirical process.
In the following, for a set $S\subset\mathcal{B}$, by $\tilde{b}^{-1}(S):=\{\gamma\in\mathbb{R}^p:\tilde{b}(\gamma)\in S\}$ we mean the pre-image of $S$ through the mapping $\mathbb{R}^p\rightarrow \mathcal{B}$, $\gamma\mapsto \tilde{b}(\gamma)$.

For any subset $S\subset\mathcal{B}$, and $\epsilon>0$, we define:
\[
N_r(S):=\sup_{Q\in \Pi}|\mathcal{N}(r,S,L^1(Q))|,\quad \varsigma_S:=\sup_{\gamma\in \tilde{b}^{-1}(S)}\sqrt{P\tilde{b}(\gamma)^2}.
\]
\begin{prop}\label{prop_bounded_sup}
Let $S\subset \mathcal{B}$. Then,
\[
\mathbb{E}\sup_{\gamma\in \tilde{b}^{-1}(S)}|(P_n-P)\tilde{b}(\gamma)|
\leq 2r+\frac{16\log(2N_r(S))}{n}+8\varsigma_S\sqrt{\frac{\log 2N_r(S)}{n}}.
\]
\end{prop}

This proposition presupposes that we have good control over the covering number $N_r(S)$ and the wimpy variance $\varsigma_S$.
The covering number can be controlled with Lemma \ref{lem_VC}, and the wimpy variance will later be bounded using Lemma \ref{lem_uniform_variance}.

\begin{proof}
The proof uses standard steps from empirical process theory.
The first step is symmetrization (see e.g. \cite[Lemma 2.3.1]{van1996weak}). Let $\varepsilon_1,\ldots,\varepsilon_n$ be Rademacher, independent of each other and independent of $(x_1,y_1),\ldots,(x_n,y_n)$.
Then,
\[
\mathbb{E}\left(\sup_{\gamma\in \tilde{b}^{-1}(S)}|(P_n-P)\tilde{b}(\gamma)|\right)
\leq 2\mathbb{E}\left(\sup_{\gamma\in \tilde{b}^{-1}(S)}|P_n\varepsilon \tilde{b}(\gamma)|\right),
\]
where $P_n\varepsilon \tilde{b}(\gamma):=n^{-1}\sum_{i=1}^n\varepsilon_i\tilde{b}(\gamma,x_i)$.
Next, we discretize.
There exists a covering $\mathcal{N}:=\mathcal{N}(1/n,S,L^1(P_n))$ with cardinality $N\leq N_r(S)$.
So,
\[
2\mathbb{E}\left(\sup_{\gamma\in \tilde{b}^{-1}(S)}|P_n\varepsilon \tilde{b}(\gamma)|\right)= 2\mathbb{E}\left(\sup_{\gamma\in \tilde{b}^{-1}(S)}\min_{\tilde{b}'\in\mathcal{N}}
|P_n\varepsilon \tilde{b}(\gamma)-P_n\varepsilon \tilde{b}'+P_n\varepsilon \tilde{b}'| \right)
\]
\[
\leq  2\mathbb{E}\left(\sup_{\gamma\in \tilde{b}^{-1}(S)}\min_{\tilde{b}'\in\mathcal{N}}
|P_n\varepsilon (\tilde{b}(\gamma)- \tilde{b}')|+\max_{\tilde{b}'\in\mathcal{N}}|P_n\varepsilon \tilde{b}'| \right)
\leq 2r+2\mathbb{E}\max_{\tilde{b}'\in\mathcal{N}}|P_n\varepsilon \tilde{b}'| .
\]
The next step is concentration.
Let $\mathbb{P}_\varepsilon$ be the conditional expectation with respect to $\varepsilon_1,\ldots,\varepsilon_n$, i.e. conditional on $x_1,\ldots,x_n$.
By \cite[Lemma 3.4]{dumbgen2010nemirovski},
\[
\mathbb{P}_\varepsilon\max_{\tilde{b}'\in\mathcal{N}}|P_n\varepsilon \tilde{b}'|
\leq \sqrt{\frac{2\log 2N}{n}}\sqrt{\max_{\tilde{b}'\in\mathcal{N}}P_n \tilde{b}^{\prime 2}}.
\]
Adding and subtracting the expectation, and by sub-additivity of the square root,
\[
\sqrt{\frac{2\log 2N}{n}}\sqrt{\max_{\tilde{b}'\in\mathcal{N}}P_n \tilde{b}^{\prime 2}}\leq  \sqrt{\frac{2\log 2N}{n}}\left(\sqrt{\max_{\tilde{b}'\in\mathcal{N}}|(P_n-P) \tilde{b}^{\prime 2}|}+\sqrt{\max_{\tilde{b}'\in\mathcal{N}}P\tilde{b}^{\prime 2}}\right)
\]
\[
\leq\frac{4\log 2N}{n}+\frac{1}{8}\max_{\tilde{b}'\in\mathcal{N}}|(P_n-P) \tilde{b}^{\prime 2}|+\varsigma_S\sqrt{\frac{4\log 2N}{n}}.
\]
In the last inequality, we used that the second summand is bounded by $\varsigma_S$ and that by the inequality of arithmetic and geometric means, for all $x,x'\geq0$, $\sqrt{xx'}\leq x/8+2x^\prime$.
\\
Finally, we take the expectation and rearrange the expressions.
Using that $N$ is bounded by the deterministic $N_r(S)$, we conclude that:
\[
\mathbb{E}\max_{\tilde{b}'\in\mathcal{N}}|P_n\varepsilon \tilde{b}'|
\leq \frac{4\log 2N_r(S)}{n}+\frac{1}{8}\mathbb{E}\max_{\tilde{b}'\in\mathcal{N}}|(P_n-P) \tilde{b}^{\prime 2}|+\varsigma_S\sqrt{\frac{4\log 2N_r(S)}{n}}
\]
\[
\stackrel{(i)}{\leq} \frac{4\log 2N_r(S)}{n}+\frac{1}{8}4\mathbb{E}\max_{\tilde{b}'\in\mathcal{N}}|P_n\varepsilon \tilde{b}'|+\varsigma_S\sqrt{\frac{4\log 2N_r(S)}{n}}.
\]
Inequality $(i)$ follows from symmetrization (see e.g. \cite[Lemma 2.3.1]{van1996weak}), as well as the contraction principle (see e.g. \cite[Theorem 11.6]{boucheron2013concentration}), noting that the mapping $z\mapsto z^2/2$ is a contraction on $[-1,1]$. 
We subtract $\mathbb{E}\max_{\tilde{b}'\in\mathcal{N}}|P_n\varepsilon \tilde{b}'|/2$ from both sides and conclude the proof.
\end{proof}

\subsubsection{Bounding the wimpy variance}
Here, we bound the supremum of the variances of the bounded terms.
We call this the \textit{wimpy variance}, in line with the terminology used in \cite{boucheron2013concentration}.
It turns out, that this expression is smaller, the closer the $\gamma$ are to $\gamma^*$, which pairs very well with our localization.
The first result will be for the large-noise case and the second for the small-noise case.

\paragraph{Large noise}
In the large noise case, we measure the proximity of $\gamma$ to $\gamma^*$ using the $*$-norm, respectively $d_*(\gamma)$, as introduced in Section \ref{sub_star_norm}.
We use that $\tau^*\gtrsim 1$ and $\tau\gtrsim \tau^*$.
This is possible due to the localization step in the proof of Theorem \ref{thm_large}, which gives $\|\gamma-\gamma^*\|_2\lesssim \tau^*$.
If it were that $\tau,\tau^*\lesssim 1$, then the term $P\left(b(\tau\beta)-b(\tau^*\beta)\right)^2$ would behave differently, see Lemma \ref{lem_moment_bounded_difference_variance}.

\begin{lem}\label{lem_uniform_variance}
Fix $\gamma\in\mathbb{R}^p\setminus\{0\}$ and define $\tau:=\|\gamma\|_2$ and $\beta:=\gamma/\tau$.
If $\tau^*\geq 6/5$ and $\tau\geq \tau^*5/6$ then:
\[
\sqrt{P\tilde{b}(\gamma)^2}\leq 3.57d_*(\gamma).
\]
\end{lem}

\begin{proof}
We start with a triangular inequality:
\[
\sqrt{P\left(b(\gamma)-b(\gamma^*)\right)^2}
\leq  \sqrt{P\left(b(\tau\beta)-b(\tau^*\beta)\right)^2}+\sqrt{P\left(b(\tau^*\beta)-b(\tau^*\beta^*)\right)^2}.
\]
We bound the first term with Lemma \ref{lem_moment_bounded_difference_variance}.
If $\tau^*\geq \tau$, using that $\tau\geq \tau^*5/6$,
\[
\sqrt{P\left(b(\tau\beta)-b(\tau^*\beta)\right)^2}
\leq \left(\frac{1}{2^5\pi}\right)^{1/4}\frac{|\tau-\tau^*|}{\tau}\sqrt{\frac{2}{\tau}+\frac{3}{4\tau^3}}
\]
\[
\leq \left(\frac{1}{2^5\pi}\right)^{1/4}\frac{|\tau-\tau^*|}{\tau^{3/2}}\sqrt{2+\frac{3}{4}}
\leq \sqrt{\frac{11}{4}\sqrt{\frac{1}{2^5\pi}}}\left(\frac{6}{5}\right)^{3/2}\frac{|\tau-\tau^*|}{\tau^{*3/2}}.
\]
If instead $\tau\geq \tau^*$, we directly get:
\[
\sqrt{P\left(b(\tau\beta)-b(\tau^*\beta)\right)^2}\leq\sqrt{\frac{11}{4}\sqrt{\frac{1}{2^5\pi}}}\frac{|\tau-\tau^*|}{\tau^{*3/2}}.
\]
The second term is bounded with Lemma \ref{lem_moment_bounded_variance_distance}, again using $\tau^*\geq 1$.
Then,
\[
\sqrt{P\left(b(\tau^*\beta)-b(\tau^*\beta^*)\right)^2}
\leq \sqrt{2}\left(\frac{1}{(2\pi)^{1/4}}+\sqrt{\frac{11}{8}\sqrt{\frac{2}{\pi}}}\right)\sqrt{\tau^*}\|\beta-\beta^*\|_2.
\]
Since:
\[\sqrt{\frac{11}{4}\sqrt{\frac{1}{2^5\pi}}\left(\frac{6}{5}\right)^{3}+2\left(\frac{1}{(2\pi)^{1/4}}+\sqrt{\frac{11}{8}\sqrt{\frac{2}{\pi}}}\right)^2}\leq 3.57\log(2),\] it follows from Cauchy-Schwarz that:
\[
\sqrt{P\left(b(\tau\beta)-b(\tau^*\beta)\right)^2}+\sqrt{P\left(b(\tau^*\beta)-b(\tau^*\beta^*)\right)^2}
\leq 3.57\log(2)d_*(\gamma).
\]
This completes the proof.
\end{proof}

\paragraph{Small noise}
Controlling the wimpy variance in the case where the noise is small is analogous to the large noise case. 
Thanks to the case distinction in the proof of Theorem \ref{thm_small}, we will only need to compare $\tau\beta$ with $k\tau\beta^*$, where $k\geq 1$ is some constant factor and $\tau\geq 1$ is arbitrary.
For $\delta>0$, we define the spherical cap of radius $\delta$ centered at $\beta^*$ as $S(\delta,\beta^*):=\{\beta\in S^{p-1}:\|\beta-\beta^*\|_2\leq \delta\}$.

\begin{lem}\label{lem_small_wimpy_variance}
Let $\tau\geq 1$, $k\geq 1$ and $\delta>0$.
Then:
\[
\sup_{\beta\in S^{p-1}\setminus S(\delta,\beta^*)}\sqrt{P\left(\frac{b(\tau\beta)-b(k\tau\beta^*)}{\log(2)}\right)^2}
\leq 5.37\delta\sqrt{\tau}+0.76\frac{(k-1)}{\sqrt{\tau}}.
\]
\end{lem}

\begin{proof}
To bound the wimpy variance, we use that by the triangular inequality, for any $\beta\in S^{p-1}$,
\[
\sqrt{P(b(\tau\beta)-b(k\tau\beta^*))^2}\leq \sqrt{P(b(\tau\beta)-b(\tau\beta^*))^2}+ \sqrt{P(b(\tau\beta^*)-b(k\tau\beta^*))^2}.
\]
By Lemma \ref{lem_moment_bounded_variance_distance}, since $\tau\geq 1$ and $\sqrt{2}((2\pi)^{1/4}+\sqrt{11/8\sqrt{2/\pi}})\leq \log(2)5.37$, we have:
\[
\sqrt{P(b(\tau\beta)-b(\tau\beta^*))^2}\leq 3.72\sqrt{\tau}\|\beta-\beta^*\|_2.
\]
Next, by Lemma \ref{lem_moment_bounded_difference_variance}, since $\tau\geq 1$, the second term is bounded by:
\[
P(b(\tau\beta^*)-b(k\tau\beta^*))^2\leq \frac{11}{4}\sqrt{\frac{1}{2^5\pi}}\frac{(k-1)^2}{\tau}
\leq 0.76\log(2)\frac{(k-1)^2}{\tau}.
\]
The proof is complete.
\end{proof}

\subsection{The empirical process of the bounded part}
\subsubsection{The large noise case}

We consolidate our results so far to give local control of the empirical process of the bounded part $b$ with Bousquet's inequality (Theorem \ref{thm_bousquet}).
We prove a local result, in a $\|\cdot\|_*$-ball of radius $R>0$ around the ground truth $(\tau^*,\beta^*)$.
This allows us to get a faster rate: if $d_*(\gamma)$ is small, so is the wimpy variance (see Lemma \ref{lem_uniform_variance}), and consequentially also the expectation of the supremum of the empirical process (see Proposition \ref{prop_bounded_sup}).
For the peeling argument to work, we remove a small ball of radius $r>0$ from the supremum.
This does not affect our results in the end, since if $d_*(\gamma)\lesssim r$ for $r$ a small enough, the desired result already holds.
Moreover, as we controlled the wimpy variance for $\tau\geq \tau^*5/6$, we exclude all $\gamma$ with $\|\gamma\|_2< \tau^*5/6$ from the supremum.

In the following, $B_\delta$ denotes the Euclidean ball of radius $\delta$, and for the $*$-norm, we define the ball $B^*_\delta:=\{\gamma\in\mathbb{R}^p\setminus\{0\}:d_*(\gamma)\leq \delta\}$.
Moreover, for $r,R>0$ we define the `peels' $B^*_{r,R}:=B_R^*\setminus B_r^*$.

\begin{prop}\label{prop_bounded_global}
Let $R>0$ and $r\in(0,1)$, and suppose that $\tau^*\geq 6/5$. 
Let $t':=t+\log\lceil \log R/r\rceil$.
With probability at least $1-\exp(-t)$,
\[
\sup_{\gamma\in B_{r,R}^*\setminus B_{\tau^*5/6}}\frac{|(P_n-P)\tilde{b}(\gamma)|- 4r-\frac{32(3p+4)\log(4e/r)+\frac{4}{3}t'}{n}}{d_*(\gamma)e 16 \cdot3.57\sqrt{\frac{2(3p+4)\log(4e/r)+t'}{n}}}
\leq 1.
\]
\end{prop}

\begin{proof}
We prepare the use of Bousquet's inequality (Theorem \ref{thm_bousquet}).
Since $r\leq 1$, by Lemma \ref{lem_VC}:
\[
N:=|\mathcal{N}(r,\mathcal{B},L^1(Q))|
\leq 2e(p+2.5)\left(4e/r\right)^{2(p+2)}\leq \frac{1}{2}\left(4e/r\right)^{3p+4}.
\]
Moreover, by Lemma \ref{lem_uniform_variance}, 
\[
\varsigma:=\sup_{\gamma\in B^*_\delta\setminus B_{\tau^*5/6}} \sqrt{P\tilde{b}(\gamma)^2}\leq 3.57\delta.
\]
By Proposition \ref{prop_bounded_sup},
\[
\mathbb{E}\sup_{\gamma\in B^*_\delta\setminus  B_{\tau^*5/6}}|(P_n-P)\tilde{b}(\gamma)|
\leq 2r+\frac{16\log (2N)}{n}+8\varsigma\sqrt{\frac{\log (2N)}{n}}
\]
\[
\leq  2r+\frac{16(3p+4)\log(4e/r)}{n}+8 \cdot3.57\delta\sqrt{\frac{(3p+4)\log(4e/r)}{n}}.
\]
We now use Bousquet's inequality (Theorem \ref{thm_bousquet}).
On an event with probability at least $1-\exp(-t)$, 
\[
\sup_{\gamma\in B^*_\delta\setminus  B_{\tau^*5/6}}|(P_n-P)\tilde{b}(\gamma)| 
\]
\[
\leq 4r+\frac{32(3p+4)\log(4e/r)+\frac{4}{3}t}{n}+16 \cdot3.57\delta\sqrt{\frac{2(3p+4)\log(4e/r)+t}{n}}
\]
\[
=:\overline{\psi}_b(t)+\underline{\psi}_b(\delta,t).
\]
Equivalently,
\begin{equation}\label{eq_bounded_large_local}
\sup_{\gamma\in B^*_\delta\setminus  B_{\tau^*5/6}}\frac{|(P_n-P)\tilde{b}(\gamma)|-\overline{\psi}_b(t)}{\underline{\psi}_b(\delta,t)}\leq 1.
\end{equation}
We now apply the peeling technique (see e.g. \cite{van2000empirical}), to replace $\delta$ by $d_*(\gamma)$.
We partition $B_{r,R}^*$ into the disjoint union $B_{r,R}^*=\cup_{j=0}^{J-1}B^*_{re^j,re^{j+1}}$, with $J=\lceil\log R/r\rceil$.
Then,
\[
\mathbb{P}\left[
\bigcup_{\gamma\in B_{r,R}^*\setminus  B_{\tau^*5/6}}\frac{|(P_n-P)\tilde{b}(\gamma)|-\overline{\psi}_b(t+\log J)}{\underline{\psi}_b(ed_*(\gamma),t+\log J)}> 1
\right]
\]
\[
\leq \sum_{j=0}^{J-1}\mathbb{P}\left[
\bigcup_{\gamma\in B_{re^j,re^{j+1}}^*\setminus  B_{\tau^*5/6}}\frac{|(P_n-P)\tilde{b}(\gamma)|-\overline{\psi}_b(t+\log J)}{\underline{\psi}_b(ed_*(\gamma),t+\log J)}> 1
\right]
\]
\[
\leq \sum_{j=0}^{J-1}\mathbb{P}\left[
\bigcup_{\gamma\in B_{re^j,re^{j+1}}^*\setminus  B_{\tau^*5/6}}\frac{|(P_n-P)\tilde{b}(\gamma)|-\overline{\psi}_b(t+\log J)}{\underline{\psi}_b(re^{j+1},t+\log J)}> 1
\right]
\]
\[
\stackrel{(ii)}{\leq }J\exp(-t-\log J)=\exp(-t).
\]
The inequality $(ii)$ follows from \eqref{eq_bounded_large_local}.
The proof is complete.
\end{proof}

\subsubsection{Small noise case}
Here, we prove two bounds for the empirical process of the bounded term in the case of small noise, based on Bousquet's inequality.
We will prove a tail bound on:
\[
\sup_{\beta\in S^{p-1}\setminus S(\delta,\beta^*)}|(P-P_n)(b(\tau\beta)-b(k\tau\beta^*))|.
\]
We will use different $k$ in the proof of the main theorem.
For part 1 of the proof ($M/2\leq \hat{\tau}\leq M$), we will use $k=1$, for part 2 ($1\leq \hat{\tau}\leq M/2$), we will use $k=2$.
In the following, we use the notation $\tilde{b}_k(\tau\beta):=(b(\tau\beta)-b(k\tau\beta^*))/\log2$.

\begin{prop}\label{prop_small_bounded_global}
Fix $k,\tau,\lambda\geq 1$ as well as $r\in(0,1)$ and $t>0$ and define $t':=t+\log\lceil\log 2/r\rceil$.
With probability at least $1-\exp(-t)$,
\[
\sup_{\beta\in  S^{p-1}\setminus S(r,\beta^*)}|(P-P_n)\tilde{b}_k(\tau\beta)|
\]
\[
\leq   13851\frac{\|\beta-\beta^*\|_2^2\tau}{\lambda}+4r
+30\frac{(k-1)^2}{\lambda\tau}
+\frac{\lambda34(3p+4)\log(4e/r)+t'(\lambda+\frac{1}{3})}{n}.
\]
\end{prop}

\begin{proof}
From Lemma \ref{lem_VC}, it follows that to cover the set of functions $\{(b(\gamma)-b(k\tau\beta^*))/\log(2):\gamma\in \mathbb{R}^p\}$ with $L^1(Q)$-balls of radius $r$ for any probability measure $Q$, one needs at most $N$ balls, where $\log2N\leq (3p+4)\log(4e/r)$, arguing as in the proof of Proposition \ref{prop_bounded_global}.
With Lemma \ref{lem_small_wimpy_variance}, the wimpy variance can be bounded as follows: 
\[
\varsigma
:=\sup_{\beta\in S^{p-1}\setminus S(\delta,\beta^*)}\sqrt{
P\tilde{b}_k(\tau\beta)^2}
\leq 5.37\delta\sqrt{\tau}+0.76\frac{(k-1)}{\sqrt{\tau}}.
\]
By Proposition \ref{prop_bounded_sup}, 
\[
\mathbb{E}\sup_{\beta\in  S(\delta,\beta^*)\setminus S(\delta,\beta^*)}\left|(P-P_n)\tilde{b}_k(\tau\beta)\right|
\]
\[
\leq 2r+\frac{16\log(2N)}{n}+8\frac{\varsigma}{\sqrt{\lambda}}\sqrt{\lambda\frac{\log 2N}{n}}
\]
\[
\leq 2r+16\frac{\varsigma^2}{\lambda}+\frac{\lambda17\log(2N)}{n}
\leq 2r+16\frac{\varsigma^2}{\lambda}+\frac{\lambda17(3p+4)\log(4e/r)}{n}.
\]
By Bousquet's inequality (Theorem \ref{thm_bousquet}), with probability at least $1-\exp(-t)$,
\[
\sup_{\beta\in  S^{p-1}\setminus S(\delta,\beta^*)}\left|(P-P_n)\tilde{b}_k(\tau\beta)\right|\leq 2PZ+\frac{\varsigma^2}{2\lambda}+\frac{t(\lambda+1/3)}{n}
\]
\[
\leq 4r+32.5\frac{\varsigma^2}{\lambda}+\frac{\lambda34(3p+4)\log(4e/r)+t(\lambda+1/3)}{n}
\]
\[
\leq \frac{65\cdot 5.37^2}{\lambda}\delta^2\tau
+4r+\frac{65\cdot 0.76^2}{\lambda}\frac{(k-1)^2}{\tau}
+\frac{\lambda34(3p+4)\log(4e/r)+t(\lambda+1/3)}{n}
\]
\[
=:\delta^2\underline{\psi}'_b+\overline{\psi}'_b(t).
\]
In other words, 
\begin{equation}\label{eq_small_bounded_local}
\mathbb{P}\left[
\sup_{\beta\in  S^{p-1}\setminus S(\delta,\beta^*)}\frac{|(P-P_n)\tilde{b}_k(\tau\beta)|-\overline{\psi}'_b(t)}{\delta^2\underline{\psi}'_b}>1
\right]\leq \exp(-t).
\end{equation}
Now, we use the peeling device.
For $a,b>0$, define: $
S_{a,b}^*:=\{\beta\in S^{p-1}:a\leq \|\beta-\beta^*\|_2\leq b\}$.
Note that for any $\beta\in S^{p-1}$, $\|\beta-\beta^*\|_2\leq 2$.
So, letting $J:=\lceil\log2/r\rceil$, we find: 
Consequentially, for $r>0$,
\[
S^{p-1}\setminus S(r,\beta^*)=\bigcup_{j=0}^{J-1} S_{re^j,re^{j+1}}^*.
\]
We want to upper bound:
\[
\mathbb{P}\left[\sup_{\beta\in S^{p-1}\setminus S(r,\beta^*)}\frac{|(P-P_n)\tilde{b}_k(\tau\beta)|-\overline{\psi}'_b(t+\log J)}{e^2\|\beta-\beta^*\|_2^2\underline{\psi}'_b}>1
\right]
\]
\[
\leq \sum_{j=0}^{J-1} \mathbb{P}\left[\sup_{\beta\in  S_{re^j,re^{j+1}}^*}
\frac{|(P-P_n)\tilde{b}_k(\tau\beta)|-\overline{\psi}'_b(t+\log J)}{e^2\|\beta-\beta^*\|_2^2\underline{\psi}'_b}>1
\right]
\]
\[
\leq \sum_{j=0}^{J-1} \mathbb{P}\left[\sup_{\beta\in  S_{0,re^{j+1}}^*}
\frac{|(P-P_n)\tilde{b}_k(\tau\beta)|-\overline{\psi}'_b(t+\log J)}{e^2r^2e^{2j}\underline{\psi}'_b}>1
\right].
\]
\[
\leq J\exp(-t-\log J)=\exp(-t).
\]
The last inequality follows from \eqref{eq_small_bounded_local}.
In the final result, we use that $65\cdot5.37^2e^2\leq 13851$, as well as the approximation $65\cdot0.76^2\leq 30$.
The proof is complete.
\end{proof}

\newpage
\section{Concentration for the unbounded term}\label{sec_con_u}
Here, we give two concentration inequalities for the empirical process of the unbounded term:
\[
u:\mathbb{R}^p\setminus\{0\}\times\mathbb{R}^p\times\{-1,1\}\rightarrow\mathbb{R},\quad
(\gamma,x,y)\mapsto |x^T\gamma|1\{yx^T\gamma<0\}.
\]
We will omit the dependence of $u$ on $x,y$, if it is clear from context, and write $u(\gamma)$ instead of $u(\gamma,x,y)$.
Moreover, we are frequently interested in the difference between the unbounded term evaluated at some point $\gamma$ and the parameter $\gamma^*$, for which we use the notation:
\[
\tilde{u}(\gamma):=u(\gamma)-u(\gamma^*).
\]
The first is for the ``large noise"-regime and the second is for the ``small noise"-regime.
In the former, we can exploit that due to localization, $\|\gamma-\gamma^*\|_2\lesssim \tau^*\sim 1/\sigma$, which is not too large.
In the small noise case, this is not helpful, as $1/\sigma$ may be arbitrarily large.
However, since we do not have the ambition to estimate $\tau^*$ in that case, we only need to compare $\beta$ and $\beta^*$.
Moreover, we can exploit that $\sigma$ is vanishingly small.

Both results are based on Bernstein's inequality.
In both cases, we proceed as follows:
First, we show that Bernstein's condition is satisfied, in a small region around the true parameter.
Second, we use a covering argument to control the empirical process.
Third, we apply peeling, to obtain a fast global rate, exploiting the good control locally around the true parameter.

\subsection{Bernstein's inequality \& covering}
Although we need two separate results for the large and the small-noise cases, there are some results we can use in both cases.
One is Bernstein's inequality itself, and the second is the covering argument.

\subsubsection{Bernstein's inequality}
Bernstein's inequality is a classical concentration inequality for the tail of sub-exponential random variables.
It states that under the ``Bernstein condition", the average of centered random variables concentrates around zero.
Bernstein's condition dictates that the moments of the random variable may not grow too fast. 
The parameters in Bernstein's condition may reflect, that the moments of the random variables are small.
For example, if the second moment is small, $\kappa$ can be chosen as small (simply take $m=2$ to see this).
We will exploit this, to obtain a fast convergence rate.

\begin{defn}
Let $\{z_i\}_{i\in[n]}$ be a sequence of random variables.
We say that $\{z_i\}_{i\in[n]}$ satisfies Bernstein's condition with constants $\{\kappa_i\}_{i=1}^n$ and $K$, if for all $i\in[n]$, 
$Pz_i=0$, and for all $m\in\{2,3,\ldots\}$, 
\[
P|z_i|^m\leq \frac{m!}{2}K^{m-2}\kappa_i^2.
\]
\end{defn}

Bernstein's inequality is formulated for centered random variables.
However, it is sometimes easier to verify Bernstein's condition for a sequence of random variables $\{z_i\}_{i\in[n]}$, instead of their centered version $\{z_i-Pz_i\}_{i\in[n]}$.
In the following remark, we show that this suffices, at the cost of a constant factor.

\begin{rem}\label{rem_bernstein}
Note that:
\[
P(2|z_i|)^m=2^{m-1}P(|z_i|^m+P|z_i|^m)
\]
\[
\stackrel{(i)}{\geq} 2^{m-1}P(|z_i|^m+|Pz_i|^m)\stackrel{(ii)}{\geq} P(|z_i|+|-Pz_i|)^m\stackrel{(iii)}{\geq} P|z_i-Pz_i|^m.
\]
Here, $(i)$ and $(ii)$ follow from Jensen's inequality, while $(iii)$ follows from the triangular inequality.
So, Bernstein's condition is met for $z_i-Pz_i$ with constants $\{\kappa_i\}_{i=1}^n$ and $K$, if for all $m\in\{2,3,\ldots,\}$:
\[
P|z_i|^m\leq \frac{m!}{2}\left(\frac{K}{2}\right)^{m-2}\left(\frac{\kappa_i}{2}\right)^2.
\]
\end{rem}

We will need a Bernstein-type inequality for a supremum of an empirical process.
To this end, we use a covering argument and a finite union bound.

\begin{lem}\label{lem_bernstein_max}
Let $L\in\mathbb{N}$.
For $l\in[L]$, let $\{z_{i,l}\}_{i\in[n]}$ be a sequence of random variables with $Pz_{i,l}=0$, satisfying Bernstein's condition with $K_l\leq K$ and $\kappa_{i,l}\leq \kappa$.
Let $t>0$. With probability at least $1-2\exp(-t)$:
\[
\max_{l\in[L]}|P_nz_l|
\leq \kappa\sqrt{\frac{2(t+\log L)}{n}}+K\frac{t+\log L}{n}.
\]
\end{lem}

\begin{proof}
By Bernstein's inequality \cite[Theorem 2.10]{boucheron2013concentration}, for all $l\in[L]$, with probability at most $2\exp(-t)$,
\[
|P_nz_{l}|\geq \kappa\sqrt{\frac{2t}{n}}+K\frac{t}{n}.
\]
In particular, with probability at most $2\exp(-t-\log L)$,
\[
|P_nz_{l}|\geq \kappa\sqrt{\frac{2(t+\log L)}{n}}+K\frac{t+\log L}{n}=:\psi.
\]
Using a union bound,
\[
P\left[\max_{l\in[L]}|P_nz_{l}|\geq \psi\right]
\leq 2L\exp(-t-\log L)=2\exp(-t).
\]
The proof is complete.
\end{proof}

\subsubsection{Covering}
To apply Bernstein's inequality (Lemma \ref{lem_bernstein_max}), which is for a maximum over a finite set, we use a covering argument.
For the unbounded terms $u$, we introduce an integrable envelope $U$, which dominates them pointwise.
Such an envelope can be found with the Cauchy-Schwarz inequality.
However, an upper bound on $\|\gamma\|_2$ will be needed, so we will assume that for some $r>0$, it holds that $\|\gamma\|_2\leq r$.
This is not a problem, as we can assume $\|\gamma-\gamma^*\|_2\lesssim\tau^*$ in the large noise case, and use $r=1$ in the small noise case.
We define:
\[
U:\mathbb{R}^p\times (0,\infty)\rightarrow\mathbb{R},\quad
(x,r)\mapsto r\|x\|_2.
\]
For any subset $S\subset\mathbb{R}^p\setminus\{0\}$, we define:
\[
\mathcal{U}(S):=\{u(\gamma):\gamma\in S\},\quad
\mathcal{U}(S)/U_r:=\left\{\frac{u(\gamma)}{U(\cdot,r)}:\gamma\in S\right\}.
\]
Note that for any $\gamma,\gamma'\in\mathbb{R}^p$, it holds that $u(\gamma)-u(\gamma')=\tilde{u}(\gamma)-\tilde{u}(\gamma')$.
So, a covering for $u$ is equivalent to a covering for $\tilde{u}$.

With Lemma \ref{lem_unbounded_discretize}, we show that a covering of $\mathcal{U}(S)/U_r$ with the supremum norm does not impair the rate of our empirical process significantly.
It then remains to calculate the covering number.
After a technical Lemma \ref{lem_covering_aux1}, we see that $u/U$ can be upper bounded by a quantity resembling the $\|\cdot\|_*$-norm in Lemma \ref{lem_covering_aux2}.
From there, we can derive a covering number using an upper bound for the covering number of the Euclidean unit ball in $\mathbb{R}^{p+1}.$

\begin{lem}\label{lem_unbounded_discretize}
Fix $r>0$.
Let $S\subset \mathbb{R}^p\setminus\{0\}$ and $\mathcal{N}\subset S$ be sets such that $\mathcal{U}(\mathcal{N})/U_r$ is an $\varepsilon$-covering of $\mathcal{U}(S)/U_r$ in the supremum-norm.
With probability at least $1-\exp(-t)$,
\[
\sup_{\gamma\in S}\left|(P_n-P)\tilde{u}(\gamma)\right|
\leq \sup_{\gamma\in \mathcal{N}}\left|(P_n-P)\tilde{u}(\gamma)\right|+
\varepsilon r\left(\sqrt{\frac{2t}{n}}+2\sqrt{p}\right).
 \]
\end{lem}

\begin{proof}
Fix any $\gamma\in S$ and an $\varepsilon$-neighbour $\gamma_\varepsilon\in\mathcal{N}$.
By a triangular inequality, 
\[
\left|(P_n-P)\tilde{u}(\gamma)\right|\leq 
\left|(P_n-P)(\tilde{u}(\gamma)-\tilde{u}(\gamma_\varepsilon))\right|
+\left|(P_n-P)\tilde{u}(\gamma_\varepsilon)\right|.
\]
All that is left is to upper bound the first summand on the right-hand side.
Note that $\tilde{u}(\gamma)-\tilde{u}(\gamma_\varepsilon)=u(\gamma)-u(\gamma_\varepsilon)$.
Using another triangular inequality,
\[
\left|(P_n-P)(u(\gamma)-u(\gamma_\varepsilon))\right|
\leq (P_n+P)U\left|\frac{u(\gamma)-u(\gamma_\varepsilon)}{U}\right|
\]
\[
\leq \varepsilon (P_n+P)U
=\varepsilon\left((P_n-P)U+2PU\right).
\]
Now, we use a concentration result.
The random variable $U(\cdot,r)/r$ is a 1-Lipschitz function of a Gaussian random variable.
So, using concentration for 1-Lipschitz functions of normally distributed random variables (see e.g. \cite[Theorem 5.5]{boucheron2013concentration}), we find that on an event with probability at least $1-\exp(-t)$, it holds that $(P_n-P)U(\cdot,r)/r\leq \sqrt{2t/n}$.
On this event,
\[
\varepsilon\left((P_n-P)U+2PU\right)
\leq 
\varepsilon\left(r\sqrt{\frac{2t}{n}}+2PU\right)
\leq \varepsilon r\left(\sqrt{\frac{2t}{n}}+2\sqrt{p}\right).
\]
This completes the proof.
\end{proof}

Before we turn our attention to controlling the supremum distance over $\mathcal{U}(S)/U_r$, we state an auxiliary result.

\begin{lem}\label{lem_covering_aux1}
Let $v_1,v_2\in S^{p-1}$. Then:
\[
\sup_{\alpha\in S^{p-1}:\alpha^Tv_2\leq 0}\alpha^Tv_1\leq \|v_1-v_2\|_2.
\]
\end{lem}

\begin{proof}
Choose any $\alpha\in S^{p-1}$.
First, suppose that $v_1^Tv_2\leq 0$.
Then, it holds that $\|v_1-v_2\|_2\geq \sqrt{2}>1\geq \alpha^Tv_1$.
Now suppose that $v_1^Tv_2\geq 0$.
Expand:
\[
\alpha^Tv_1
=\alpha^Tv_2v_2^Tv_1+\alpha^T(I-v_2v_2^T)v_1
\stackrel{(i)}{\leq} \|(I-v_2v_2^T)v_1\|_2
\]
\[
=\sqrt{v_1^T(I-v_2v_2^T)v_1}
=\sqrt{1-(v_1^Tv_2)^2}\leq \sqrt{2-2v_1^Tv_2}=\|v_1-v_2\|_2.
\]
In $(i)$, we used that $v_1^Tv_2\geq 0$ for the first term and Cauchy-Schwarz for the second.
The proof is complete.
\end{proof}

Here, we prove an upper bound for the supremum norm in terms of the orientation $\beta$ and the length $\tau$. 
As a consequence, we obtain a bound for the covering number.

\begin{lem}\label{lem_covering_aux2}
Let $r>0$ and define $\tau_1,\tau_2\in(0,r]$ as well as $\beta_1,\beta_2\in S^{p-1}$.
Then:
\[
\sup_{(x,y)\in\mathbb{R}^p\times\{-1,1\}}\left|\frac{u(\tau_1\beta_1,x,y)-u(\tau_2\beta_2,x,y)}{U(x,r)}\right|
\leq\|\beta_1-\beta_2\|_2+ \frac{|\tau_1-\tau_2|}{r}.
\]
Consequentially, it takes at most $(12/\varepsilon)^{p+1}$ many elements to cover $\mathcal{U}(B_r)/U_r$ with balls of radius $\varepsilon$ in the supremum distance.
\end{lem}

\begin{proof}
Fix any $(x,y)\in\mathbb{R}^p\times\{-1,+1\}$.
For $j\in\{1,2\}$, we define $1\{A_j\}:=1\{sign(x^T\beta_j)\neq y\}$.
We write out the numerator:
\[
|\tau_1|x^T\beta_1|1\{A_1\}-\tau_2|x^T\beta_2|1\{A_2\}|
=\begin{cases}
|\tau_1|x^T\beta_1|-\tau_2|x^T\beta_2||&\text{if }A_1\cap A_2\\
\tau_1|x^T\beta_1|& \text{if }A_1\cap A_2^c\\
\tau_2|x^T\beta_2|& \text{if }A_1^c\cap A_2\\
0&\text{else.}
\end{cases}
\]
If $A_1\cap A_2$, then:
\[
\frac{|\tau_1|x^T\beta_1|-\tau_2|x^T\beta_2||}{\|x\|_2r}
\leq \frac{\|\tau_1\beta_1-\tau_2\beta_2\|_2}{r}
\leq \frac{\tau_1\|\beta_1-\beta_2\|_2+|\tau_1-\tau_2|}{r}.
\]
Now suppose $A_1\cap A_2^c$ holds.
By Lemma \ref{lem_covering_aux1}:
\[
\frac{\tau_1|x^Tb_1|}{\|x\|_2r}\leq \frac{\tau_1\|\beta_1-\beta_2\|_2}{r}.
\]
The case $A_1^c\cap A_2$ is analogous.
To upper bound the covering number, take any $\tau_1\beta_1,\tau_2\beta_2\in U(B_r)/U_r$.
We just established that:
\[
\left\|\frac{u(\tau_1\beta_1)-u(\tau_2\beta_2)}{U(\cdot,r)}\right\|_\infty
\leq 2\sqrt{\left(\frac{\|\beta_1-\beta_2\|_2^2}{2}+\frac{|\tau_1-\tau_2|^2}{2r^2}\right)}.
\]
Note that on the right-hand side, we have the Euclidean distance over the set $\{(\beta/2,t/2):(\beta,t)\in S^{p-1}\times[0,1]\}$, which is a subset of the Euclidean unit ball in $\mathbb{R}^{p+1}$.
To cover the unit ball with elements in itself takes at most $(3/\varepsilon)^{p+1}$ balls, so to cover a subset of it with elements in itself takes at most $(6/\varepsilon)^{p+1}$ elements.
Taking care of the extra factor $2$ completes the proof.
\end{proof}

\subsection{The large noise case}
Here we provide an inequality for the empirical process of the unbounded term, in the case where the noise is large.
Due to a localization argument, we can assume that $\|\gamma-\gamma^*\|_2\lesssim \tau^*$ throughout, i.e. $\gamma\in B_r(\gamma^*)$, where $r\lesssim\tau^*$.
After verifying Bernstein's condition locally, in terms of the $\|\cdot\|_*$-norm, we use Lemma \ref{lem_bernstein_max}, based on Bernstein's inequality, to control the empirical process.

\subsubsection{Local verification of Bernstein's condition}
The following lemma is an intermediate step to verify Bernstein's condition.
We do not yet exploit any localization.

\begin{lem}\label{lem_bernstein_global}
Let $\gamma\in \mathbb{R}^p\setminus\{0\}$, $\tau:=\|\gamma\|_2$ and $\beta:=\gamma/\tau$.
Then, for all $m\in\{2,3,\ldots\}$,
\[
\frac{P|\tilde{u}(\gamma)|^m}{\Gamma\left(\frac{m+1}{2}\right)2^{m-1}}
\]
\[
\leq (\sqrt{32}\tau\|\beta-\beta^*\|_2)^m\left(\|\beta-\beta^*\|_2
+\frac{\sigma}{4\pi^{3/2}}\right)+\frac{\sigma}{4\pi}(\sqrt{8}\pi\sigma|\tau\beta^T\beta^*-\tau^*|)^m.
\]
\end{lem}

\begin{proof}
Define $y^*:=sign(x^T\gamma^*)$.
To start, note that:
\[
|\tilde{u}(\gamma,x,y)|^m
=\left|\frac{|x^T\gamma|-yx^T\gamma-(|x^T\gamma^*|-yx^T\gamma^*)}{2}\right|^m
\]
\[
=\left|\frac{|x^T\gamma|-y^*x^T\gamma}{2}-\frac{(y-y^*)x^T(\gamma-\gamma^*)}{2}\right|^m
\]
\[
\stackrel{(i)}{\leq} 2^{m-1}\left(|x^T\gamma|^m1\{\beta^Txx^T\beta^*<0\}+\frac{|(y-y^*)x^T(\gamma-\gamma^*)|^m}{2}\right).
\]
Inequality $(i)$ follows from Jensen's inequality, and the observation that the term $||x^T\gamma|-y^*x^T\gamma|/2$ is equal to $|x^T\gamma|1\{\beta^Txx^T\beta^*<0\}$.
We proceed to upper bound the expectations of those two summands separately.
The left term can be bounded using Corollary \ref{cor_trig}:
\[
P|x^T\gamma|^m1\{\beta^T x x^T\beta^*<0\}
\leq \tau^m\frac{1}{\sqrt{2}\pi}\frac{\Gamma(m/2+1)}{m+1}\left(\frac{\pi}{\sqrt{2}}\|\beta-\beta^*\|_2\right)^{m+1}
\]
\[
\stackrel{(ii)}{\leq} \tau^m\frac{1}{\sqrt{2}\pi}\Gamma\left(\frac{m+1}{2}\right)\left(\frac{\pi}{\sqrt{2}}\|\beta-\beta^*\|_2\right)^{m+1}.
\]
In $(ii)$, we used that $\Gamma(m/2+1)/(m+1)\leq \Gamma((m+1)/2)$.
The right term is bounded with Lemma \ref{lem_unbounded_bound4bernstein_2}:
\[
P\frac{|(y-y^*)x^T(\gamma-\gamma^*)|^m}{2}
\]
\[
\leq 
\frac{\sigma}{4\pi} \Gamma\left(\frac{m+1}{2}\right)\left(\frac{1}{\sqrt{\pi}}(\sqrt{32}\tau\|\beta-\beta^*\|_2)^m+(\sqrt{8}\pi|\tau\beta^T\beta^*-\tau^*|\sigma)^m  \right).
\]
Combining these bounds we find:
\[
\frac{P|\tilde{u}(\gamma)|^m}{\Gamma\left(\frac{m+1}{2}\right)2^{m-1}}
\leq (\tau\|\beta-\beta^*\|_2)^m\left(\left(\frac{\pi}{\sqrt{2}}\right)^{m+1}\frac{\|\beta-\beta^*\|_2}{\sqrt{2}\pi}
+\sqrt{32}^m\frac{\sigma}{4\pi^{3/2}}
\right)
\]
\[
+\frac{\sigma}{4\pi}(\sqrt{8}\pi\sigma|\tau\beta^T\beta^*-\tau^*|)^m
\]
\[
\leq (\sqrt{32}\tau\|\beta-\beta^*\|_2)^m\left(\|\beta-\beta^*\|_2
+\frac{\sigma}{4\pi^{3/2}}\right)+\frac{\sigma}{4\pi}(\sqrt{8}\pi\sigma|\tau\beta^T\beta^*-\tau^*|)^m.
\]
The proof is complete.
\end{proof}

In the following, we assume that $\tau\gtrsim \tau^*$, which is the case if $\|\gamma-\gamma^*\|_2\lesssim \tau^*$.
This assumption allows us to express Lemma \ref{lem_bernstein_global} in terms of the $\|\cdot\|_*$-norm, giving us nice constants for Bernstein's condition.
Moreover, we use that $\tau^*\sigma\lesssim 1$, which holds true for $\sigma\lesssim 1$, see Lemma \ref{lem_sigmaVStau}.

\begin{prop}\label{prop_bernstein_constants}
Suppose that $\sigma\leq 1/\sqrt{2}$.
Let $\beta\in S^{p-1}$ and $\tau>0$ such that $|\tau-\tau^*|\leq \tau^*/6$.
Then, Bernstein's condition is met for $\tilde{u}(\tau\beta)-P\tilde{u}(\tau\beta)$ with constants:
\[
K:=2^8\sqrt{\tau^*}d_*(\tau\beta),
\quad \kappa:=2K\sqrt{\frac{d_*(\tau\beta)}{\sqrt{\tau^*}}+0.13\sigma}.
\]
\end{prop}

\begin{proof}
The result follows from Lemma \ref{lem_bernstein_global}.
Since $\tau\leq \tau^*7/6$,
\[
|\tau\beta^T\beta^*-\tau^*|
=\tau|\beta^T\beta^*-1|+|\tau-\tau^*|
\leq \frac{7}{6}\tau^*\frac{\|\beta-\beta^*\|_2^2}{2}+\tau^{*3/2}\frac{|\tau-\tau^*|}{\tau^{*3/2}}
\]
\[
\leq \frac{7}{6} \sqrt{\tau^*}\sqrt{\tau^*}\|\beta-\beta^*\|_2+\tau^{*3/2}\frac{|\tau-\tau^*|}{\tau^{*3/2}}
\stackrel{(i)}{\leq}
\sqrt{\tau^*}\sqrt{\frac{49}{36}+\tau^{*2}}d_*(\tau\beta)
\leq \sqrt{2\tau^{*3}}d_*(\tau\beta).
\]
Inequality $(i)$ follows from the Cauchy-Schwarz inequality.
In the last inequality we recall from Lemma \ref{lem_tau_sigma_derivatives} that $\sigma\leq 1/\sqrt{2}$ implies $\tau^*>\sqrt{6}$.
Now, using the upper bound $\tau\|\beta-\beta^*\|_2\leq 7/6\sqrt{\tau^*}d_*(\tau\beta)$ we can write the bound in Lemma \ref{lem_bernstein_global} as:
\[
\frac{P|\tilde{u}(\tau\beta)|^m}{\Gamma\left(\frac{m+1}{2}\right)2^{m-1}}
\]
\[
\leq (\sqrt{\tau^*}d_*(\tau\beta))^m
\left(
\left(\sqrt{32}\frac{7}{6}\right)^m\left(\|\beta-\beta^*\|_2
+\frac{\sigma}{4\sqrt{\pi^3}}\right)+\frac{\sigma}{4\pi}(4\pi\sigma\tau^{*})^m
\right)
\]
\[
\stackrel{(ii)}{\leq }
(\sqrt{\tau^*}d_*(\tau\beta))^m
\left(
\left(\sqrt{32}\frac{7}{6}\right)^m\left(\|\beta-\beta^*\|_2
+\frac{\sigma}{4\sqrt{\pi^3}}\right)+\frac{\sigma}{4\pi}\sqrt{32\pi^3}^m
\right)
\]
\[
\leq (\sqrt{32}\pi^{3/2}\sqrt{\tau^*}d_*(\tau\beta))^m
\left(\|\beta-\beta^*\|_2
+\sigma\left(\frac{1}{4\sqrt{\pi^3}}+\frac{1}{4\pi}\right)
\right).
\]
Finally, we use that $\|\beta-\beta^*\|_2\leq d_*(\tau\beta)/\sqrt{\tau^*}$ and simplify the constants $2\sqrt{32}\pi^{3/2}<2^6$ and $1/(4\sqrt{\pi^3})+1/(4\pi)< 0.13$.
To find the constants in Bernstein's condition, note that $\Gamma((m+1)/2)\leq m!/2$ for $m\geq 2$.
Recall that an extra factor of $2$ is necessary, as $\tilde{u}$ is not centered.
The proof is complete.
\end{proof}

\subsubsection{Controlling the empirical process}
All ingredients are prepared to prove an upper bound for the empirical process of the unbounded term, in the large noise case.
We will exploit Bernstein's inequality, or actually, the derived Lemma \ref{lem_bernstein_max}, using the constants from Proposition \ref{prop_bernstein_constants}, and making use of the discretization result (Lemma \ref{lem_covering_aux2}).
To exploit that the constants in Proposition \ref{prop_bernstein_constants} are better if $d_*(\gamma)$ is small, we start the proof with a local bound, and then move on to a global bound with peeling.
Recall the definition of the $\delta$-ball for the Euclidean distance $B_\delta:=\{\gamma\in\mathbb{R}^p:\|\gamma\|_2\leq r\}$, as well as for the $\|\cdot\|_*$-norm: $B_\delta^*:=\{\gamma\in\mathbb{R}^p\setminus\{0\}:d_*(\gamma)\leq \delta\}$.

\begin{prop}\label{prop_unbounded_global}
Suppose that $\sigma\leq 1/\sqrt{2}$.
Let $R>0$, $r\in(0,1)$, fix $t>0$ and define $t':=t+\log\lceil \log R/r\rceil$.
Then, with probability at least $1-3\exp(-t)$,
\[
\sup_{\gamma\in S_{r,R}}\frac{|(P_n-P)\tilde{u}(\gamma)|-\Delta(t')}{\psi_u(ed_*(\gamma),t')}\leq 1.
\]
Here, 
\[
S_{r,R}:=
\left\{
\gamma\in\mathbb{R}^p\setminus\{0\}:\quad r\leq d_*(\gamma)\leq R,\quad |\|\gamma\|_2-\tau^*|<\frac{\tau^*}{6}\right\},
\]
\[
\rho(t):=\frac{2(t+(p+1)\log(12\tau^*/r))}{n},\quad
\Delta(t):=\frac{7r}{6}\left(\sqrt{\frac{2t}{n}}+2\sqrt{p}\right),
\]
\[
\psi_u(\delta,t):=
2^8\sqrt{\tau^*}\delta\left(2\sqrt{\rho(t)\left(\frac{\delta}{\sqrt{\tau^*}}+0.13\sigma \right)}
+\rho(t)
\right).
\]
\end{prop}

\begin{proof}
Let $\delta>0$ be arbitrary.
We start with discretization.
Let $\mathcal{N}\subset S_{r,\delta}$ be such that $\mathcal{U}(\mathcal{N})/U_{\tau^*7/6}$ is a covering of $\mathcal{U}(S_{r,\delta})/U_{\tau^*7/6}$ in the supremum norm of radius $r/\tau^*$.
By Lemma \ref{lem_covering_aux2}, $\mathcal{N}$ can be chosen so that it has at most $(12\tau^*/r)^{p+1}$ elements.
We apply Lemma \ref{lem_unbounded_discretize} with radius $\tau^*7/6$ and approximation error $r/\tau^*$.
On an event with probability at least $1-\exp(-t)$,
\[
\sup_{\gamma\in S_{r,\delta}}|(P_n-P)\tilde{u}(\gamma)|\leq \sup_{\gamma\in \mathcal{N}}|(P_n-P)\tilde{u}(\gamma)|+\Delta(t).
\]
By Proposition \ref{prop_bernstein_constants}, $\{\tilde{u}(\tau\beta)-P\tilde{u}(\tau\beta):\beta\in S_{r,\delta}\}$ satisfies Bernstein's condition with constants:
\[
K:=2^8\sqrt{\tau^*}\delta,
\quad \kappa:=2K\sqrt{\frac{\delta}{\sqrt{\tau^*}}+0.13\sigma}.
\]
Therefore, by Lemma \ref{lem_bernstein_max}, on an event with probability at least $1-2\exp(-t)$,
\[
\sup_{\gamma\in \mathcal{N}}|(P_n-P)\tilde{u}(\gamma)|
\leq 2^8\sqrt{\tau^*}\delta\left(2\sqrt{\rho(t)\left(\frac{\delta}{\sqrt{\tau^*}}+0.13\sigma \right)}
+\rho(t)
\right)
=\psi_u(\delta,t),
\]
where $\rho$ and $\psi_u$ are defined above.
We conclude that for any $\delta>0$,
\[
\mathbb{P}\left[
\bigcup_{\gamma\in S_{r,\delta}}\frac{|(P_n-P)\tilde{u}(\gamma)|-\Delta(t)}{\psi_u(\delta,t)}>1
\right]\leq 3\exp(-t).
\]
Now we use peeling, to replace the dependency on $\delta$ with a dependency on $d_*(\gamma)$.
Define $J:=\lceil\log R/r\rceil$ and $t':=t+\log J$.
Then:
\[
\mathbb{P}\left[
\bigcup_{\gamma\in S_{r,R}}\frac{|(P_n-P)\tilde{u}(\gamma)|-\Delta(t')}{\psi_u(ed_*(\gamma),t')}>1
\right]
\]
\[
\leq \sum_{j=0}^{J-1}\mathbb{P}\left[
\bigcup_{\gamma\in S_{re^j,re^{j+1}}}\frac{|(P_n-P)\tilde{u}(\gamma)|-\Delta(t')}{\psi_u(ed_*(\gamma),t')}>1
\right]
\]
\[
\leq \sum_{j=0}^{J-1}\mathbb{P}\left[
\bigcup_{\gamma\in S_{re^j,re^{j+1}}}\frac{|(P_n-P)\tilde{u}(\gamma)|-\Delta(t')}{\psi_u(re^{j+1},t')}>1
\right]
\]
\[
\leq J3\exp(-t-\log J)=3\exp(-t).
\]
The proof is complete.
\end{proof}

\subsection{The small noise case}
We provide an upper bound for the unbounded term, in the case where the noise is small. 
Compared to the large noise case, we do not have any upper bound for $\tau^*$.
On the other hand, we can exploit that $\sigma$ is very small.
Moreover, the proof for the small noise case only requires that we control $\|\beta-\beta^*\|_2$, which simplifies the expressions considerably.
As before, we first verify Bernstein's condition.
In this subsection, we use the notation: 
\[\tilde{u}(\beta):=u(\beta)-u(\beta^*).\]

\subsubsection{Local verification of Bernstein's condition}
Similarly to the large noise case, we verify Bernstein's condition.
The verification is somewhat more straightforward, as we can compare $\beta$ on the sphere with $\beta^*$ directly.

\begin{lem}\label{lem_small_bernstein_condition}
Let $\beta\in S^{p-1}$.
Bernstein's condition is met for $\tilde{u}(\beta)-P\tilde{u}(\beta)$ with constants:
\[
K:=\left(\frac{\pi}{\sqrt{2}}\right)\left(\|\beta-\beta^*\|_2+2\sigma\right),\quad
\kappa:=K\sqrt{\|\beta-\beta^*\|_2+2\sigma}.
\]
\end{lem}

\begin{proof}
It holds that:
\[
P||x^T\beta^*|1\{yx^T\beta^*<0\}-|x^T\beta|1\{yx^T\beta<0\}|^m
\]
\[
\leq P|x^T\beta^*|^m1\{yx^T\beta^*<0\}+|x^T\beta|^m1\{yx^T\beta<0\}.
\]
We first consider the summand with $\beta$.
To control it, we use Corollary \ref{cor_trig}.
If $\tilde{x}\sim\mathcal{N}(0,I_{p+1})$:
\[
P|x^T\beta|^m1\{yx^T\beta<0\}
=P|x^T\beta|^m1\{(\beta,0)^T \tilde{x}\tilde{ x}^T(\beta^*,\sigma)<0\}
\]
\[
\stackrel{(i)}{\leq} \frac{1}{\sqrt{2}\pi}\frac{\Gamma(m/2+1)}{m+1}\left(\frac{\pi}{\sqrt{2}}\left\|(\beta,0)-\frac{1}{\sqrt{1+\sigma^2}}(\beta^*,\sigma)\right\|_2\right)^{m+1}
\]
\[
\stackrel{(ii)}{\leq} \frac{1}{2}\frac{\Gamma(m/2+1)}{m+1}\left(\frac{\pi}{\sqrt{2}}\right)^{m}\left(\|\beta-\beta^*\|_2+\sigma\right)^{m+1}
\]
\[
\stackrel{(iii)}{\leq} \frac{m!}{2}\left(\frac{\pi}{\sqrt{8}}\right)^{m}\left(\|\beta-\beta^*\|_2+\sigma\right)^{m+1}.
\]
In $(i)$, we used Corollary \ref{cor_trig}.
In $(ii)$, we used the triangular inequality and that $1-\frac{1}{\sqrt{1+\sigma^2}}\leq \frac{\sigma^2}{2}$.
In $(iii)$, we used that or $m\geq 2$ it holds that $\Gamma(m/2+1)/(m+1)\leq m!/2^m$.
We can apply the same reasoning to the term $P|x^T\beta^*|^m1\{yx^T\beta^*<0\}$ by choosing $\beta$ as $\beta^*$.
Exploiting the fact that integer monomials are superadditive over positive scalars, we find:
\[
P|\tilde{u}(\beta)|^m\leq \frac{m!}{2}\left(\frac{\pi}{\sqrt{8}}\right)^{m}\left(\|\beta-\beta^*\|_2+2\sigma\right)^{m+1}.
\]
This gives constants for Bernstein's condition.
Recall the extra factor 2, since $\tilde{u}$ are not necessarily centered.
The proof is complete.
\end{proof}

\subsubsection{Controlling the empirical process}
Here, we prove a high probability upper bound for the empirical process of the unbounded term, in the small noise case.
The procedure is analogous to the large noise case.
We recall the notation $S(\delta,\beta'):=\{\beta\in S^{p-1}:\|\beta-\beta'\|_2\leq \delta\}$ for $\delta>0$ and $\beta'\in S^{p-1}$.

\begin{prop}\label{prop_small_unbounded_global}
Let $r, R>0$.
Fix any $t>0$ and let $t':=t+\log \lceil\log R/r\rceil$.
With probability at least $1-\exp(-t)$,
\[
\sup_{\beta\in S(R,\beta^*)\setminus S(r,\beta^*)}
\frac{|(P_n-P)\tilde{u}(\beta)|-\Delta(t')}{\frac{\pi}{\sqrt{2}}\xi(\beta)(\sqrt{\xi(\beta)\rho(t')}+\rho(t'))}\leq 1.
\]
Here,
\[
\Delta(t):=r\left(\sqrt{\frac{2t}{n}}+2\sqrt{p}\right),\quad
\rho(t):=\frac{t+(p+1)\log\left(\frac{12}{r}\right)}{n},
\]
\[
\xi(\beta):=e\|\beta-\beta^*\|_2+2\sigma.
\]
\end{prop}

\begin{proof}
Let $\delta>0$ be arbitrary.
We start with discretization.
By Lemma \ref{lem_covering_aux2}, it takes at most $(12/r)^{p+1}$ balls of radius $r$ to cover $\{\tilde{u}(\beta)/U(x,1):\beta\in S^{p-1}\}$ with elements in itself in the supremum norm.
Let $\mathcal{N}\subset S(\delta,\beta^*)\setminus S(r,\beta^*)$ be such a covering of $S(\delta,\beta^*)\setminus S(r,\beta^*)$.
We use Lemma \ref{lem_unbounded_discretize} with bound $1$ and covering radius $\varepsilon:=r$.
On an event with probability at least $1-\exp(-t)$, 
\[
\sup_{\beta\in S(\delta,\beta')\setminus S(r,\beta^*)}|(P_n-P)\tilde{u}(\beta)|\leq \sup_{\beta\in\mathcal{N}}|(P_n-P)\tilde{u}(\beta)|
+\Delta(t).
\]
By Lemma \ref{lem_small_bernstein_condition}, the centered $\tilde{u}$ with parameters in $S(\delta,\beta')\setminus S(r,\beta^*)$ satisfy Bernstein's condition, with parameters $K:=\frac{\pi}{\sqrt{2}}\left(\delta+2\sigma\right)$ and $\kappa:=K\sqrt{\delta+2\sigma}$.
By Lemma \ref{lem_bernstein_max}, with probability at least $1-2\exp(-t)$,
\[
\sup_{\beta\in\mathcal{N}}\frac{|(P_n-P)\tilde{u}(\beta)|}{\frac{\pi}{\sqrt{2}}(\delta+2\sigma)}\leq \sqrt{(\delta+2\sigma)\rho(t)}+\rho(t).
\]
We use peeling to replace the dependency on $\delta$ with dependency on $\|\beta-\beta^*\|_2$.
Define $J:=\lceil\log R/r\rceil$ and $t':=t+\log J$.
\[
\mathbb{P}\left[
\bigcup_{\beta\in S(R,\beta^*)\setminus S(r,\beta^*)}
\frac{|(P_n-P)\tilde{u}(\beta)|-\Delta(t')}{\frac{\pi}{\sqrt{2}}\xi(\beta)(\sqrt{\xi(\beta)\rho(t')}+\rho(t'))}>1
\right]
\]
\[
\stackrel{(i)}{\leq} J3\exp(-t-\log J)=3\exp(-t).
\]
As in the large noise proof, $(i)$ follows from peeling, i.e. a union bound over a partition of the form $S(e^{j+1}r,\beta^*)\setminus S(e^{j}r,\beta^*)$.
The proof is complete.
\end{proof}

\newpage
\section{Proofs of the main theorems}\label{sec_state}

\subsection{Large noise case}
In this subsection, we prove the main theorem for the case where $\sigma$ is large.
As a final auxiliary step, we prove the margin condition.
It provides a lower bound for the excess risk $Pl(\gamma)-Pl(\gamma^*)$ in terms of $d_*(\gamma)$.
Once this result is established, we are ready to prove the main result.

\subsubsection{The margin condition}
Recall our notation $\gamma\in\mathbb{R}^p\setminus\{0\}$ and $\tau:=\|\gamma\|_2$, as well as $\beta:=\gamma/\tau$.
To derive bounds on $\|\hat{\beta}-\beta^*\|_2$ and $|\hat{\tau}-\tau^*|$ from our bounds on the excess risk, we would like to have an inequality of the following type, for some $\alpha_1,\alpha_2>0$:
\[
\alpha_1|\hat{\tau}-\tau^*|+\alpha_2\|\hat{\beta}-\beta^*\|_2\leq  Pl(\gamma)-Pl(\gamma^*).
\]
This is what we call the \textit{margin condition}.
It was originally introduced in \cite{mammen1999smooth}.
The statement we use is reminiscent of the version in \cite[Chapter 6.4]{buhlmann2011statistics}.
In the following, we show that indeed the margin condition is satisfied, with:
\[
d_*(\gamma)^2:=\|(\tau,\beta)-(\tau^*,\beta^*)\|_*^2\lesssim Pl(\gamma)-Pl(\gamma^*).
\]
This result is the motivation for our introduction of the $*$-norm (see Section \ref{sub_star_norm}).

\paragraph{Re-parametrization}
As the signal strength $1/\sigma\sim \tau^*$ grows, it becomes more difficult to estimate $\tau^*$ but easier to estimate $\beta^*$.
A lower bound on the excess risk must reflect this.
Hence, in our attempt to derive such a bound, we re-parametrize $\gamma$ into its length $\tau:=\|\gamma\|_2$ and its orientation $\beta:=\gamma/\|\gamma\|_2$.
We will not work with $\beta$ directly but rather exploit that $\beta$ lives in a $p-1$-dimensional manifold.
To do so, we re-parametrize $\beta$ in the following.

We find an orthonormal basis, which includes $\beta^*$ as one of its elements.
For any $\beta^*\in S^{p-1}$, there exists a set $\{b_1,\ldots,b_{p-1}\}\subset S^{p-1}$, such that the collection $\{\beta^*,b_1,\ldots,b_{p-1}\}$ forms an orthonormal basis.
So, we can find the Fourier expansion of any vector $v\in\mathbb{R}^p$ as:
\[
v=v^T\beta^* \beta^{*}+\sum_{i=1}^{p-1} v^T b_i b_i.
\]
We write this more compactly for vectors on $S^{p-1}$.
Let $V$ be the $p\times (p-1)$ matrix, whose columns are the vectors $\{b_1,\ldots,b_{p-1}\}$.
For any $v\in\mathbb{R}^p$, we define the coefficient vector $h(v):=( v^T b_1,\ldots,  v^Tb_{p-1})$.
So, for any $\beta\in S^{p-1}$,
\begin{equation}\label{eq_expansion_h}
\beta = sign(\beta^T\beta^*)\sqrt{1-\|h(\beta)\|_2^2}\beta^*+Vh(\beta).
\end{equation}

The following result shows that the Euclidean norm of $h(\beta)$ behaves as the distance between $\beta$ and $\beta^*$, up to a constant factor.

\begin{lem}\label{lem_h_inequality}
For any $\beta^*,\beta\in S^{p-1}$ with $\beta^T\beta^*\geq 0$, it holds that:
\[
\|h(\beta)\|_2\leq \|\beta-\beta^*\|_2\leq \sqrt{2}\|h(\beta)\|_2.
\]
\end{lem}

\begin{proof}
Fix $\beta\in S^{p-1}$.
To reduce notation, we write $h:=h(\beta)$.
By definition,
\[
\|\beta-\beta^*\|_2^2= \left(1-\sqrt{1-\|h\|_2^2}\right)^2+\|h\|_2^2
=2\left(1-\sqrt{1-\|h\|_2^2}\right).
\]
From this, the lower bound already follows.
For the upper bound, multiply out the first term, simplify, and then use that $1-\sqrt{1-\|h\|_2^2}=\|h\|_2^2/(1+\sqrt{1-\|h\|_2^2})\leq \|h\|_2^2$.
This completes the proof.
\end{proof}

We now re-write the true risk (the expectation of the loss function) at a point $\gamma\in\mathbb{R}^p\setminus\{0\}$.
For $\gamma\in\mathbb{R}^p$ such that $\gamma^T\gamma^*\geq 0$, we write the `true risk' as:
\[
P\log(1+\exp(-yx^T\gamma))
\]
\[
=P\log\left(1+\exp\left(-\tau yx^T\left(\sqrt{1-\|h(\beta)\|_2^2}\beta^*+Vh(\beta)\right)\right)\right).
\]
The expansion \eqref{eq_expansion_h} shows that on the hemisphere $H:=\{b\in S^{p-1}:b^T\beta^*> 0\}$, the map $\beta\mapsto h(\beta)$ is injective.
So, we can reparametrize $(\tau,\beta)\in(0,\infty)\times H$ with $(\tau,h)\in(0,\infty)\times B_1^\circ$, where $B_1^\circ$ is the open Euclidean unit ball, and write the true risk as:
\begin{equation}\label{eq_defn_risk}
P\log\left(1+\exp\left(-\tau yx^T\left(\sqrt{1-\|h\|_2^2}\beta^*+Vh\right)\right)\right)
=:R(\tau,h).
\end{equation}
The mapping $R:(0,\infty)\times B_1^\circ\rightarrow \mathbb{R}$ is twice continuously differentiable.
In the following, we will study its Hessian.

\paragraph{The Hessian of the re-parametrized risk}

\begin{lem}\label{lem_hessian}
Let $z\sim\mathcal{N}(0,1)$.
For any $\tau>0$ and any $h\in B_1^\circ$, we have:
\[
\ddot{R}(\tau,h):=\begin{pmatrix}
P\frac{|z|^2}{(\exp(\tau|z|/2)+\exp(-\tau|z|/2))^2},&-\frac{1}{\sqrt{2\pi(1+\sigma^2)}\sqrt{1-\|h\|_2^2}}h^T\\
-\frac{1}{\sqrt{2\pi(1+\sigma^2)}\sqrt{1-\|h\|_2^2}}h,&\frac{\tau}{\sqrt{2\pi(1+\sigma^2)}\sqrt{1-\|h\|_2^2}}(I+\frac{hh^T}{1-\|h\|_2^2})
\end{pmatrix}.
\]
\end{lem}

\begin{proof}
Fix $\tau>0$ and $h\in B_1^\circ$.
Let $\beta\in S^{p-1}$ be the unique vector such that $\beta^T\beta^*>0$ and the identity \eqref{eq_expansion_h} holds.
Let $\gamma:=\tau\beta$.
We use the decomposition $\log(1+\exp(-yx^T\gamma))=\log(1+\exp(-|\gamma^Tx|))+|x^T\gamma|1\{yx^T\gamma<0\}$.
Using Corollary \ref{cor_trig_y} for the right summand, we find:
\[
R(\tau,h)=P\log(1+\exp(-\tau|z|))+\tau \frac{1}{\sqrt{2\pi}}\left(1-\frac{\beta^T\beta^*}{\sqrt{1+\sigma^2}}\right).
\]
By construction of $\beta$, it holds that $\beta^T\beta^*=\sqrt{1-\|h\|_2^2}$.
Differentiating twice, we find the expression stated in the claim.
The proof is complete.
\end{proof}

We now give a lower bound on the second-order term of the Taylor expansion of $R$ near $\gamma^*$, i.e. $(\tau^*,0)$.
At this moment, the $*$-norm appears in our calculation.

\begin{lem}\label{lem_hessian_bound}
Let $\tau,\tau^*\geq \sqrt{6+\sqrt{51}}$ and $\kappa_\tau>0$ such that $|\tau-\tau^*|\leq \kappa_\tau\tau^*$.
Let $(\bar{\tau},\bar{h})$ be a convex combination of $(\tau,h)$ and $(\tau^*,0)$.
Then,
\[
\begin{pmatrix}
\tau-\tau^*\\ h
\end{pmatrix}^T \ddot{R}(\bar{\tau},\bar{h})\begin{pmatrix}
\tau-\tau^*\\ h
\end{pmatrix}\geq \frac{1}{\sqrt{8\pi}}\left(\frac{1}{(1+\kappa_\tau)^3}\wedge\frac{1-3\kappa_\tau}{\sqrt{1+\sigma^2}} \right)d_*(\tau\beta(h))^2.
\]
\end{lem}

\begin{proof}
We make use of the expression given in Lemma \ref{lem_hessian}.
We bound each term separately.
First, by Lemma \ref{lem_f_dotdot},
\[
P\frac{|z|^2}{(e^{\bar{\tau}|z|/2}+e^{-\bar{\tau}|z|/2})^2}(\tau-\tau^*)^2\geq \sqrt{\frac{1}{8\pi}}\frac{(\tau-\tau^*)^2}{\bar{\tau}^3}\geq \sqrt{\frac{1}{8\pi}}\frac{(\tau-\tau^*)^2}{(\tau^{*}(1+\kappa_\tau))^3}.
\]
For the cross terms,
\[
h^T\ddot{R}((\bar{\tau},\bar{h}))_{21}(\tau-\tau^*)=-\frac{\tau-\tau^*}{\sqrt{2\pi(1+\sigma^2)}}\frac{\bar{h}^Th}{\sqrt{1-\|\bar{h}\|_2^2}}
\]
\[
\stackrel{(i)}{\geq} -\frac{|\tau-\tau^*|}{\sqrt{2\pi(1+\sigma^2)}}\frac{\|h\|_2^2}{\sqrt{1-\|\bar{h}\|_2^2}}
\geq  -\frac{\kappa_\tau\tau^*}{\sqrt{2\pi(1+\sigma^2)}}\frac{\|h\|_2^2}{\sqrt{1-\|\bar{h}\|_2^2}}.
\]
In $(i)$, we used that $\bar{h}$ is an intermediate point between $h$ and $0$, $\|\bar{h}\|_2\leq \|h\|_2$.
The inequality then follows from Cauchy-Schwarz.
The last product gives:
\[
h^T\ddot{R}(\bar{\theta})_{22}h
\geq 
\frac{\bar{\tau}}{\sqrt{2\pi(1+\sigma^2)}}\frac{\|h\|_2^2}{\sqrt{1-\|\bar{h}\|_2^2}}
\]
\[
\geq \frac{\tau^*(1-\kappa_\tau)}{\sqrt{2\pi(1+\sigma^2)}}\frac{\|h\|_2^2}{\sqrt{1-\|\bar{h}\|_2^2}}.
\]
Together with the two cross terms, this is:
\[
\frac{\tau^*(1-3\kappa_\tau)}{\sqrt{2\pi(1+\sigma^2)}}\frac{\|h\|_2^2}{\sqrt{1-\|\bar{h}\|_2^2}}
\geq \frac{\tau^*(1-3\kappa_\tau)}{\sqrt{2\pi(1+\sigma^2)}}\|h\|_2^2.
\]
By Lemma \ref{lem_h_inequality}, $\|h\|_2^2\geq \|\beta-\beta^*\|_2^2/2$.
Combining these lower bounds, we find:
\[
\frac{1}{\sqrt{8\pi}}\left(\frac{1}{(1+\kappa_\tau)^3}\frac{(\tau-\tau^*)^3}{\tau^{*3}}
+\frac{1-3\kappa_\tau}{\sqrt{1+\sigma^2}}\tau^*\|\beta-\beta^*\|_2^2\right)
\]
\[
\geq \frac{1}{\sqrt{8\pi}}\left(\frac{1}{(1+\kappa_\tau)^3}\wedge\frac{1-3\kappa_\tau}{\sqrt{1+\sigma^2}} \right)d_*(\tau\beta)^2.
\]
The proof is complete.
\end{proof}

\paragraph{Proof of the margin condition}

\begin{lem}\label{lem_margin}
Let $\tau,\tau^*\geq \sqrt{6+\sqrt{51}}$ and $\kappa_\tau>0$ such that $|\tau-\tau^*|\leq \kappa_\tau\tau^*$.
Suppose that $\beta,\beta^*\in S^{p-1}$ such that $\beta^T\beta^*>0$.
Then,
\[
Pl(\tau\beta)-l(\tau^*\beta^*)\geq \frac{1}{\sqrt{32\pi}}\left(\frac{1}{(1+\kappa_\tau)^3}\wedge\frac{1-3\kappa_\tau}{\sqrt{1+\sigma^2}} \right)d_*(\tau\beta)^2.
\]
\end{lem}

\begin{proof}
Using the re-parametrization, we have $Pl(\tau\beta)-l(\tau^*\beta^*)=R(\tau,h(\beta))-R(\tau^*,0)$.
By Taylor's Theorem, there exists a convex combination of $(\tau,h(\beta))$ and $(\tau,h(\beta))$, say $(\bar{\tau},\bar{h})$, such that:
\[
R(\tau,h(\beta))-R(\tau^*,0)\geq
\dot{R}(\tau^*,0)^T\begin{pmatrix}\tau-\tau^*\\h\end{pmatrix}
+\frac{1}{2}\begin{pmatrix}\tau-\tau^*\\ h\end{pmatrix}^T\ddot{R}(\bar{\tau},\bar{h})\begin{pmatrix}\tau-\tau^*\\h\end{pmatrix}.
\]
By convexity and the definition of $\tau^*\beta^*$, $\dot{R}(\tau^*,0)=0$.
The result now follows from Lemma \ref{lem_hessian_bound}.
\end{proof}

\subsubsection{Proof of large noise theorem}
Here, we prove Theorem \ref{thm_cor_large}.
In fact, we prove Theorem \ref{thm_large}, which is a slightly more general version, as it allows for a weaker restriction on the choice of $M$.
Moreover, it is marginally less restrictive on $\sigma$, as with Lemma \ref{lem_tau_sigma_derivatives} one can verify that $\sigma\leq 1/\sqrt{6}$ implies $\tau^*>\frac{6}{5}\sqrt{6+\sqrt{51}}$.

\begin{thm}\label{thm_large}
Let $\hat{\gamma}$ be a solution to \eqref{eq_opt}.
There exists universal constants $C>0$ and $0<c<1$, such that for any $t>0$, if:
\[
\frac{6}{5}\sqrt{6+\sqrt{51}}<\tau^*\leq c\frac{n}{p\log n+t}\leq M,
\]
then with probability at least $1-4\exp(-t)$,
\[
d_*(\hat{\gamma})\leq  C\sqrt{\frac{p\log n+t}{n}}.
\]
\end{thm}

We refer the reader to Section \ref{sec_main} for a sketch of the proof.

\begin{proof}
We define:
\[
\alpha:=\frac{1}{1+\frac{6\|\hat{\gamma}-\gamma^*\|_2}{\tau^*}},\quad
\tilde{\gamma}:=\alpha\hat{\gamma}+(1-\alpha)\gamma^*
\]
We show that we can apply Lemma \ref{lem_large_aux} to $\tilde{\gamma}$.
By construction of $\tilde{\gamma}$,
\[
\|\tilde{\gamma}-\gamma^*\|_2=\alpha\|\hat{\gamma}-\gamma^*\|_2\leq \frac{\tau^*}{6}.
\]
By convexity, $P_nl(\tilde{\gamma})\leq P_nl(\gamma^*)$ holds on an event with probability 1.
So, by Lemma \ref{lem_large_aux}, there exists a constant $C'>0$, such that on an event $B_t$ with probability at least $1-4\exp(-t)$, 
\[
d_*(\tilde{\gamma})\leq  C'\sqrt{\frac{p\log(n)+t}{n}}.
\]
By Lemma \ref{lem_proxVSeuclid}, it follows that:
\[
\|\tilde{\gamma}-\gamma^*\|_2\leq \sqrt{2\tau^{*3}}C'\sqrt{\frac{p\log(n)+t}{n}}\leq c\cdot C'\sqrt{2}\cdot \tau^{*}
\stackrel{(i)}{\leq}\frac{\tau^*}{12}.
\]
The inequality $(i)$ holds if $c\leq 1/(C'12\sqrt{2})$, which we require.
If so, after rearranging we find:
\[
\|\hat{\gamma}-\gamma^*\|_2=\frac{\|\tilde{\gamma}-\gamma^*\|_2}{1-\frac{6\|\tilde{\gamma}-\gamma^*\|_2}{\tau^*}}
\leq 2\|\tilde{\gamma}-\gamma^*\|_2.
\]
Consequentially, $\|\hat{\gamma}-\gamma^*\|_2\leq \tau^*/6$.
So, by Lemma \ref{lem_large_aux} on the event $B_t$, 
\[
d_*(\hat{\gamma})\leq  C'\sqrt{\frac{p\log(n)+t}{n}}.
\]
The proof is complete.
\end{proof}

\begin{lem}\label{lem_large_aux}
Let $\mu:=\frac{p\log(n)+t}{n}$ and suppose that:
\[
\frac{6}{5}\sqrt{6+\sqrt{51}}\leq \tau^*\leq \frac{1}{\mu}.
\]
There exists a universal constant $C>0$ and for any $t>0$ an event $B_t$ with probability at least $1-4\exp(-t)$, such that the following is true.
Let $\tilde{\gamma}$ be any random vector in $\mathbb{R}^p$ which satisfies $\|\tilde{\gamma}-\gamma^*\|_2\leq \tau^*/6$ and $P_n l(\tilde{\gamma})\leq P_nl(\gamma^*)$ on an event $A$.
Then, on $A\cap B_t$, it holds that:
\[
d_*(\tilde{\gamma})\leq C\sqrt{\mu}.
\]

\end{lem}

\begin{proof}
We will apply the margin condition (Lemma \ref{lem_margin}) with $\kappa_\tau=1/6$, which gives us:
\[
\frac{1}{\sqrt{2^7\pi}}\frac{d_*(\tilde{\gamma})^2}{\sqrt{1+\sigma^2}}
\leq Pl(\tilde{\gamma})-l(\gamma^*)\leq (P-P_n)l(\tilde{\gamma})-l(\gamma^*)
\]
\[
=(P-P_n)\log(2)\tilde{b}(\tilde{\gamma})+\tilde{u}(\tilde{\gamma}).
\]
In the following, we distinguish whether the bounded or unbounded term dominates.
Suppose that the bounded term dominates.
If so, we apply Proposition \ref{prop_bounded_global}.
Henceforth, we borrow its notation.
Since $\|\tilde{\gamma}-\gamma^*\|_2\leq \tau^*/6$, we have $\tilde{\tau}:=\|\tilde{\gamma}\|_2\geq \tau^*5/6$ and by Lemma \ref{lem_proxVSeuclid}  $d_*(\tilde{\gamma})\leq \sqrt{\tau^{*}}/2$.
So, we may choose the bigger radius in Proposition \ref{prop_bounded_global} as $R:=\sqrt{\tau^{*}}/2$. Moreover, we choose $r:=\mu$.
By Proposition \ref{prop_bounded_global}, there exists a universal constant $C_b$ such that on an event $B_b$, which has probability at least $1-\exp(-t)$,
\[
(P-P_n)\tilde{b}(\tilde{\gamma})\lesssim \mu+\sqrt{\mu}d_*(\tilde{\gamma}).
\]
Distinguishing whether the first or second term dominates, we obtain $d_*(\tilde{\gamma})\lesssim \mu$.
Here, we used that $\sigma\lesssim 1$ by Lemma \ref{lem_sigmaVStau} and since $\tau^*\gtrsim 1$.

Now suppose instead that the unbounded term dominates.
We apply Proposition \ref{prop_unbounded_global}, borrowing its notation.
The radii are chosen as $r:=\mu/\sqrt{p}$ and $R:=\sqrt{\tau^{*}}/2$. 
We see that $\Delta(t')\lesssim \mu$ and $\rho(t')\lesssim \mu$, as well as:
\[
\psi_u(d_*(\tilde{\gamma}),t')\lesssim
\sqrt{\tau^*}d_*(\tilde{\gamma})\left(\sqrt{\left(\frac{d_*(\tilde{\gamma})}{\sqrt{\tau^*}}+\sigma \right)\mu}
+\mu\right).
\]
So, by Proposition \ref{prop_unbounded_global}, on an event $B_u$ with probability at least $1-3\exp(-t)$,
\[
(P-P_n)\tilde{u}(\tilde{\gamma})
\lesssim d_*(\tilde{\gamma})^{3/2}\sqrt{\mu}\tau^{*1/4}+d_*(\tilde{\gamma})\sqrt{\mu\tau^*\sigma}+d_*(\tilde{\gamma})\mu\sqrt{\tau^*}+\mu
\]
\[
\lesssim d_*(\tilde{\gamma})^{3/2}\mu^{1/4}+d_*(\tilde{\gamma})\sqrt{\mu}+\mu.
\]
Here, we used Lemma \ref{lem_sigmaVStau} to simplify $\sigma\tau^*\lesssim1$.
Whichever term dominates, we obtain $d_*(\tilde{\gamma})\lesssim\sqrt{\mu}$.

If we denote the event $B_b\cap B_u$ by $B_t$, the event $B_t$ has probability at least $1-4\exp(-t)$, and on $A\cap B_t$, the event $d_*(\tilde{\gamma})\lesssim\mu$ holds.
The proof is complete.
\end{proof}

\newpage
\subsection{Proof for small noise}
In this subsection, we prove Theorem \ref{thm_cor_small}, the main result for the case where the noise is small.
In fact, we prove the slightly stronger Theorem \ref{thm_small}, which neither requires that $M\geq n/(p\log n+t)$ nor that $\sigma\leq (p\log n + t)/n$.
Let $\hat{\gamma}$ be a solution to \eqref{eq_opt}.
Define $\hat{\tau}:=\|\hat{\gamma}\|_2$ and $\hat{\beta}:=\hat{\gamma}/\hat{\tau}$.
It can be verified that almost surely $\hat{\tau}\neq 0$.
Consequentially, $\hat{\beta}$ is well-defined.

\begin{thm}\label{thm_small}
Fix $t>0$ and let $\mu:=(p\log(n)+t)/n$.
Assume $\sigma\leq 1/\sqrt{7}$ and require $M> 2\sqrt{6}$ and:
\begin{equation}\label{eq_small_assumption_n_easy}
0.051>\sqrt{\frac{9t+3p(2+\log(3\sqrt{6n}))}{n}}.
\end{equation}
Then, with probability at least $1-6\exp(-t)$,
 \begin{equation}\label{eq_small_goal}
\|\hat{\beta}-\beta^*\|_2\lesssim\frac{1}{M}\vee \sigma\vee\mu,\end{equation}
and:
\begin{equation}\label{eq_small_goal_tau}
M\wedge \frac{1}{\sigma}\wedge\frac{1}{\mu}\lesssim \hat{\tau}.
\end{equation}
\end{thm}

As a direct consequence of Theorem \ref{thm_small}, we get Theorem \ref{thm_cor_small}, which was stated in Section 2.
To recover its statement, we require $\sigma\leq \mu$ and $M\geq 1/\mu$.
We point the reader to Section \ref{sec_main} for a sketch of the proof.

\begin{proof}[Proof of Theorem \ref{thm_small}]
The proof is divided into 3 cases.
Case 3 describes $\hat{\tau}\leq \sqrt{6}$, which we show happens only with small probability.
Cases 1-2 capture $\hat{\tau}> \sqrt{6}$.
In case 1, we additionally assume $M/2< \hat{\tau}\leq M$, whereas in case 2 we assume $\sqrt{6}<\hat{\tau}\leq M/2$.
On these events, we show that with large probability, \eqref{eq_small_goal} and \eqref{eq_small_goal_tau} hold.

\paragraph{Case 1: $M/2\vee \sqrt{6}\leq \hat{\tau}\leq M$.}
In case 1, \eqref{eq_small_goal_tau} follows directly by assumption, so we only need to prove \eqref{eq_small_goal}.
By the lower bound in Lemma \ref{lem_moment_bounded_difference}, using that $\hat{\tau}\geq \sqrt{6}$,
\[
\frac{1}{8\sqrt{2\pi}}\frac{1}{\hat{\tau}}\leq P\left(b(\hat{\tau}\beta^*)-b(2\hat{\tau}\beta^*)\right)
\stackrel{(i)}{=}P\left(b(\hat{\tau}\hat{\beta})-b(2\hat{\tau}\beta^*)\right)
\]
\[
\stackrel{(ii)}{\leq} P\left(b(\hat{\tau}\hat{\beta})-b(\hat{\tau}\beta^*)\right)+\frac{1}{\hat{\tau}}\sqrt{\frac{2}{\pi}}
\stackrel{(iii)}{\leq} P\left(b(\hat{\tau}\hat{\beta})-b(\hat{\tau}\beta^*)\right)+\frac{1}{M}\sqrt{\frac{8}{\pi}}.
\]
The equality in $(i)$ follows from rotational invariance, whereas the inequality $(ii)$ follows from the upper bound in Lemma \ref{lem_moment_bounded_difference}.
Inequality $(iii)$ follows from the assumption $M/2\leq \hat{\tau}$.
It turns out that this is the only step where the assumption $\sqrt{6}\leq \hat{\tau}\leq M$ would not be enough in the proof of part 1.
On the other hand, Corollary \ref{cor_trig_y} gives:
\[
P\left(u(\hat{\tau}\hat{\beta})-u(\hat{\tau}\beta^*)\right)=\frac{\hat{\tau}}{\sqrt{8\pi(1+\sigma^2)}}\|\hat{\beta}-\beta^*\|_2^2.
\]
Together, this gives:
\[
\frac{1}{8\sqrt{2\pi}}\frac{1}{\hat{\tau}}+\frac{\hat{\tau}}{\sqrt{8\pi(1+\sigma^2)}}\|\hat{\beta}-\beta^*\|_2^2
\leq P\left(b(\hat{\tau}\hat{\beta})-b(\hat{\tau}\beta^*)+u(\hat{\tau}\hat{\beta})-u(\hat{\tau}\beta^*)\right)+\frac{1}{M}\sqrt{\frac{8}{\pi}}.
\]
If the term $1/M$ dominates the right-hand side, using the inequality of arithmetic and geometric means we get:
\[
\frac{1}{\sqrt{8\pi}}
\frac{\|\hat{\beta}-\beta^*\|_2}{(1+\sigma^2)^{1/4}}
\leq \frac{1}{8\sqrt{2\pi}}\frac{1}{\hat{\tau}}+\frac{\hat{\tau}}{\sqrt{8\pi(1+\sigma^2)}}\|\hat{\beta}-\beta^*\|_2^2
\leq\frac{2}{M}\sqrt{\frac{8}{\pi}}.
\]
Additionally, since $\tau^*\gtrsim 1$, we have $\sigma\lesssim 1$ by Lemma \ref{lem_sigmaVStau}.
We conclude that the inequality \eqref{eq_small_goal} holds.
If instead, the other terms dominate, we have the inequality:
\[
\frac{1}{16\sqrt{2\pi}}\frac{1}{\hat{\tau}}+\frac{\hat{\tau}}{2\sqrt{8\pi(1+\sigma^2)}}\|\hat{\beta}-\beta^*\|_2^2
\leq P\left(b(\hat{\tau}\hat{\beta})-b(\hat{\tau}\beta^*)+u(\hat{\tau}\hat{\beta})-u(\hat{\tau}\beta^*)\right)
\]
\[
\stackrel{(iv)}{\leq }(P-P_n)\left(b(\hat{\tau}\hat{\beta})-b(\hat{\tau}\beta^*)+u(\hat{\tau}\hat{\beta})-u(\hat{\tau}\beta^*)\right).
\]
In $(iv)$ we used that the estimator $\hat{\tau}\hat{\beta}$ minimizes the empirical risk over a Euclidean ball of radius $M$, and $\|\hat{\tau}\beta^*\|_2=\hat{\tau}\leq M$.
We distinguish whether the bounded or the unbounded term dominates.
Suppose first that the unbounded term dominates. 
Then, we use Proposition \ref{prop_small_unbounded_global}.
We borrow its notation henceforth.
Let $R:=2$, since $\|\hat{\beta}-\beta^*\|_2\leq 2$, and $r:=\mu^2/\sqrt{p}$.
This choice of $r$ is possible since if $\|\hat{\beta}-\beta^*\|_2\leq \mu^2/\sqrt{p}$ holds then \eqref{eq_small_goal} is true.
We note that $\rho(t')\lesssim \mu$ and $\Delta(t')\lesssim \mu^2$.
On an event $A_{1,u}$ with probability at least $1-\exp(-t)$,
\[
\frac{1}{4\sqrt{8\pi(1+\sigma^2)}}\|\hat{\beta}-\beta^*\|_2^2\leq (P-P_n)\left(u(\hat{\beta})-u(\beta^*)\right)=(P-P_n)\tilde{u}(\hat{\beta})
\]
\[
\lesssim (\|\hat{\beta}-\beta^*\|_2+\sigma)\left(\sqrt{(\|\hat{\beta}-\beta^*\|_2+\sigma)\mu}+\mu\right)+\mu^2
\]
\[
\stackrel{(v)}{\lesssim}
\|\hat{\beta}-\beta^*\|_2\sqrt{\mu\|\hat{\beta}-\beta^*\|_2}+\mu^2.
\]
In $(v)$ we assume that $\|\hat{\beta}-\beta^*\|_2\geq \sigma\vee\mu$, otherwise \eqref{eq_small_goal} is true.
Whichever of the two terms dominates in this expression \eqref{eq_small_goal} is true.

Next, suppose that the bounded term dominates.
We will use Proposition \ref{prop_small_bounded_global} with $k=1$.
Taking $r=\mu$, if $\|\hat{\beta}-\beta^*\|_2\leq \mu$, \eqref{eq_small_goal} holds.
Else, there exist constants $c_1,c_2>0$, such that on an event $A_{1,b}$ with probability at least $1-\exp(-t)$,
\[
\frac{1}{4\log2}\left(\frac{1}{8\sqrt{2\pi}}\frac{1}{\hat{\tau}}+\frac{\hat{\tau}}{\sqrt{8\pi(1+\sigma^2)}}\|\hat{\beta}-\beta^*\|_2^2\right)
\leq(P-P_n)\frac{b(\hat{\tau}\hat{\beta})-b(\hat{\tau}\beta^*)}{\log2}
\]
\[
\leq c_1\frac{\|\hat{\beta}-\beta^*\|_2^2\hat{\tau}}{\lambda}+c_2\lambda\mu.
\]
Choosing $\lambda:=2c_1\log(2)4\sqrt{8\pi(1+\sigma^2)}$, we get:
\[
\frac{1}{4\log2}\left(\frac{1}{8\sqrt{2\pi}}\frac{1}{\hat{\tau}}+\frac{\hat{\tau}}{2\sqrt{8\pi(1+\sigma^2)}}\|\hat{\beta}-\beta^*\|_2^2\right)\lesssim \mu.
\]
Using the inequality of arithmetic and geometric means, we get $\|\hat{\beta}-\beta^*\|_2\lesssim\mu$, so \eqref{eq_small_goal} holds.
This concludes case 1.

\paragraph{Case 2: $\sqrt{6}\leq \hat{\tau}\leq M/2$.}
As in case 1, we get:
\[
\frac{1}{8\sqrt{2\pi}}\frac{1}{\hat{\tau}}\leq P\left(b(\hat{\tau}\hat{\beta})-b(2\hat{\tau}\beta^*)\right),
\]
as well as:
\[
\frac{\hat{\tau}}{\sqrt{8\pi(1+\sigma^2)}}\|\hat{\beta}-\beta^*\|_2^2
= P\left(u(\hat{\tau}\hat{\beta})-u(\hat{\tau}\beta^*)\right)
\]
\[
=P\left(u(\hat{\tau}\hat{\beta})-u(2\hat{\tau}\beta^*)+u(\hat{\tau}\beta^*)\right)
\stackrel{(i)}{\leq }
P\left(u(\hat{\tau}\hat{\beta})-u(2\hat{\tau}\beta^*)\right)+\frac{\hat{\tau}\sigma^2}{\sqrt{8\pi}}.
\]
Inequality $(i)$ follows from Corollary \ref{cor_trig_y} and that $1-\frac{1}{\sqrt{1+\sigma^2}}\leq \frac{\sigma^2}{2}$.
Together, this gives:
\[
\frac{1}{8\sqrt{2\pi}}\frac{1}{\hat{\tau}}+\frac{\hat{\tau}}{\sqrt{8\pi(1+\sigma^2)}}\|\hat{\beta}-\beta^*\|_2^2
\]
\[
\leq P\left(b(\hat{\tau}\hat{\beta})-b(2\hat{\tau}\beta^*)+u(\hat{\tau}\hat{\beta})-u(2\hat{\tau}\beta^*)\right)+\frac{\hat{\tau}\sigma^2}{\sqrt{8\pi}}.
\]
If the last summand on the right-hand side dominates, we conclude that \eqref{eq_small_goal} and \eqref{eq_small_goal_tau} hold.
If not, we can use that $\hat{\tau}\hat{\beta}$ minimizes the empirical risk over a Euclidean ball of radius $M$ and that $\|2\hat{\tau}\beta^*\|_2\leq M$.
We remark that this was not true in case 1. 
It follows that:
\[
\frac{1}{8\sqrt{2\pi}}\frac{1}{\hat{\tau}}+\frac{\hat{\tau}}{\sqrt{8\pi(1+\sigma^2)}}\|\hat{\beta}-\beta^*\|_2^2
\]
\[
\leq 2(P-P_n)\left(b(\hat{\tau}\hat{\beta})-b(2\hat{\tau}\beta^*)+u(\hat{\tau}\hat{\beta})-u(2\hat{\tau}\beta^*)\right).
\]
Again, we make a case distinction.
If the unbounded term dominates, we try to control:
\[
\frac{1}{4\cdot{8\sqrt{2\pi}}}\frac{1}{\hat{\tau}}+\frac{\hat{\tau}}{4\sqrt{8\pi(1+\sigma^2)}}\|\hat{\beta}-\beta^*\|_2^2
\leq (P-P_n)\left(u(\hat{\tau}\hat{\beta})-u(2\hat{\tau}\beta^*)\right)
\]
\[
=\hat{\tau}(P-P_n)\tilde{u}(\hat{\beta})-\hat{\tau}(P-P_n)u(\beta^*).
\]
Here, $\tilde{u}(\hat{\beta}):=u(\hat{\tau}\hat{\beta})-u(\hat{\tau}\beta^*)$.
If the right term dominates, we can directly use Bernstein's inequality \cite[Theorem 2.10]{boucheron2013concentration}.
With Corollary \ref{cor_trig}, we see that Bernstein's condition is met with constants $K\lesssim \sigma$ and $\kappa\lesssim \sigma^{3/2}$.
Therefore, by Bernstein's inequality, on an event $A_{2,u}$ with probability at least $1-\exp(-t)$, it holds that:
\[
-\hat{\tau}(P-P_n)u(\beta^*)\lesssim \sigma^{3/2}\sqrt{\mu}+\sigma\mu\leq \sigma^2\vee\mu^2.
\]
So, both \eqref{eq_small_goal} and \eqref{eq_small_goal_tau} hold.
If instead the term in $\tilde{u}(\hat{\beta})$ dominates, we use Proposition \ref{prop_small_unbounded_global} to bound it. 
So, as in case 1, we find that on the event $A_{1,u}$ which has probability at least $1-\exp(-t)$,
\[
\frac{\hat{\tau}}{4\sqrt{8\pi(1+\sigma^2)}}\|\hat{\beta}-\beta^*\|_2^2
\lesssim
\hat{\tau}\left(\|\hat{\beta}-\beta^*\|_2\sqrt{\mu\|\hat{\beta}-\beta^*\|_2}+\mu^2\right).
\]
So, \eqref{eq_small_goal} holds.
Consequentially, on the event $A_{1,u}$, we also have:
\[
\frac{1}{4\cdot{8\sqrt{2\pi}}}\frac{1}{\hat{\tau}}
\lesssim 
\hat{\tau}\left(\|\hat{\beta}-\beta^*\|_2\sqrt{\mu\|\hat{\beta}-\beta^*\|_2}+\mu^2\right)\lesssim \hat{\tau}\mu^2.
\]
So, \eqref{eq_small_goal_tau} is true.

Now suppose that the bounded term dominates.
Then:
\[
\frac{1}{4\sqrt{8\pi}}\frac{1}{\hat{\tau}}+\frac{\hat{\tau}}{4\sqrt{8\pi(1+\sigma^2)}}\|\hat{\beta}-\beta^*\|_2^2
\leq (P-P_n)\left(b(\hat{\tau}\hat{\beta})-b(2\hat{\tau}\beta^*)\right).
\]
As in case 1, we use Proposition \ref{prop_small_bounded_global} with $r:=\mu$, although now with $k=2$.
There are absolute constants $c_1,c_2,c_3>0$, such that on the event $A_{2,b}$ which has probability at least $1-\exp(-t)$,
\[
(P-P_n)\frac{b(\hat{\tau}\hat{\beta})-b(2\hat{\tau}\beta^*) }{\log 2}
\leq c_1\frac{\|\hat{\beta}-\beta^*\|_2^2\hat{\tau}}{\lambda}+c_2\frac{1}{\lambda\hat{\tau}}+c_3\lambda\mu.
\]
We choose $\lambda:=2\max\{c_14\sqrt{8\pi(1+\sigma^2)},c_24\cdot{8\sqrt{2\pi}}\}$, so that we get:
\[
\frac{1}{8\sqrt{8\pi}}\frac{1}{\hat{\tau}}+\frac{\hat{\tau}}{8\sqrt{8\pi(1+\sigma^2)}}\|\hat{\beta}-\beta^*\|_2^2\lesssim \mu.
\]
Ignoring the second summand on the left-hand side, we see that \eqref{eq_small_goal_tau} is met.
By the inequality of arithmetic and geometric means, we see that \eqref{eq_small_goal} is met.
This concludes case 2.

\paragraph{Case 3: $\hat{\tau}\leq {\sqrt{6}}$.}
Before we start with case 3, we show that assumption \eqref{eq_small_assumption_n_easy} implies:
\begin{equation}\label{eq_small_assumption_n}
Pl(\sqrt{6}\beta^*)-Pl(\min\{\tau^*,M\}\beta^*)>\sqrt{\frac{9t+3p(2+\log(3\sqrt{6n}))}{n}}
\end{equation}
Let $\tau^\sharp:=\min\{\tau^*,M\}$.
From Lemma \ref{lem_sigmaVStau} and $\sigma\leq 1/\sqrt{7}$ we get $\tau^*\geq 2\sqrt{6}$.
It follows that $\tau^\sharp\geq 2\sqrt{6}$, as we assumed $M\geq 2\sqrt{6}$. 
By Corollary \ref{cor_trig_y},
\[
P\left(l(\sqrt{6}\beta^*)-l(\tau^\sharp\beta^*)\right)
=P\left(b(\sqrt{6}\beta^*)-b(\tau^\sharp\beta^*)\right)-\frac{\tau^\sharp-\sqrt{6}}{\sqrt{2\pi}}\left(1-\frac{1}{\sqrt{1+\sigma^2}}\right).
\]
\[
\geq P\left(b(\sqrt{6}\beta^*)-b(2\sqrt{6}\beta^*)\right)-\frac{\tau^*-\sqrt{6}}{\sqrt{2\pi}}\left(1-\frac{1}{\sqrt{1+\sigma^2}}\right)
\]
\[
\geq  P\left(b(\sqrt{6}\beta^*)-b(2\sqrt{6}\beta^*)\right)-\frac{\tau^*-\sqrt{6}}{\sqrt{2\pi}}\frac{\sigma^2}{2}
\]
\[
\stackrel{(i)}{\geq} P\left(b(\sqrt{6}\beta^*)-b(2\sqrt{6}\beta^*)\right)-\frac{\sqrt{2\pi}\sigma-6\sigma^2}{\sqrt{8\pi}}
\stackrel{(ii)}{\geq} 0.051.
\]
In $(i)$ we used Lemma \ref{lem_sigmaVStau}.
Inequality $(ii)$ follows from our bound on $\sigma$, the constant is found by simulation of the integrals.
Consequentially, assumption \eqref{eq_small_assumption_n_easy} implies \eqref{eq_small_assumption_n}.

Now we start with case 3.
First, by Lemma \ref{lem_opt_direction}, $Pl(\hat{\gamma})\geq Pl(\hat{\tau}\beta^*)$.
Moreover, since $\hat{\tau}\leq 6$, by convexity $Pl(\hat{\tau}\beta^*)\geq Pl(\sqrt{6}\beta^*)$.
Therefore,
\[
P\left(l(\hat{\gamma})-l(\tau^\sharp\beta^*)\right)\geq P\left(l(\sqrt{6}\beta^*)-l(\tau^\sharp\beta^*)\right)
\stackrel{(iii)}{\geq} \sqrt{\frac{9t+3p(2+\log(3\sqrt{6n}))}{n}}.
\]
Inequality $(iii)$ follows from \eqref{eq_small_assumption_n}.
By Lemma \ref{lem_slow_concentration}, this event together with $\hat{\tau}\leq 6$ occurs on an event $A_3^c$, which has probability at most $2\exp(-t)$.
This concludes case 3.

We conclude that \eqref{eq_small_goal} and \eqref{eq_small_goal_tau} are true on the event $A_{1,u}\cap A_{1,b}\cap A_{2,u}\cap A_{2,b}\cap A_3$, which has probability at least $1-6\exp(-t)$.
The proof is complete.
\end{proof}

The following result was used in the proof of case 3 in Theorem \ref{thm_small}. 
There, we use it with $\gamma^\sharp$ as $M\beta^*$ or $\tau^*\beta^*$.
The idea is, that two events are unlikely to occur together.
The first one states that $\hat{\tau}\leq r$, and the second one states that the excess risk is larger than $\sqrt{(t+p\log(n))/n}$ up to constant factors.
In Theorem \ref{thm_small}, we assume that this lower bound is small, implying that the second statement holds.
Consequentially, $\hat{\tau}> r$ must be true with high probability.

\begin{lem}\label{lem_slow_concentration}
Fix $r>0$, and $\gamma^\sharp\in\mathbb{R}^p\setminus\{0\}$ such that $\|\gamma^\sharp\|_2\leq M$.
Let $A$ be the event that $\|\hat{\gamma}\|_2\leq r$.
Moreover, let $B$ be the event that:
\[
P\left(l(\hat{\gamma})-l(\gamma^\sharp)\right)> \sqrt{\frac{9t+3p(2+\log(3r\sqrt{n}))}{n}}.
\]
Then, $\mathbb{P}[A\cap B]\leq 2\exp(-t)$.
\end{lem}

\begin{proof}
Let $B_r$ be the Euclidean ball of radius $r$ centered at zero.
We start by bounding the following expression with high probability:
\[
\sup_{\gamma \in B_r}(P-P_n)\left(l(\gamma)-l(\gamma^\sharp)\right).
\]
The ball $B_r$ can be covered with at most $(3r/\epsilon)^p$ balls of radius $\epsilon\in(0,1)$.
Let $\mathcal{N}\subset B_r$ be such a covering.
\[
\sup_{\gamma \in B_r}(P-P_n)\left(l(\gamma)-l(\gamma^\sharp)\right)
\]
\[
=\inf_{\gamma'\in \mathcal{N}} \sup_{\gamma \in B_r}(P-P_n)(l(\gamma)-l(\gamma'))+(P-P_n)(l(\gamma')-l(\gamma^\sharp))
\]
\begin{equation}\label{eq_small_contradiction}
\leq \inf_{\gamma'\in \mathcal{N}} \sup_{\gamma \in B_r}(P-P_n)(l(\gamma)-l(\gamma'))+ \sup_{\gamma'\in \mathcal{N}}(P-P_n)(l(\gamma')-l(\gamma^\sharp)).
\end{equation}
To bound the first summand, we proceed as in the proof of Lemma \ref{lem_unbounded_discretize}.
\[
\inf_{\gamma'\in \mathcal{N}} \sup_{\gamma \in B_r}\left|(P-P_n)\|x\|_2\frac{l(\gamma)-l(\gamma')}{\|x\|_2}\right|
\leq \inf_{\gamma'\in \mathcal{N}} \sup_{\gamma \in B_r}(P+P_n)\|x\|_2\left|\frac{l(\gamma)-l(\gamma')}{\|x\|_2}\right|
\]
\[
\stackrel{(i)}{\leq} \epsilon (P+P_n)\|x\|_2
\leq \epsilon\left((P_n-P)\|x\|_2+2P\|x\|_2\right).
\]
In $(i)$, we used that the function $l(\cdot,yx):\mathbb{R}^p\rightarrow\mathbb{R}$, mapping $\gamma \mapsto l(\gamma,yx)$ is $\|x\|_2$-Lipschitz.
By concentration for 1-Lipschiz functions of standard normally distributed random variables (see e.g. \cite[Theorem 5.6]{boucheron2013concentration}), with probability at most $\exp(-t)$, it holds that $(P_n-P)\|x\|_2> \sqrt{2t/n}$.
By Jensen's inequality, $P\|x\|_2\leq \sqrt{p}$.
We conclude that with probability at least $1-\exp(-t)$,
\[
\inf_{\gamma'\in \mathcal{N}} \sup_{\gamma \in B_r}\left|(P-P_n)\|x\|_2\frac{l(\gamma)-l(\gamma')}{\|x\|_2}\right|
\leq \epsilon \left(\sqrt{\frac{2t}{n}}+2\sqrt{p}\right).
\]
For the second term in \eqref{eq_small_contradiction}, we use that the function $\mathbb{R}^p\rightarrow\mathbb{R}$, mapping $\gamma\mapsto l(\gamma,yx)-l(\gamma^\sharp)$ takes values in $(-\log(2),\log(2))$.
So, by Hoeffding's inequality (see e.g. \cite[Theorem 2.8]{boucheron2013concentration}), for any $\gamma'\in\mathcal{N}$,
\[
\mathbb{P}\left[(P-P_n)(l(\gamma')-l(\gamma^\sharp))>\sqrt{\frac{t}{n}}\right]\leq \exp\left(-\frac{t}{2\log(2)^2}\right)\leq \exp(-t).
\]
Moreover, by a union bound,
\[
\mathbb{P}\left[\max_{\gamma'\in\mathcal{N}}(P-P_n)(l(\gamma')-l(\gamma^\sharp))>\sqrt{\frac{t+\log |\mathcal{N}|}{n}}\right]\leq \exp\left(-t\right).
\]
Moreover, by choice of $\mathcal{N}$, it holds that $\log |\mathcal{N}|\leq p\log(3r/\epsilon)$.
Altogether, we find that with probability at least $1-2\exp(-t)$,
\[
 \sup_{\gamma \in B_r}(P-P_n)\left(l(\gamma)-l(\gamma^\sharp)\right)
\leq \epsilon \left(\sqrt{\frac{2t}{n}}+2\sqrt{p}\right)+\sqrt{\frac{t+p\log(3r/\epsilon)}{n}}.
\]
Choosing $\epsilon:=1/\sqrt{n}$ and Jensen's inequality gives:
\[
\sup_{\gamma \in B_r}(P-P_n)\left(l(\gamma)-l(\gamma^\sharp)\right)\leq \frac{\sqrt{2t}}{n}+2\sqrt{\frac{p}{n}}+\sqrt{\frac{t+p\log(3r\sqrt{n})}{n}}
\]
\[
\leq \sqrt{\frac{9t+3p(2+\log(3r\sqrt{n}))}{n}}.
\]
Denote this event by $C$. So, we concluded that $\mathbb{P}[C]\geq 1-2\exp(-t)$.

Now suppose that $A\cap B$ occurs.
Since $B$ occurs,
\[
\sqrt{\frac{9t+3p(2+\log(3r\sqrt{n}))}{n}}< P\left(l(\hat{\gamma})-l(\gamma^\sharp)\right).
\]
Since $\|\gamma^\sharp\|_2\leq M$, $P_nl(\hat{\gamma})\leq P_nl(\gamma^\sharp)$.
Consequentially,
\[
P\left(l(\hat{\gamma})-l(\gamma^\sharp)\right)\leq (P-P_n)\left(l(\hat{\gamma})-l(\gamma^\sharp)\right).
\]
Since $A$ occurs, $\hat{\tau}\leq r$. So:
\[
\leq \sup_{\gamma \in B_r}(P-P_n)\left(l(\gamma)-l(\gamma^\sharp)\right).
\]
We conclude that $A\cap B$ implies $C^c$.
Consequentially, $\mathbb{P}[A\cap B]\leq 1-P[C]\leq 2\exp(-t)$.
The proof is complete.
\end{proof}

\subsection{Proof of the consequence for linear separation}\label{sub_sep}
Here, we provide a proof of Proposition \ref{prop_cor_separation}.

\begin{proof}[Proof of Proposition \ref{prop_cor_separation}]
Let $A_M$ be the event with probability at least $1-4\exp(-t)$ on which the conclusion in Theorem \ref{thm_cor_large} holds, for a valid choice of $M$.
Let $B$ be the event that the data are linearly separable.
We show that if $M$ is chosen large enough, the event $A_M\cap B$ is empty.
Hence, $\mathbb{P}[B]\leq \mathbb{P}[A_M^c]\leq 4\exp(-t)$.
We start with the observation that on $B$, it holds that $\hat{\tau}=M$.
So, on $B$, by a triangular inequality,
\[
d_*(\hat{\gamma})\geq \frac{M-\tau^*}{\tau^{*3/2}}.
\]
By Theorem \ref{thm_cor_large}, on $A_M$ the quantity $d_*(\hat{\gamma})$ is upper bounded by a deterministic expression $c>0$ not depending on $M$.
So, if $M>c\tau^{*3/2}+\tau^*$, we reach a contradiction, meaning that the event $A_M\cap B$ is empty.
The proof is complete.
\end{proof}

\section*{Acknowledgments}
FK would like to thank Christoph Schultheiss, Malte Londschien, Fadoua Balabdaoui, and Nikita Zhivotovskiy for helpful comments and discussions.
The authors would like to thank the anonymous reviewers for helpful comments.

\section*{Conflict of interest and funding statement}
The authors have no conflicts of interest to declare.
F.K. is funded by the Swiss National Science Foundation (SNF).
\newpage

\bibliographystyle{plainnat}
\bibliography{master}

\newpage
\appendix
\section{Approximations for the (un)bounded term}\label{sec_app}

\subsection{Approximations for bounded term}\label{app_bounded}
In this subsection, we provide approximations for several quantities related to the bounded term.

\subsubsection{Gaussian integrals of exponential functions}

\begin{lem}\label{lem_moment_zero}
Let $\tau> 0$ and $z\sim\mathcal{N}(0,1)$. We have:
\[
 \sqrt{\frac{2}{\pi}}\left(\frac{1}{\tau}-\frac{1}{\tau^3}\right)
\leq P\exp(-\tau|z|)
\leq \sqrt{\frac{2}{\pi}}\frac{1}{\tau}.
\]
\end{lem}

The result follows from completing the square and repeated integration by parts.

\begin{lem}\label{lem_bound_noiseless}
Let $z\sim\mathcal{N}(0,1)$. Then, for $\tau>0$,
\[
\sqrt{\frac{1}{2\pi}}\frac{1}{\tau^2}\left(1-\frac{3}{\tau^2}\right)\leq 
P\frac{|z|}{1+\exp(\tau|z|)}
\leq \sqrt{\frac{2}{\pi}}\frac{1}{\tau^2}.
\]
\end{lem}

\begin{proof}
Since $\tau>0$,
\[
P\frac{|z|}{2\exp(\tau|z|)}\leq P\frac{|z|}{1+\exp(\tau|z|)}
\leq P\frac{|z|}{\exp(\tau|z|)}.
\]
We continue working with the expression for the lower bound.
The derivation for the upper bound is the same, with a factor $2$.
It holds that:
\[
P\frac{|z|}{2\exp(\tau|z|)}
=2\int_0^\infty \frac{z}{2\sqrt{2\pi}}\exp\left(-\tau z-\frac{z^2}{2}\right)dz
\]
\[
=
\sqrt{\frac{1}{2\pi}}e^{\tau^2/2}\left(\int_0^\infty (z+\tau)\exp\left(-\frac{(z+\tau)^2}{2}\right)dz-\int_0^\infty \tau\exp\left(-\frac{(z+\tau)^2}{2}\right)dz\right).
\]
\begin{equation}\label{eq_identity_tau2}
=\sqrt{\frac{1}{2\pi}}\left(1-e^{\tau^2/2}\tau\int_\tau^\infty \exp\left(-\frac{z^2}{2}\right)dz\right)
\end{equation}
Repeated integration by parts shows:
\begin{equation}\label{eq_gaussian_tails}
\frac{1}{\tau}-\frac{1}{\tau^3}\leq e^{\tau^2/2}\int_\tau^\infty \exp\left(-\frac{z^2}{2}\right)dz\leq \frac{1}{\tau}-\frac{1}{\tau^3}+\frac{3}{\tau^5}.
\end{equation}
Inserting these bounds in \eqref{eq_identity_tau2} completes the proof.
\end{proof}

\begin{lem}\label{lem_bound_noiseless_2}
Let $z\sim\mathcal{N}(0,1)$. Then, for $\tau>0$,
\[
\sqrt{\frac{2}{\pi}}\left(\frac{2}{\tau^3}-\frac{12}{\tau^5}-\frac{15}{\tau^7}\right)\leq 
P\frac{|z|^2}{\exp(\tau|z|)}
\leq  \sqrt{\frac{2}{\pi}}\left(\frac{2}{\tau^3}+\frac{3}{\tau^5}\right).
\]
\end{lem}

The lower bound is non-zero if $\tau>\sqrt{3+\sqrt{33/2}}\approx 2.657$. 

\begin{proof}
We decompose the integral into two parts, with $I_1$ and $I_2$ defined below.
\[
P\frac{|z|^2}{\exp(\tau|z|)}
= \sqrt{\frac{2}{\pi}}\int_0^\infty z^2\exp\left(-\tau z-\frac{z^2}{2}\right)dz=:\sqrt{\frac{2}{\pi}}(I_1-I_2).
\]
To evaluate the first integral $I_1$, we use integration by parts.
\[
I_1:=e^{\tau^2/2}\int_0^\infty (z+\tau)^2\exp\left(-\frac{(z+\tau)^2}{2}\right)dz
\]
\[
=e^{\tau^2/2}\left(\left[(z+\tau)\left(-\exp\left(-\frac{(z+\tau)^2}{2}\right)\right)\right]^\infty_0
+\int_0^\infty \exp\left(-\frac{(z+\tau)^2}{2}\right)dz\right)
\]
\[
=\tau+e^{\tau^2/2}\int_\tau^\infty \exp\left(-\frac{z^2}{2}\right)dz.
\]
The second integral is:
\[
I_2:=e^{\tau^2/2}\int_0^\infty (2z\tau+\tau^2)\exp\left(-\frac{(z+\tau)^2}{2}\right)dz
\]
\[
=2\tau\left(1-\tau e^{\tau^2/2}\int_\tau^\infty \exp\left(-\frac{z^2}{2}\right)dz\right)
+\tau^2e^{\tau^2/2}\int_\tau^\infty\exp\left(-\frac{z^2}{2}\right)dz
\]
\[
=2\tau-\tau^2e^{\tau^2/2}\int_\tau^\infty\exp\left(-\frac{z^2}{2}\right)dz.
\]
We conclude that:
\[
P\frac{|z|^2}{\exp(\tau|z|)}
=\sqrt{\frac{2}{\pi}}(I_1-I_2)
=\sqrt{\frac{2}{\pi}}\left(e^{\tau^2/2}\int_\tau^\infty\exp\left(-\frac{z^2}{2}\right)dz(1+\tau^2)-\tau\right).
\]
Repeated integration by parts shows:
\[
\frac{1}{t}-\frac{1}{t^3}+\frac{3}{t^5}-\frac{15}{t^7}\leq \exp\left(\frac{t^2}{2}\right)\int_t^\infty \exp\left(-\frac{x^2}{2}\right)dx
\leq \frac{1}{t}-\frac{1}{t^3}+\frac{3}{t^5}.
\]
Inserting this concludes the proof.
\end{proof}

\subsubsection{A first-order bound}

\begin{lem}\label{lem_moment_distance}
Suppose that $x\sim\mathcal{N}(0,I)$ and $\gamma,\gamma'\in\mathbb{R}^p\setminus\{0\}$.
Moreover, define $\rho:=\gamma^T\gamma'/(\|\gamma\|_2\|\gamma'\|_2)$.
Then:
\[
\sqrt{P\frac{|x^T(\gamma-\gamma')|^2}{\exp(2|\gamma^{'T}x|)}}
\]
\[
\leq \|\gamma\|_2\sqrt{\frac{1}{\sqrt{2\pi}}
\frac{(1-\rho^2)}{\|\gamma'\|_2}}+|\rho\|\gamma\|_2-\|\gamma'\|_2|
\sqrt{\sqrt{\frac{2}{\pi}}\left(\frac{1}{4}\frac{1}{\|\gamma'\|_2^3}+\frac{3}{32}\frac{1}{\|\gamma'\|_2^5}\right)}.
\]
\end{lem}

If $\|\gamma'\|_2\geq 1$, we can bound this term as:
\[
\sqrt{P\frac{|x^T(\gamma-\gamma')|^2}{\exp(2|\gamma^{'T}x|)}}
\leq \|\gamma\|_2\sqrt{\frac{1}{\sqrt{2\pi}}
\frac{(1-\rho^2)}{\|\gamma'\|_2}}+
\sqrt{\frac{11}{32}\sqrt{\frac{2}{\pi}}}\frac{|\rho\|\gamma\|_2-\|\gamma'\|_2|}{\|\gamma'\|_2^{3/2}}.
\]
In addition, suppose that $\|\gamma\|_2=\|\gamma'\|_2=\tau$, and define $\beta:=\gamma/\tau$ as well as $\beta':=\gamma'/\tau$.
Since $1-\rho^2\leq 2(1-\rho)=\|\beta-\beta'\|_2^2$, we have:
\[
\sqrt{P\frac{|x^T(\gamma-\gamma')|^2}{\exp(2|\gamma^{'T}x|)}}
\leq \sqrt{\frac{\tau}{\sqrt{2\pi}}}\|\beta-\beta'\|_2
+\sqrt{\frac{11}{32}\sqrt{\frac{2}{\pi}}}\frac{\|\beta-\beta'\|_2^2}{\sqrt{\tau}}.
\]

\begin{proof}
Let $\Pi_{\gamma'}
:=\gamma'\gamma^{\prime T}/\|\gamma'\|_2^2$ be the projection onto the span of $\gamma'$.
Then, 
\[
x^T(\gamma-\gamma')
=x^T((I-\Pi_{\gamma'})\gamma+\Pi_{\gamma'}(\gamma-\gamma')).
\]
By the triangular inequality,
\[
\sqrt{P\frac{|x^T(\gamma-\gamma')|^2}{\exp(2|\gamma^{'T}x|)}}\leq 
\sqrt{P\frac{|x^T(I-\Pi_{\gamma'})\gamma|^2}{\exp(2|\gamma^{'T}x|)}}+\sqrt{P\frac{|x^T\Pi_{\gamma'}(\gamma-\gamma')|^2}{\exp(2|\gamma^{'T}x|)}}.
\]
We start with the first term on the right-hand side.
The random variables $\gamma^{\prime T}x$ and $x^T(I-\Pi_{\gamma'})\gamma$ are uncorrelated.
Since $x$ is Gaussian, they are independent.
So:
\[
\sqrt{P\frac{|x^T(I-\Pi_{\gamma'})\gamma|^2}{\exp(2|\gamma^{'T}x|)}}
=\sqrt{P\exp(-2|\gamma^{'T}x|)P|x^T(I-\Pi_{\gamma'})\gamma|^2}.
\]
By Lemma \ref{lem_moment_zero}, $P\exp(-2|\gamma^{\prime T}x|)\leq 1/(\sqrt{2\pi}\|\gamma'\|_2)$.
Moreover, by definition of $\rho$, we have: $P|x^T(I-\Pi_{\gamma'})\gamma|^2=\|\gamma\|_2^2(1-\rho^2)$.
So, the first term is bounded by:
\[
\sqrt{P\exp(-2|\gamma^{'T}x|)P|x^T(I-\Pi_{\gamma'})\gamma|^2}
\leq \|\gamma\|_2\sqrt{\frac{1}{\sqrt{2\pi}\|\gamma'\|_2}
(1-\rho^2)}.
\]
The second term can be bounded as follows:
\[
\sqrt{P\frac{|x^T\Pi_{\gamma'}(\gamma-\gamma')|^2}{\exp(2|\gamma^{'T}x|)}}
= |\rho\|\gamma\|_2-\|\gamma'\|_2|\sqrt{P\frac{(|x^T\gamma'|/\|\gamma'\|_2)^2}{\exp(2|\gamma^{'T}x|)}}
\]
\[
\stackrel{(i)}{\leq}  |\rho\|\gamma\|_2-\|\gamma'\|_2|
\sqrt{\sqrt{\frac{2}{\pi}}\left(\frac{2}{(2\|\gamma'\|_2)^3}+\frac{3}{(2\|\gamma'\|_2)^5}\right)}
\]
\[
=|\rho\|\gamma\|_2-\|\gamma'\|_2|
\sqrt{\sqrt{\frac{2}{\pi}}\left(\frac{1}{4}\frac{1}{\|\gamma'\|_2^3}+\frac{3}{32}\frac{1}{\|\gamma'\|_2^5}\right)}.
\]
In $(i)$, we used Lemma \ref{lem_bound_noiseless_2}. 
The proof is complete.
\end{proof}

\subsubsection{A second-order bound}

\begin{lem}\label{lem_f_dotdot}
Let $f:\mathbb{R}\rightarrow\mathbb{R}$ be the mapping $\tau\mapsto P\log(1+\exp(-\tau|z|))$, where $z\sim\mathcal{N}(0,1)$.
Suppose that for some $\kappa>\sqrt{3+\sqrt{33/2}}$, we have $\min\{\tau,\tau^*\}\geq \kappa$.
Then, for all $\bar{\tau}\in[\tau,\tau^*]\cup[\tau^*,\tau]$,
\[
\frac{1}{\sqrt{8\pi}}\left(2-\frac{12}{\kappa^2}-\frac{15}{\kappa^4}\right)\frac{1}{\bar{\tau}^3}
\leq 
\ddot{f}(\bar{\tau}).
\]
\end{lem}

In particular, if $\min\{\tau,\tau^*\}\geq \sqrt{6+\sqrt{51}}$,
\[
\sqrt{\frac{1}{8\pi}}\frac{1}{\bar{\tau}^3}
\leq \ddot{f}(\bar{\tau}).
\]

\begin{proof}
Using Lebesgue's dominated convergence theorem, we find for any $t>0$,
\[
\ddot{f}(t)=P\frac{z^2\exp(t|z|)}{(1+\exp(t|z|))^2}\geq P\frac{z^2}{4\exp(t|z|)}.
\]
Since $\bar{\tau}$ is an intermediate point of $\tau$ and $\tau^*$, it holds that $\bar{\tau}\geq \kappa$.
Consequentially, by Lemma \ref{lem_bound_noiseless_2},
\[
\sqrt{\frac{2}{\pi}}\left(2-\frac{12}{\kappa^2}-\frac{15}{\kappa^4}\right)\frac{1}{\bar{\tau}^3}
\leq 
P\frac{|z|^2}{\exp(\bar{\tau}|z|)}.
\]
The proof is complete.
\end{proof}

\subsubsection{Differences between bounded terms}

\begin{lem}\label{lem_moment_bounded_difference}
Let $k>1$, and $z\sim\mathcal{N}(0,1)$.
Then,
\[
\frac{1}{\sqrt{2\pi}}\frac{k-1}{k^2\tau}\left(
1-\frac{3}{\tau^2}
\right)
\leq P\log\left(
\frac{1+e^{-\tau|z|}}{1+e^{-k\tau|z|}}
\right)
\leq \frac{k-1}{\tau}\sqrt{\frac{2}{\pi}}.
\]
Additionally, if $\tau\leq l$,
\[
\frac{1}{\sqrt{2\pi}}\frac{\tau(k-1)}{k^2l^2}\left(1-\frac{3}{k^2l^2}\right)
\leq 
P\log\left(
\frac{1+e^{-\tau|z|}}{1+e^{-k\tau|z|}}
\right)
\leq \frac{\tau(k-1)}{\sqrt{2\pi}}.
\]
\end{lem}

For $\tau\leq \sqrt{3}$, the first lower bound is void.
In this case, we can use the second lower bound, with a large enough bound $l$, to make the bound non-trivial.
In the first lower bound, one could assume $\tau>\sqrt{6}$, to get a nicer lower bound.

\begin{proof}
Let $h_z:(0,\infty)\rightarrow(0,\log 2)$ be the mapping $\tau\mapsto \log(1+\exp(-\tau|z|))$.
By the mean value theorem, for some (random) $t_z\in[\tau,k\tau]$,
\[
Ph_z(\tau)-h_z(k\tau)=\tau(1-k)Ph'_z(t_z)=\tau(k-1)P\frac{|z|}{1+\exp(t_z|z|)}.
\]
For the upper bound, by Lemma \ref{lem_bound_noiseless}:
\[
\tau(k-1)P\frac{|z|}{1+\exp(t_z|z|)}\leq\tau(k-1)P\frac{|z|}{1+\exp(\tau|z|)}\leq \sqrt{\frac{2}{\pi}}\frac{k-1}{\tau}.
\]
On the other hand, since $\tau|z|>0$,
\[
\tau(k-1)P\frac{|z|}{1+\exp(\tau|z|)}\leq \tau(k-1)P\frac{|z|}{2}\leq \frac{\tau(k-1)}{\sqrt{2\pi}}.
\]
For the first lower bound, by Lemma \ref{lem_bound_noiseless}:
\[
\tau(k-1)P\frac{|z|}{1+\exp(t_z|z|)}\geq \tau(k-1)P\frac{|z|}{1+\exp(k\tau|z|)}
\geq \frac{1}{\sqrt{2\pi}}\frac{k-1}{k^2\tau}\left(
1-\frac{3}{k^2\tau^2}
\right).
\]
For the second lower bound, use that since $\tau\leq l$, $t_z\leq kl$. Using Lemma \ref{lem_bound_noiseless},
\[
\tau(k-1)P\frac{|z|}{1+\exp(t_z|z|)}
\geq \tau(k-1)P\frac{|z|}{1+\exp(kl|z|)}
\]
\[
\geq \frac{1}{\sqrt{2\pi}}\frac{\tau(k-1)}{k^2l^2}\left(1-\frac{3}{k^2l^2}\right).
\]
\end{proof}

\begin{lem}\label{lem_moment_bounded_difference_variance}
Let $k>1$, and $z\sim\mathcal{N}(0,1)$.
Then,
\[
P\left(\log
\frac{1+e^{-\tau|z|}}{1+e^{-k\tau|z|}}
\right)^2
\leq (k-1)^2\min\left\{\sqrt{\frac{1}{2^5\pi}}\left(\frac{2}{\tau}+\frac{3}{4\tau^3}\right), 
\frac{\tau^2}{4}\right\}.
\]
Additionally, if $\tau\leq l$, 
\[
  \tau^2(k-1)^2\frac{1}{\sqrt{2^7\pi}}\left(\frac{1}{l^3}-\frac{3}{2l^5}-\frac{15}{2^5l^7}\right)
 \leq P\log\left(
\frac{1+e^{-\tau|z|}}{1+e^{-k\tau|z|}}
\right)^2.
\]
\end{lem}

Note that in the lower bound, we may simply take $l=\tau$ to match the upper bound if $\tau$ is large.

\begin{proof}
Let $h_z:(0,\infty)\rightarrow(0,\log 2)$ be the mapping $\tau\mapsto \log(1+\exp(-\tau|z|))$.
By the mean value theorem, for some $t_z\in[\tau,k\tau]$,
\[
P(h_z(\tau)-h_z(k\tau))^2=\tau^2(1-k)^2Ph'_z(t_z)^2=\tau^2(k-1)^2P\frac{|z|^2}{(1+\exp(\tau|z|))^2}.
\]
For the upper bound, by Lemma \ref{lem_bound_noiseless_2}:
\[
\tau^2(k-1)^2P\frac{|z|^2}{(1+\exp(\tau|z|))^2}
\leq\tau^2(k-1)^2P\frac{|z|^2}{\exp(2\tau|z|)}.
\]
\[
\leq 
\tau^2(k-1)^2\sqrt{\frac{2}{\pi}}\left(\frac{2}{2^3\tau^3}+\frac{3}{2^5\tau^5}\right)
\leq \sqrt{\frac{1}{2^5\pi}}(k-1)^2\left(\frac{2}{\tau}+\frac{3}{4\tau^3}\right).
\]
On the other hand, since $\tau|z|>0$,
\[
\tau^2(k-1)^2P\frac{|z|^2}{(1+\exp(\tau|z|))^2}\leq \tau^2(k-1)^2P\frac{|z|^2}{2^2}\leq \frac{\tau^2(k-1)^2}{4}.
\]
For the lower bound, since $\tau\leq  l$,
\[
\tau^2(k-1)^2P\frac{|z|^2}{(1+\exp(\tau|z|))^2}\geq \tau^2(k-1)^2P\frac{|z|^2}{2^2\exp(2l|z|)}
\]
\[
\stackrel{(i)}{\geq} \tau^2(k-1)^2\frac{1}{\sqrt{2^7\pi}}\left(\frac{1}{l^3}-\frac{3}{2l^5}-\frac{15}{2^5l^7}\right).
\]
Inequality $(i)$ follows from Lemma \ref{lem_bound_noiseless_2}.
The proof is complete.
\end{proof}

\begin{lem}\label{lem_moment_bounded_variance_distance}
Fix $\beta,\beta'\in S^{p-1}$ and $\tau\geq 1$.
Then,
\[
P\left(
\log\frac{1+\exp(-\tau |x^T\beta|)}{1+\exp(-\tau |x^T\beta'|)}
\right)^2
\leq 2\|\beta-\beta'\|_2^2\left(
\sqrt{\frac{\tau}{\sqrt{2\pi}}}+\sqrt{\frac{11}{8}\sqrt{\frac{2}{\pi}}} \frac{1}{\sqrt{\tau}}\right)^2.
\]
\end{lem}

\begin{proof}
Let $h:(0,\infty)\rightarrow(0,\log 2)$ be the mapping $z\mapsto \log(1+\exp(-\tau |z|))$.
There exists a $|z|\in[\min\{|x^T\beta|,|x^T\beta'|\},\max\{|x^T\beta|,|x^T\beta'|\}]$, such that:
\[
P\left(
h(|x^T\beta|)-h(|x^T\beta'|)
\right)^2
=P\left(
|x^T(\beta-\beta')|\frac{\tau}{1+\exp(\tau |z|)}
\right)^2
\]
\[
\stackrel{(i)}{\leq} \tau^2 \left(P
\frac{|x^T(\beta-\beta')|^2}{\exp(2\tau |x^T\beta|)}
+P\frac{|x^T(\beta-\beta')|^2}{\exp(2\tau |x^T\beta'|)}\right)
= 2 P\frac{|x^T(\tau\beta-\tau\beta')|^2}{\exp(2\tau |x^T\beta|)}.
\]
In $(i)$, we used that the derivative of a convex function is increasing, so it reaches its maximum on the boundary of its domain.
By Lemma \ref{lem_moment_distance}, the right-hand-side is bounded from above by:
\[
2\left(\sqrt{\tau}\sqrt{\frac{1}{\sqrt{2\pi}}}\|\beta-\beta'\|_2
+\sqrt{\frac{11}{32}\sqrt{\frac{2}{\pi}}} \frac{\|\beta-\beta'\|_2^2}{\sqrt{\tau}}
\right)^2
\]
\[
=2\|\beta-\beta'\|_2^2\left(
\sqrt{\frac{\tau}{\sqrt{2\pi}}}+2\sqrt{\frac{11}{32}\sqrt{\frac{2}{\pi}}} \frac{1}{\sqrt{\tau}}
\right)^2.
\]
\end{proof}

\newpage
\subsection{Approximations and identities for unbounded term}
Recall the definition of the unbounded term $u(\gamma,x,y)=|x^T\gamma|1\{yx^T\gamma<0\}$.
Here we derive several results to control $u$.
The Gaussian distribution proves very handy here, as we can obtain sharp upper and lower bounds, in some cases even identities.
As $u$ is positivey homogeneous, for all $\beta\in S^{p-1}$ and $\tau>0$, it holds that $u(\tau\beta)=\tau u(\beta)$.
Consequentially, we restrict our attention to the domain $S^{p-1}$.

\subsubsection{A trigonometric argument and its consequences}
The following identity is the core of our analysis of the unbounded term.
An approximation of its right-hand side will serve us to derive the parameters for Bernstein's inequality.
The case $m=1$ will form the basis of lower and upper bounds for the excess risk.
The integral of the sine function raised to the $m$-th power can be calculated using some trigonometry and basic integration techniques.
However, we will only calculate the case $m=1$, while for $m>1$, we will use the bound $\sin(x)\leq x$ for $x\geq 0$.

\begin{prop}\label{prop_trig} 
Let $x\sim\mathcal{N}(0,I)$ and $\beta,\beta'\in S^{p-1}$.
Let $m\in\mathbb{N}$.
Then,
\[
P|x^T\beta|^m1\{\beta^Txx^T\beta'<0 \}
=2^{1+m/2}\Gamma(1+m/2)\int_0^{\arccos(\beta^T\beta')}\sin(\alpha)^m \frac{d\alpha}{2\pi}.
\]
\end{prop}

\begin{rem}[Non-identity covariance matrix]
We prove Proposition \ref{prop_trig} and related results for $x\sim\mathcal{N}(0,I)$. 
Yet, the conversion to $x\sim\mathcal{N}(0,\Sigma)$ with $\Sigma\neq I$ is straightforward:
Since $\Sigma$ is a covariance matrix, it is positive definite.
Consequentially, there exists a unique positive definite matrix $\sqrt{\Sigma}$, such that $\Sigma=\sqrt{\Sigma}\sqrt{\Sigma}$.
So, $x=\sqrt{\Sigma}z$, where $z\sim\mathcal{N}(0,I)$.
Therefore, $x^T\beta=z^T(\sqrt{\Sigma}\beta)$.
So, Proposition \ref{prop_trig} gives us:
\[
P|x^T\beta|^m1\{\beta^Txx^T\beta'<0\}
\]
\[
=(\beta^T\Sigma\beta)^{m/2}2^{1+m/2}\Gamma(1+m/2)\int_0^{\arccos(\beta^T\Sigma\beta')}\sin(\alpha)^m \frac{d\alpha}{2\pi}.
\]
Note that the expression is no longer symmetric in $\beta,\beta'$, but for specific choices of $\Sigma$.
\end{rem}

\begin{proof}
There are two special cases: $\beta\in\{\beta',-\beta'\}$.
They follow from continuity or separate verification.
Excluding those two cases, $\beta,\beta'$ are in general position.
We show that we may take $p=2$.
Let $\Pi:\mathbb{R}^p\rightarrow\mathbb{R}^p$ be the orthogonal projection of $\mathbb{R}^p$ onto the span of $\beta$ and $\beta'$.
Then:
\[
P|x^T\beta|^m1\{\beta^Tx x^T\beta'<0\}
=P|x^T(\Pi\beta)|^m1\{(\Pi\beta)^Tx x^T(\Pi\beta')<0\}
\]
\[
=P|(\Pi x)^T\beta|^m1\{\beta^T(\Pi x) (\Pi x)^T\beta'<0\}.
\]
Note that $\Pi x$, restricted to the span of $\beta$ and $\beta'$, has a standard Gaussian distribution on a 2-dimensional linear subspace.
Consequentially, we may take $p=2$, as claimed.

Now we assume $p=2$. Let $\tilde{x}:=x/\|x\|_2$. Then:
\[
P|x^T\beta|^m1\{\beta^Txx^T\beta'<0\}
= P\|x\|_2^m\left|\tilde{x}^T\beta\right|^m1\{\beta^T\tilde{x}\tilde{x}^T\beta'<0\}.
\]
Since $x$ is Gaussian, $\|x\|_2$ is independent of $\tilde{x}$. 
Moreover, since $x$ is Gaussian, $\|x\|_2$ follows a $\chi$-distribution, with $p=2$ degrees of freedom.
Therefore, $P\|x\|_2^m=2^{m/2}\Gamma(1+m/2)$.
Consequentially:
\[
=2^{m/2}\Gamma(1+m/2)
P\left|\tilde{x}^T\beta\right|^m1\{\beta^T\tilde{x}\tilde{x}^T\beta'<0\}.
\]
Note that $\tilde{x}$ is uniformly distributed on $S^1$.
By symmetry, $\mathbb{P}[\beta^T\tilde{x}\tilde{x}^T\beta'<0]=2\mathbb{P}[\beta^T\tilde{x}>0\cap \tilde{x}^T\beta'<0]$.
So,
\[
P\left|\tilde{x}^T\beta\right|^m1\{\beta^T\tilde{x}\tilde{x}^T\beta'<0\}
=2P|\tilde{x}^T\beta|^m1\{\beta^T\tilde{x}>0\cap \tilde{x}^T\beta'<0\}.
\]
The random variable $\tilde{x}$ follows a uniform distribution on $S^1$.
Consequentially, by rotational invariance, we may assume that $\beta=e_1$.
\[
P|\tilde{x}^T\beta|^m1\{\beta^T\tilde{x}>0\cap \tilde{x}^T\beta'<0\}
=P|\tilde{x}_1|^m\{\tilde{x}_1>0\cap \tilde{x}^T\beta'<0\}.
\]
Let $\mathcal{C}:=\{z\in S^1:z^T\beta>0\cap z^T\beta'<0\}$.
Note that the kernel of a non-zero vector in $\mathbb{R}^2$ is spanned by a 90-degree rotation $R$ of the vector.
So, the boundaries of the cap $\mathcal{C}:=\{z\in S^1:z^T\beta>0\cap z^T\beta'<0\}$ are given by $R\beta$ and $R\beta'$.
Since angles are preserved under rotation, the arc length of $\mathcal{C}$ is $\arccos((R\beta)^T(R\beta'))=\arccos(\beta^T\beta')$.
Hence:
\[
P|\tilde{x}_1|^m\{\tilde{x}_1>0\cap \tilde{x}^T\beta'<0\}=\int_0^{\arccos(\beta^T\beta')}\sin(\alpha)^m \frac{d\alpha}{2\pi}.
\]
This proves the result for $\beta,\beta'$ in general position.
The proof is complete.
\end{proof}

As a consequence of Proposition \ref{prop_trig}, we obtain the following identity.

\begin{cor}\label{cor_trig} 
Let $x\sim\mathcal{N}(0,I)$ and $\beta,\beta'\in S^{p-1}$.
Then,
\[
P|x^T\beta|1\{\beta^Txx^T\beta'<0\}
=\frac{\|\beta-\beta'\|_2^2}{\sqrt{8\pi}}.
\]
Moreover, 
\[
P|x^T\beta|^m1\{\beta^T x x^T\beta'<0\}
\leq \frac{1}{\sqrt{2}\pi}\frac{\Gamma(m/2+1)}{m+1}\left(\frac{\pi}{\sqrt{2}}\|\beta-\beta'\|_2\right)^{m+1}.
\]
\end{cor}

\begin{proof}
The first identity follows directly from Proposition \ref{prop_trig}, applied with $m=1$.
To see why, note that:
\[
\int_0^{\arccos(\beta^T\beta')}\sin(\alpha) \frac{d\alpha}{2\pi}
=\frac{1}{2\pi}\left(-\cos(\arccos(\beta,\beta'))+1\right)
=\frac{1}{2\pi}\left(-\beta^T\beta'+1\right)
\]
\[
=\frac{1}{4\pi}\|\beta-\beta'\|_2^2.
\]
Using that $\Gamma(3/2)=\sqrt{\pi}/2$ establishes the identity.
For $m>1$, we use the bound $\sin(x)\leq x$ for $x\geq 0$.
Therefore:
\[
\int_0^{\arccos(\beta^T\beta')}\sin(\alpha)^m \frac{d\alpha}{2\pi}
\leq \frac{1}{m+1}\arccos(\beta^T\beta')^{m+1}\frac{1}{2\pi}
\]
\[
\stackrel{(i)}{\leq} \frac{1}{m+1}\left(\frac{\pi}{2}\|\beta-\beta'\|_2\right)^{m+1}\frac{1}{2\pi}.
\]
In $(i)$, we upper bounded the geodesic distance on $S^{p-1}$ with the Euclidean distance as in Proposition \ref{prop_bilip}.

The bound now follows from Proposition \ref{prop_trig}.
\end{proof}

The following result gives a closed-form solution for the expectation of the unbounded term.
This identity relies on the assumption that $(x^T,\epsilon)$ is standard Gaussian.
This motivates the assumption of the probit model in conjunction with the logistic loss.

\begin{cor}\label{cor_trig_y}
Let $x\sim\mathcal{N}(0,I)$, $\epsilon\sim\mathcal{N}(0,1)$, $\beta,\beta^*\in S^{p-1}$ and $\sigma\geq 0$. Then,
\[
P|x^T\beta|1\{yx^T\beta<0\}=\frac{1}{\sqrt{2\pi}}\left(1-\frac{\beta^T\beta^*}{\sqrt{1+\sigma^2}}\right).
\]
In particular,
\[
P|x^T\beta^*|1\{yx^T\beta^*<0\}=\frac{1}{\sqrt{2\pi}}\left(1-\frac{1}{\sqrt{1+\sigma^2}}\right).
\]
Consequentially:
\[
P|x^T\beta|1\{yx^T\beta<0\}-|x^T\beta^*|1\{yx^T\beta^*<0\}
=\frac{1}{\sqrt{8\pi(1+\sigma^2)}}\|\beta-\beta^*\|_2^2.
\]
\end{cor}

\begin{proof}
Apply Proposition \ref{prop_trig} with $\tilde{x}:=(x,\epsilon)\sim\mathcal{N}(0,I_{p+1})$ and $\tilde{\beta}:=(\beta,0)$, as well as $\tilde{\beta}':=(\beta,\sigma)/\sqrt{1+\sigma^2}$.
Note that $y$ is the sign of $\tilde{x}^T\tilde{\beta}'$.
\end{proof}

\subsubsection{A projection-based argument}
In the following, we see another bound for the unbounded term.
We use it in the verification of Bernstein's condition.
Even in the proof of this result, we use Proposition \ref{prop_trig}.
Recall that $\sigma,\tau^*>0$, $\beta^*\in S^{p-1}$, $(x,\varepsilon)\sim\mathcal{N}(0,I_{p+1})$,  $\gamma^*:=\tau^*\beta^*$ and $y:=sign(x^T\beta^*+\sigma\epsilon)$.
Moreover, we define $y^*:=sign(x^T\beta^*)$, as well as for a fixed $\gamma\in\mathbb{R}^p\setminus\{0\}$, $\beta:=\gamma/\|\gamma\|_2$ and $\tau:=\|\gamma\|_2$.

\begin{lem}\label{lem_unbounded_bound4bernstein_2}
For any $\gamma\in \mathbb{R}^p\setminus\{0\}$ and $m\in\mathbb{N}$, it holds that:
\[
P|(y-y^*)x^T(\gamma-\gamma^*)|^m
\]
\[
\leq\frac{\sigma}{2\pi} \Gamma\left(\frac{m+1}{2}\right)\left(\frac{1}{\sqrt{\pi}}(\sqrt{32}\tau\|\beta-\beta^*\|_2)^m+(\sqrt{8}\pi|\tau\beta^T\beta^*-\tau^*|\sigma)^m  \right).
\]
\end{lem}

\begin{proof}
Let $\Pi:=\beta^*\beta^{*T}$ be the orthogonal projection onto the span on $\beta^*$, and let $\Pi^\perp:=I-\Pi$ be the orthogonal projection to its kernel.
We write:
\[
\gamma-\gamma^*
=\tau \beta-\tau^*\beta^*
=\tau \Pi^\perp\beta+\tau \Pi\beta-\tau^*\beta^*
=\tau \Pi^\perp\beta+(\tau \beta^T\beta^*-\tau^*)\beta^*.
\]
Let $\Phi:\mathbb{R}\rightarrow[0,1]$ be the cumulative distribution function of a standard Gaussian random variable.
Then, after conditioning on $x$, we have: 
\[
P|(y-y^*)x^T(\gamma-\gamma^*)|^m
=P2^m\Phi\left(-\frac{|x^T\beta^*|}{\sigma}\right)|x^T(\gamma-\gamma^*)|^m
\]
\begin{equation}\label{eq_unbounded_bound4bernstein_2}
\stackrel{(i)}{\leq} 2^{2m-1}P\Phi\left(-\frac{|x^T\beta^*|}{\sigma}\right)\left(
\left|\tau x^T\Pi^\perp\beta\right|^m+\left|(\tau\beta^T\beta^*-\tau^*)x^T\beta^*\right|^m
\right).
\end{equation}
The inequality $(i)$ follows from Jensen's inequality, which implies that for any positive scalars $a,b>0$ we have $(a/2+b/2)^m\leq a^m/2+b^m/2$.

We bound the two terms in the expression \eqref{eq_unbounded_bound4bernstein_2} separately.
We start with the left term.
Since $x$ is Gaussian and $\beta^*$ is orthogonal to $\Pi^\perp \beta$, it follows that $x^T\beta^*$ and $x^T\Pi^\perp \beta$ are independent.
Consequentially,
\[
P\Phi\left(-\frac{|x^T\beta^*|}{\sigma}\right)\left|\tau x^T\Pi^\perp\beta\right|^m
=P\left[\Phi\left(-\frac{|x^T\beta^*|}{\sigma}\right)\right] P\left|\tau x^T\Pi^\perp\beta\right|^m.
\]
The first term can be bounded using \eqref{eq_pi_bound}: 
\[
P\left[\Phi\left(-\frac{|x^T\beta^*|}{\sigma}\right)\right] 
=\mathbb{P}[y\neq y^*]
\leq \frac{\sigma}{\pi}.
\]
To bound the second term, note that the random variable $\beta^T\Pi^\perp x$ is Gaussian with mean zero and variance $1-(\beta^T\beta')^2\leq 2(1-\beta^T\beta')=\|\beta-\beta'\|_2^2$. 
Using that the $m$-th absolute moment of the a standard Gaussian random variable is $\sqrt{2}^m\Gamma((m+1)/2)/\sqrt{\pi}$, we find:
\[
P\left|\tau x^T\Pi^\perp\beta\right|^m
\leq \tau^m \left(\sqrt{2}\|\beta-\beta^*\|_2\right)^m\Gamma\left(\frac{m+1}{2}\right)\frac{1}{\sqrt{\pi}}.
\]

Now for the right term in expression \eqref{eq_unbounded_bound4bernstein_2}.
It carries the deterministic factor $|\tau\beta^T\beta^*-\tau^*|^m$, which we omit in the following calculation.
We define the unit vectors $\tilde{\beta}^*:=(\beta^{*T},0)^T\in S^{p}$ and $\tilde{\beta}^*:=(\beta^{*T},\sigma)^T/\sqrt{1+\sigma^2}\in S^{p}$ and a standard Gaussian $\tilde{x}\sim\mathcal{N}(0,I_{p+1})$.
Then:
\[
P\Phi\left(-\frac{|x^T\beta^*|}{\sigma}\right)\left|x^T\beta^*\right|^m
=P1\{\tilde{\beta}^T\tilde{x}\tilde{x}^T\tilde{\beta}^*<0\}\left|\tilde{x}^T\tilde{\beta}\right|^m.
\]
We bound this expression with Corollary \ref{cor_trig} and use that $1-\frac{1}{\sqrt{1+\sigma^2}}\leq \frac{\sigma^2}{2}$, which gives $\|\tilde{\beta}-\tilde{\beta}^*\|_2\leq \sigma$.
To make the expression comparable with the first term, we use that $\Gamma(m/2+1)/(m+1)\leq \Gamma((m+1)/2)$.
We conclude:
\[
P1\{\tilde{\beta}^T\tilde{x}\tilde{x}^T\tilde{\beta}^*<0\}\left|\tilde{x}^T\tilde{\beta}\right|^m\leq \frac{1}{\sqrt{2}\pi}\Gamma\left(\frac{m+1}{2}\right)\left(\frac{\pi}{\sqrt{2}}\sigma\right)^{m+1}.
\]
Combining these bounds completes the proof.
\end{proof}


\end{document}